\documentclass[a4paper,12pt,oneside,reqno]{amsart}

\usepackage{hyperref}
\usepackage[headinclude,DIV13]{typearea}
\areaset{15.1cm}{25.0cm}
\usepackage{amsfonts,amssymb,amsmath,amsthm,bbm}
\usepackage[latin1] {inputenc}
\usepackage{graphicx, psfrag}

\usepackage{bold-extra}

\newtheorem{theorem}{Theorem}[section]
\newtheorem{lemma}[theorem]{Lemma}
\newtheorem{proposition}[theorem]{Proposition}
\newtheorem{corollary}[theorem]{Corollary}

\theoremstyle{definition}

\theoremstyle{remark}
\newtheorem*{remark}{Remark}

\def\paragraph#1{\noindent \textbf{#1}}

\numberwithin{equation}{section}

\def\dist{\mathop{\rm dist}\nolimits}

\def\d{\mathrm{d}}

\def\rf{\rfloor}

\def\lf{\lfloor}

\def\<{\langle}
\def\>{\rangle}
\def\a{\alpha}
\def\b{\beta}
\def\e{\epsilon}

\def\g{\gamma}

\def\l{\lambda}

\def\s{\sigma}
\def\t{\tau}
\def\th{\theta}

\def\o{\omega}

\def\G{\Gamma}
\def\O{\Omega}

\def\del{\partial}
\def\R{{\Bbb R}}  
\def\N{{\Bbb N}}  
\def\P{{\Bbb P}}  
\def\Z{{\Bbb Z}}  
\def\E{{\Bbb E}}  

\def\Va{{\Bbb V}}  

\let\cal=\mathcal

\def\BB{{\cal B}}
\def\CC{{\cal C}}
\def\DD{{\cal D}}
\def\EE{{\cal E}}
\def\FF{{\cal F}}

\def\HH{{\cal H}}
\def\II{{\cal I}}
\def\JJ{{\cal J}}
\def\KK{{\cal K}}
\def\LL{{\cal L}}
\def\MM{{\cal M}}

\def\OO{{\cal O}}
\def\PP{{\cal P}}

\def\RR{{\cal R}}
\def\SS{{\cal S}}

\def\VV{{\cal V}}

\def\VV{{\cal V}}
\def\WW{{\cal W}}

 \def \G {{\Gamma}}

 \def \b {{\beta}}
\def \e {{\epsilon}}
 \def \s {{\sigma}}

 \def \t {{\tau}}
 \def \th {{\theta}}
 
 \def \g {{\gamma}}
 \def \l {{\lambda}}
 \def \d {{\delta}}
 \def \a {{\alpha}}
 \def \o {{\omega}}
 \def \O {{\Omega}}
 \def \th {{\theta}}

 \def \del {{\partial}}

 \def \ba {\begin{array}}
 \def \ea {\end{array}}

 \newcommand{\be}{\begin{equation}}
 \newcommand{\ee}{\end{equation}}

\newcommand{\bea}{\begin{eqnarray}}
 \newcommand{\eea}{\end{eqnarray}}
\def\TH(#1){\label{#1}}\def\thv(#1){\ref{#1}}
\def\Eq(#1){\label{#1}}\def\eqv(#1){(\ref{#1})}

\def\sfrac#1#2{{\textstyle{#1\over #2}}}

 \def \1{\mathbbm{1}}
\def\wt {\widetilde}
\def\wh{\widehat}

\begin{document}
\title{Aging in Metropolis dynamics of the REM: a proof}
\author[V. Gayrard]{V\'eronique Gayrard}
 \address{
V. Gayrard\\Aix Marseille Universit\'e, CNRS,
Centrale Marseille, I2M UMR 7373, 13453, Marseille, France 
}
\email{veronique.gayrard@math.cnrs.fr}

\subjclass[2000]{82C44,60K35,60G70} \keywords{random dynamics,
random environments, clock process, L\'evy processes,
spin glasses, aging, Metropolis dynamics
}
\date{\today}

\begin{abstract} 
We study the aging behavior of the Random Energy Model (REM) evolving under Metropolis dynamics. 
We prove that a classical two-time correlation function converges almost surely to the arcsine law distribution 
function that characterizes activated aging, as predicted in the physics literature,  in the optimal domain of 
the time-scale and temperature parameters where this result can be expected to hold. In the course of the 
proof we establish that a certain continuous time clock process,  after proper rescaling, converges almost surely 
to a stable subordinator, improving upon the result of Ref.~\cite{CW15} where a closely related clock is shown 
to converge in probability only, and in a restricted region of the time-scale and temperature parameters.
The random rescaling involved in this convergence is controlled at the fine level of fluctuations. 
As a byproduct, we refine and prove a conjecture made in Ref.~\cite{CW15}.
\end{abstract}

 \maketitle


\section{Introduction}
    \TH(1)

While there is as yet no established theory for the  description of glasses, a consensus exists 
that  this amorphous state of matter is intrinsically dynamical in nature \cite{MR2866716}, \cite{Ku01}, \cite{Gold69}.
Measuring suitable two-time correlation functions indeed reveals that glassy dynamics 
are history dependent and dominated by ever slower transients:  they are \emph{aging}.
The realization in the late 80's that \emph{mean-field} spin glass dynamics could provide a 
mathematical formulation for this phenomenon sparked renewed interest in models, 
such as Derrida's REM and $p$-spin SK models  \cite{De1}, \cite{D85}, whose statics had, until then, 
been the main focus of attention \cite{BCKM98}. Despite this, Bouchaud's phenomenological 
\emph{trap models} first took the center stage as they succeeded in predicting the power-law decay 
of two-time correlation functions observed experimentally,  even though they did so at the cost of an 
ad hoc construction and  drastically simplifying assumptions \cite{Bou92}, \cite{BD95}.

It was not until 2003  that a trap model dynamics  was shown to result for the microscopic 
Glauber dynamics of a (random) mean-field spin glass Hamiltonian, namely, the REM endowed 
with the so-called \emph{Random Hopping} dynamics and observed on time-scales near equilibrium 
\cite{BBG02,BBG03a, BBG03b}. Quite remarkably, the predicted functional form of two-time correlation functions 
was recovered. Rapid progress followed over the ensuing decade, beginning with  \cite{BC06b}. 
The optimal domain of temperature and time-scales were this prediction applies was obtained in 
Ref.~\cite{G10b} (almost surely in the random environment except for times scales near equilibrium 
where the results hold in probability only) and these  results were partially extended to the 
$p$-spin SK models \cite{BBC08},  \cite{BG13}.
%

The choice of the Random Hopping dynamics, however, clearly favored the emergence of
trap models. Just as in trap model constructions, its trajectories are those of a simple random walk on 
the underlying graph, and thus, do not depend on the random Hamiltonian. This is in sharp contrast 
with \emph{Metropolis} \cite{Metro}  dynamics, a choice heralded in the physic's literature as \emph{the} 
natural microscopic Glauber dynamics \cite{JKu}, whose trajectories are biased against increasing the energy.
This dependence on the random Hamiltonian makes the analysis of the two-time 
correlation functions much harder. This problem was first tackled in  \cite{G15}  were a
truncated REM is considered, and a natural two-time correlation function is proved to behave as in the 
Random Hopping dynamics, in the same, optimal range of time-scales and temperatures for which 
this result holds almost surely in the random environment. In the present paper, we free ourselves of 
the simplifying truncation assumption and prove that the same result holds true almost surely for the full REM.
A recent paper  \cite{CW15}, by establishing the convergence of a so-called clock process,
suggested that this might be the case but failed short of proving aging:   the sole clock convergence, indeed, 
does not suffice to deduce aging, a property of correlation functions.

\subsection{Main result} 
    \TH(S1.1)

Let us now specify the model. Denote by $\VV_n=\{-1,1\}^n$  the n-dimensional discrete cube 
and by $\EE_n$ its edge set. The Hamiltonian (or energy) of the REM is a collection of independent 
Gaussian random variables, $(\HH_n(x), x\in\VV_n)$, satisfying
\be
\E \HH_n(x)=0, \quad \E \HH^2_n(x)=n.
\Eq(1.1.0)
\ee
The sequence $(\HH_n(x), x\in\VV_n)$, $n>1$, is defined on a common probability space denoted by $(\O, \FF, \P)$.
On $\VV_n$, we consider the Markov jump process $(X_{n}(t), t>0)$ with rates
\be
\l_n(x,y)=\frac{1}{n}e^{-\b\left[\HH_n(y)-\HH_n(x)\right]^+},\quad\text{if $(x,y)\in\EE_n$},
\Eq(1.1.4)
\ee
and $\l_n(x,y)=0$ else, were $a+=\max\{a,0\}$.
This defines the single spin-flip continuous time Metropolis dynamics of the REM at 
temperature $\beta^{-1}>0$. 
Note that the rates are reversible with respect to the measure that assigns to $x\in\VV_n$  the mass
\be
\t_n(x)\equiv \exp\{-\b \HH_n(x)\}.
\Eq(1.1.3)
\ee
When studying aging the choice of the observation time-scale, $c_n$, is all-important. 
Given $0<\varepsilon< 1$ and  $0<\beta<\infty$, we let $c_n\equiv c_n(\b,\varepsilon)$ 
be the two-parameter sequence defined by
\be
2^{\varepsilon n}\P(\t_n(x)\geq c_n)=1.
\Eq(1.theo0.0)
\ee
Gaussian tails estimates yield the explicit  form
\be
c_n=\exp\left\{n\b\b_c(\varepsilon)-(1/2\a(\varepsilon))\left(\log(\b^2_c(\varepsilon)n/2)+\log 4\pi+o(1)\right)\right\}
\Eq(1.theo0.0')
\ee
where
\bea
\Eq(1.2.3)
&\b_c(\varepsilon)=\sqrt{\varepsilon2\log 2},
\\
\Eq(1.2.4)
&\a(\varepsilon)=\b_c(\varepsilon)/\b.
\eea
A classical  choice of two-time correlation function  is the probability $\CC_n(t,s)$ to find the
process in the same state at the two endpoints of the time interval $[c_n t, c_n (t+s)]$, 
\be
\CC_n(t,s)\equiv\PP_{\mu_n}\left(X_n(c_n t)=X_n(c_n (t+s))\right),\quad t,s>0.
\Eq(1.1.2)
\ee
Here $\PP_{\mu_n}$ denotes the law of $X_n$ conditional on  $\FF$  (i.e. for fixed realizations of 
the random Hamiltonian) when the initial distribution, $\mu_n$, is  the uniform measure on $\VV_n$.

\begin{theorem}
     \TH(1.theo0)
 
For all $0<\varepsilon<1$ and all $\b>\b_c(\varepsilon)$, for all $t>0$ and $s>0$, $\P$-almost surely,
\be
\lim_{n\to \infty }\PP_{\mu_n}\left(X_n(c_n t)=X_n(c_n (t+s))\right)
= \frac{\sin\a(\varepsilon) \pi }{\pi }\int_0^{t/(t+s )} u^{\a(\varepsilon) -1}(1-u)^{-\a(\varepsilon) }\,d u.
\Eq(1.theo0.1)
\ee
\end{theorem}

\begin{remark}
We in fact prove the  more general statement that \eqv(1.theo0.1) holds along any
$n$-dependent sequences of the form
$
0<\varepsilon_n\leq 1-c'\b \sqrt{{n}^{-1}{\log n}}+ c''{n}^{-1}{\log n}
$
where $0<c',c''<\infty$ are constants, that satisfy 
$\lim_{n\rightarrow\infty}\varepsilon_n=\varepsilon$, $0<\varepsilon\leq 1$.
Relaxation to stationarity is known to occur, to leading order, on time-scales $c_n$ 
of the form \eqv(1.theo0.0') with 
$
\varepsilon_n=1
$
\cite{FIKP}.  
At the other extremity, a behavior known as extremal aging is expected to characterize 
the process on times scales that are sub-exponential in the volume and defined through sequences $\varepsilon_n$
that decay to 0 slowly enough \cite{BGS13}, \cite{BGu12}. This will be the object of a follow up paper.
\end{remark}

As in virtually all papers on aging, the proof of Theorem \thv(1.theo0) relies on a two-step scheme 
that seeks to isolate  the causes of aging by writing the process of interest, $X_n$, as  an 
\emph{exploration process} time-changed by (the inverse of) a \emph{clock process}.
Aging  is then linked to the arcsine law for stable subordinators through the convergence of the 
suitably rescaled clock process to an  $\a$-stable subordinator, $0<\a<1$.
This is provided that the  two-time correlation function at hand can be brought into a suitable function of the clock.
Both steps heavily depend on the properties of the exploration process.

While this scheme offers the methodological underpinnings of the analysis of aging, two distinct ways of 
implementing it, through \emph{discrete} or \emph{continuous}  time objects, respectively, have emerged from the  
literature (we refer to the recent papers \cite{G15}, \cite{GS15}, and \cite{CW15} for in-depth bibliographies). The first 
arose from the study of models whose exploration process can be chosen as the simple random walk on the underlying 
graph. As mentioned earlier, this includes all Random Hopping dynamics  and several trap models (e.g.~on the complete 
graph or on $\Z^d$). In physically more realistic dynamics the discrete scheme may quickly become intractable.  
As shown in Ref.~\cite{G15} for Metropolis dynamics of a truncated REM, the associated exploration process is itself 
an aging process that presents the same complexity as the original dynamics. A similar situation arises when considering
asymmetric trap models on $\Z^d$. Initiated in that context, the continuous scheme consists in choosing a 
(now continuous time) exploration process that mimics the simple random walk.

Prescribing the exploration process completely determines the clock process. 
Clearly, having  efficient  tools available to prove their convergence to stable subordinators is essential.
Such tools were provided in Ref.~\cite{G12} and \cite{BG13} for discrete time clock processes in the 
general setting of  reversible Markov jumps processes in random environment on sequences of finite graphs  
and, more recently, for both discrete and continuous time clock processes of similar Markov jumps processes   
on infinite graphs \cite{GS15}. These tools both allowed one to improve all earlier results on the 
Random Hopping dynamics of mean-field models  \cite{G10b}, \cite{BG13}, \cite{BGS13}, 
turning statements previously obtained in law into almost sure statements in the random environment,
and to obtain the first aging results for several two-time correlation functions of asymmetric trap model 
on $\Z^d$ \cite{GS15}.

In Section \thv(S1.3) below we fill the gap left by continuous time clock processes in the case of sequences 
of finite graphs and, thus, extent the results of  Ref.~\cite{BG13} to that setting. This is perhaps no more than 
an exercise but the results we present (Theorem \thv(1.theoA) and Theorem \thv(1.theoB)) are the cornerstone 
of our approach and, hopefully,  of other papers to come. We close this introduction out in Section \thv(S1.4) 
by stating a clock process convergence result for Metropolis dynamics of the REM (Theorem \thv(1.theo1)) 
that is at the heart of the proof of Theorem \thv(1.theo0).

\subsection{Convergence of continuous time clock processes}
   \TH(S1.3)

We now enlarge our focus to the following abstract setting.
Let  $G_n(\VV_n, \EE_n)$ be a sequence of loop-free graphs with set of vertices $\VV_n$ and set of edges $\EE_n$. 
A \emph{random environment} is a family of possibly dependent positive random variables, $(\t_n(x), x\in \VV_n)$. 
The sequence $(\t_n(x), x\in \VV_n)$, $n>1$, is defined on a common probability space denoted by $(\O,\FF, \P)$.
On $\VV_n$ we consider a Markov jump process, $(X_{n}(t), t>0)$, with initial distribution $\mu_n$ and 
jump rates $(\l_n(x,y))_{x,y\in \VV_n}$ satisfying $ \l_n(y,x)=0$ if $(x,y)\notin \EE_n$ and 
\be
\t_n(x) \l_n(x,y)=\t_n(y) \l_n(y,x) \quad\text{if $(x,y)\in\EE_n$, $\ x\neq y$.}
\Eq(1.03.1)
\ee
Thus $X_n$ is reversible with respect to the (random measure) 
that assigns to $x\in\VV_n$  the mass $\t_n(x)$. To $X_n$ we associate an \emph{exploration process} $Y_n$. 
This is any Markov jump process, $(Y_n(t), t>0)$, with state space $\VV_n$,
initial distribution $\mu_n$, and  jump rates $(\wt\l_n(x,y))_{x,y\in \VV_n}$ chosen such that 
$X_n$ and $Y_n$ have the same trajectories, that is to say,
\be
\frac{\l_n(x,y)}{\l_n(x)}=\frac{\wt\l_n(x,y)}{\wt\l_n(x)}\quad\forall (x,y)\in \EE_n,
\Eq(1.03.6)
\ee
where $\wt\l_n^{-1}(x)$ and $\l_n^{-1}(x)$ are, respectively, the mean holding times at $x$ of $Y_n$ and $X_n$:
\bea
&&
\wt\l_n(x)\equiv \sum_{y:(x,y)\in \EE_n}\wt\l_n(x,y),
\Eq(1.03.5)
\\
&&
\l_n(x)\equiv \sum_{y:(x,y)\in \EE_n}\l_n(x,y).
\Eq(1.03.2)
\eea
Then $X_n$ and $Y_n$ are related to each other through the time change
\be
X_n(t)=Y_n(\wt S_n^{\leftarrow}(t)),\quad \quad t\geq0,
\Eq(1.03.3)
\ee
where $\wt S_n^{\leftarrow}$ denotes the generalized right continuous inverse of $\wt S_n$, and $\wt S_n$, the so-called \emph{continuous time clock process}, is given by
\be
\wt S_n(t)=\int_{0}^{t}\l_n^{-1}(Y_n(s))\wt\l_n(Y_n(s))ds,\quad \quad t\geq 0.
\Eq(1.03.4)
\ee

Note that there is considerable freedom in the choice of the exploration process $Y_n$. 
We will come back to this issue at the end of this subsection and focus, for the time being, on the  
analysis of the  asymptotic behavior of the general clock process \eqv(1.03.4).

For future reference, we denote by $\FF^Y$ the $\s$-algebra generated by the
processes $Y_n$. We  write $P$ for the law of the process $Y_n$ conditional on
the $\s$-algebra $\FF$, i.e. for fixed realizations of the random environment.
Likewise we  call $\PP$ the law of $X_n$ conditional on $\FF$.
If the initial distribution, $\mu_n$, has to be specified we write $\PP_{\mu_n}$
and $P_{\mu_n}$. Expectation with respect to $\P$, $P_{\mu_n}$, and $\PP_{\mu_n}$ 
are denoted  by $\E$, $E_{\mu_n}$, and $\EE_{\mu_n}$, respectively.

Our main aim is to obtain simple and robust criteria for the convergence of the 
(suitably rescaled) clock process \eqv(1.03.4) to a stable subordinator. Since  
the clock is a doubly stochastic process,  the desired convergence mode must  be specified.
We will ask whether there exist sequences $a_n$ and $c_n$ that make the rescaled clock process
\be
S_n(t)=c_n^{-1}\wt S_n(a_n t)\,,\quad t\geq 0,
\Eq(1.03.7)
\ee
converge weakly, as $n\uparrow\infty$, as a sequence of random elements in 
Skorokhod's space $D((0,\infty])$, and strive to obtain $\P$-almost sure results in 
the random environment since such results (also referred to as \emph{quenched}) 
contain the most useful information from the point of view of physics.

As for discrete time clock processes \cite{G12}, \cite{BG13}, the driving force behind our 
approach is a powerful method developed by Durrett and  Resnick  \cite{DR78} to prove 
functional limit theorems for sums of dependent variables. Clearly this method does not 
cover the case of our continuous time clock processes. The simple idea (already present 
in  \cite{GS15} ) is to introduce a suitable ``blocking'' that turns the rescaled clock process 
\eqv(1.03.7) into a partial sum process to which  Durrett and Resnick method can now be 
applied. For this we introduce a new scale, $\th_n$, and set
\be
k_n(t)\equiv\lf a_n t/\theta_n\rf.
\Eq(1.03.8)
\ee
The  \emph{blocked clock process}, $S^b_n(t)$,  is defined through
\be
S^b_n(t)=\sum_{i=1}^{k_n(t)}Z_{n,i}
\Eq(1.03.9)
\ee
where, for each $i\geq 1$,
\be
Z_{n,i} \equiv c_n^{-1}\sum_{x\in\VV_n}\bigl(\l_n^{-1}(x)\wt\l_n(x)\bigr)[\ell_n^x(\theta_ni)-\ell_n^x(\theta_n(i-1))],
\Eq(1.03.10)
\ee
and where, for each $x\in\VV_n$, 
\be
\ell_n^x(t)=\int_{0}^{t}\1_{\{Y_n(s)=x\}}ds
\Eq(1.03.11)
\ee
is the local time at $x$. 
The next theorem gives sufficient conditions for $S^b_n$ to converge. 
These conditions are expressed in terms of a small number of key quantities.
For each $t>0$, let
\be
\pi_n^{Y,t}(y)=
k_n^{-1}(t)\sum_{i=1}^{k_n(t)-1}\1_{\{Y_n(i\theta)=y\}}
\Eq(1.3.5)
\ee
be the empirical measure on $\VV_n$ constructed from the sequence 
$
(Y_n(i\theta), i\in\N)
$.
For $y\in\VV_n$ and $u>0$, denote by
\be
Q^{u}_n(y)\equiv P_{y}(Z_{n,1}>u)
\Eq(1.3.6)
\ee
the tail distribution of the aggregated jumps when $X_n$ (equivalently, $Y_n$) starts in $y$.
Using  these quantities, define the functions
\bea
\nu_n^{Y,t}(u,\infty)
&\equiv&k_n(t) \sum_{y\in\VV_n}\pi_n^{Y,t}(y)Q^{u}_n(y),
\Eq(1.3.7)
\\
\s_n^{Y,t}(u,\infty)
&\equiv&
k_n(t) \sum_{y\in\VV_n}\pi_n^{Y,t}(y)\left[Q^{u}_n(y)\right]^2.
\Eq(1.3.8)
\eea
Observe that the sequence of measures $\pi_n^{Y,t}$ as well as the sequence of functions 
$Q^{u}_n(y)$, $y\in\VV_n$,  are random variables on the probability space $(\O, \FF, \P)$ 
of the random environment. Thus, the functions $\nu_n^{Y,t}$ and $\s_n^{Y,t}$ also are 
random variables on that space.

We now formulate four conditions for the sequence $S^b_n$ to converge to a subordinator. 
These conditions refer to a given sequence of initial distributions  $\mu_n$, given sequences 
of numbers $a_n, c_n$, and $\th_n$ as well as a given realization of the random environment. 
\smallskip

\noindent{\bf Condition (A0).} For all $u>0$,
\be
\lim_{n\rightarrow\infty}P_{\mu_n}(Z_{n,1}>u)=0.
\Eq(A0.1)
\ee

\noindent{\bf Condition (A1).}
There exists a $\s$-finite measure $\nu$ on $(0,\infty)$ satisfying $\int_0^\infty (x\wedge 1)\nu(dx)<\infty$
and such that for all continuity points $x$ of the distribution function of
$\nu$, for all $t>0$ and all $u>0$,
\be
P_{\mu_n}^{}\left(
\left|
\nu_n^{Y,t}(u,\infty)-t\nu(u,\infty)
\right|
<\e
\right)=1-o(1)\,,\quad\forall\e>0\,.
\Eq(A1.1)
\ee

\noindent{\bf Condition (A2).}  For all $u>0$ and all $t>0$,
\be
P_{\mu_n}^{}\left(\s_n^{Y,t}(u,\infty)<\e\right)=1-o(1)\,,\quad\forall\e>0\,.
\Eq(A2.1)
\ee

\noindent {\bf Condition (A3).}  For all $t>0$,
\be
\lim_{\e\downarrow 0}\limsup_{n\uparrow \infty}
k_n(t) \sum_{y\in\VV_n}E_{\mu_n}(\pi_n^{Y,t}(y))E_{y}(Z_{n,1}\1_{\{Z_{n,1}\leq \e\}})=0.
\Eq(A3.1)
\ee

\begin{theorem}
\TH(1.theoA)
 For all sequences of initial distributions  $\mu_n$
and all sequences $a_n$, $c_n$, and $1\leq \theta_n\ll a_n$
for which Conditions (A0), (A1), (A2), and (A3) are verified,
either $\P$-almost surely or in $\P$-probability,
the following holds w.r.t. the same convergence mode:
\be
 S^b_n\Rightarrow_{J_1}  S_\nu,
\Eq(1.theoA.1)
\ee
where $S_\nu$ is the L\'evy subordinator with L\'evy measure
$\nu$ and zero drift. Convergence holds weakly on the space
$D([0,\infty))$ equipped with the Skorokhod $J_1$-topology.
\end{theorem}

\begin{remark}
Note that the theorem is stated for the \emph{blocked} process $S^b_n$ rather than the original 
process $S_n$ (defined in \eqv(1.03.7)). This may falsely appear as an undesirable consequence 
of our techniques. We stress that for applications to correlation functions, one needs statements that 
are valid in the strong $J_1$ topology whereas forming blocks is needed in order to make sense of 
writing $J_1$ convergence statements in the setting of continuous time clocks.
\end{remark}
\begin{remark}Also note that convergence of $S^b_n$ in the strong $J_1$ topology immediately implies 
the strictly weaker result that  $S_n$ converges to the same limit in the 
$M_1$ topology.
\end{remark}
 
As for discrete time clocks of Ref.~\cite{BG13}, our next step consists in reducing
Conditions (A1) and (A2) of Theorem \thv(1.theoA) to (i) a\emph{ mixing condition} for the
chain $Y_n$, and (ii) a \emph{law of large numbers} for the random variables
$Q_n$. 
Again we formulate three conditions for 
a given sequence of initial distributions  $\mu_n$,
given sequences $a_n, c_n$, and $\th_n$, and a
given  realization of the random environment. 

\smallskip

\noindent{\bf Condition (B0).} 
Denote by $\pi_n$  the invariant measure of $Y_n$.
There exists a sequence $\kappa_n\in\N$ and a positive decreasing sequence $\rho_n$, satisfying
$\rho_n\downarrow 0$ as $n\uparrow\infty$, such that, for all pairs $x,y\in\VV_n$, and all $t\geq 0$,
\be
|P_{x}\left(Y_n(t+\kappa_n)=y\right)-\pi_n(y)|\leq \rho_n\pi_n(y).
\Eq(B0.1)
\ee

\noindent{\bf Condition (B1).}
There exists a measure $\nu$ as in Condition (A1) such that,  for all $t>0$ and all $u>0$,
\be
\nu_n^{t}(u,\infty)\equiv k_n(t)\sum_{y\in\VV_n}\pi_n(y)Q^{u}_n(y)\rightarrow t\nu(u,\infty),
\Eq(B1.1)
\ee

\noindent{\bf Condition (B2).}  For all $t>0$ and all $u>0$,
\be
\s_n^{t}(u,\infty)\equiv k_n(t) \sum_{y\in\VV_n}\pi_n(y)\left[Q^{u}_n(y)\right]^2\rightarrow 0.
\Eq(B2.1)
\ee

\noindent {\bf Condition (B3).}
 For all $t>0$,
 \be
\lim_{\e\downarrow 0}\limsup_{n\uparrow \infty}
k_n(t) \sum_{y\in\VV_n}\pi_n(y)E_{y}(Z_{n,1}\1_{\{Z_{n,1}\leq \e\}})=0.
\Eq(B3.1)
\ee

\begin{theorem}
     \TH(1.theoB)
      Assume that for all sequences of initial distributions  $\mu_n$ and all sequences $a_n$, $c_n$, 
$\kappa_n$, and $\kappa_n\leq \theta_n\ll a_n$, Conditions (A0), (B0), (B1), (B2), and (B3) 
hold $\P$-almost surely, respectively in $\P$-probability. Then, as in \eqv(1.theoA.1), 
$
 S^b_n\Rightarrow_{J_1}S_\nu,
$
$\P$-almost surely, respectively in $\P$-probability.
\end{theorem}

Theorem \thv(1.theoB) is our key tool for proving convergence of blocked clock processes 
to subordinators. It is of course essential for the success of our strategy that the convergence 
criteria we obtained be tractable. Going back to \eqv(1.03.6) we thus now ask, in this light, 
how best to choose  the exploration process $Y_n$.

A tentative answer to this question is to mimic the exploration process of the Random Hopping dynamics, 
which means   choose $Y_n$ such that  its invariant measure, $\pi_n$,  is ``close'' to the uniform measure
and its mixing time, $\kappa_n$, is short compared to that of the process $X_n$. The following class of 
jump rates, inspired from an ingenious choice made in Ref.~\cite{CW15}, is intended to favor the 
emergence of these properties. Given a fresh  sequence $\eta_n\geq 0$, set
\be
\begin{split}
\wt\l_n(x,y)
 & =\max(\eta_n,\t_n(x))\l_n(x,y).
\end{split}
\Eq(1.03.12)
\ee
One easily checks that \eqv(1.03.6) is verified, that $Y_n$  is reversible with respect to the measure
\be
\pi_n(x)=
\frac{\min\bigl(\eta_n,\t_n(x)\bigr)}{\sum_{x\in \VV_n}\min\bigl(\eta_n,\t_n(x)\bigr)}\1_{\{\eta_n>0\}}
+ |\VV_n|^{-1}\1_{\{\eta_n=0\}},\quad x\in \VV_n,
\Eq(1.4.3)
\ee
and that the clock \eqv(1.03.4)  becomes
\be
\wt S_n(t)=\int_{0}^{t}\max\bigl(\eta_n,\t_n(Y_n(s))\bigr)ds.
\Eq(1.4.4)
\ee
Let us discuss the role of $\eta_n$ on the example of Metropolis dynamics of REM.
When $\eta_n=0$, $\pi_n$ nicely reduces to the uniform measure but the mixing time, 
$\kappa_n$, of the resulting exploration process turns out to be of the same order as that 
of $X_n$, that is to say, of the order of
$
\max_{(x,y)}(\min{(\t_n(x),\t_n(y))})^{-1}=e^{\b n\sqrt{\log 2}(1+o(1))}
$.
This leaves little hope that the conditions of Theorem \thv(1.theoB) can be verified.
A moment's thought suffices, however, to see that such a large mixing time is a side effect
of the symmetry of the Hamiltonian \eqv(1.1.0). By breaking this symmetry, the term 
$\max(\eta_n,\t_n(x))$ in \eqv(1.03.12) places an $\eta_n$-dependent cap on $\kappa_n$  
(see Section \thv(S3.1)). One is then left to choose $\eta_n$ small enough so that $\pi_n$ 
remains close to the uniform measure but large enough so that $\kappa_n$ is kept
as small as needed. A similar strategy should hopefully apply to more general mean-field 
spin glass Hamiltonians.

\begin{remark}
We stress that the sole convergence of the clock process does not suffice to deduce 
aging, namely, the specific power law decay of the two-time correlation function. 
One still has to solve the problem of reducing the behavior of the two-time correlation function, 
as $n\rightarrow\infty$, to the arcsine law for stable subordinators, and this requires more 
information on the exploration process than needed to only prove  convergence of the clock. 
Notice also that  unlike the discrete time clock process, the continuous time clock process is not a physical time. 
It thus has no physical meaning on its own.
\end{remark}

\subsection{Application to Metropolis dynamics of the REM}
   \TH(S1.4)

From that point onwards we focus on Metropolis dynamics of the REM (see \eqv(1.1.0)-\eqv(1.1.4)) 
started in the uniform measure on $\VV_n$. Applying the abstract results of Section \thv(S1.3) enables 
us to prove  $\P$-almost sure convergence of the blocked clock process $S^b_n(t)$, defined in  
\eqv(1.03.9), when the continuous time clock process $\wt S_n(t)$, given by \eqv(1.03.4), 
is chosen as in \eqv(1.4.4).

To sate this result we must specify several quantities: the parameter $\eta_n$,  the time-scales, 
$a_n$ and $c_n$, and the block length, $\theta_n$, entering the definitions of  $\wt S_n(t)$ and $S^b_n(t)$. 
We begin by defining a sequence, $r^{\star}_n$, that is ubiquitous throughout the rest of the paper: 
given $\beta>0$ and a constant $c_{\star}>1+\log 4$, we let $r^{\star}_n\equiv r_n(\b,c_{\star})$
be the solution of 
\be
n^{c_{\star}}\P(\t_n(x)\geq r^{\star}_n)=1.
\Eq(1.4.1)
\ee
In explicit form
\be
r^{\star}_n
=\exp\left\{\b
\sqrt{2c_{\star}n\log n}
\left(1-\sfrac{\log\log n}{8 c_{\star}\log n}(1+o(1))\right)
\right\}.
\Eq(2.lem2.2)
\ee
We now take
$
\eta_n\equiv\left(r^{\star}_n\right)^{-1}
$
in \eqv(1.03.12) which, combined with \eqv(1.1.4), yields
\be
\wt\l_n(x,y)=\frac{1}{n r^{\star}_n}\frac{\min(\t_n(y), \t_n(x))}{\min\bigl(\frac{1}{r^{\star}_n},\t_n(x)\bigr)},\quad
\text{if $(x,y)\in\EE_n$},
\Eq(1.4.2)
\ee
and $\wt\l_n(x,y)=0$ else. The physical observation time-scale, $c_n$,  is chosen as in \eqv(1.theo0.0).
It is naturally the same as in the Random Hopping dynamics. On the contrary, the definition of the auxiliary 
time-scale, $a_n$, contrasts sharply with the simple choice $a_n=2^{\varepsilon n}$ made in the 
Random Hopping dynamics.  We here must take
\be
a_n=2^{\varepsilon n}/b_n
\Eq(1.theo1.0)
\ee
where the sequence $b_n$ is defined as follows. Recalling \eqv(1.2.3) and \eqv(1.2.4), define
\be
F_{\b,\varepsilon,n}(x)\equiv
x^{\a_{n}(\varepsilon)-\frac{\log x}{2n\b^2}}
\bigl(1-\sfrac{\log x}{n\b\b_c(\varepsilon)}\bigr)^{-1}, \quad x>0,
\Eq(4.prop2.16)
\ee
where $\a_{n}(\varepsilon)\equiv(n\beta^2)^{-1}\log c_n$, that is, in view of \eqv(1.theo0.0'),
\be
\a_{n}(\varepsilon)=
\a(\varepsilon)
-\sfrac{\log(\b^2_c(\varepsilon)n/2)+\log 4\pi+o(1)}{2n\b\b_c(\varepsilon)}.
\Eq(4.prop2.17)
\ee
Further introduce the random set 
\be
T_n\equiv\left\{x\in \VV_n\mid\t_n(x)\geq  c_n  (n^2 \theta_n)^{-1}\right\}.
\Eq(1.theo1.4)
\ee
Then, for $\ell_n^x$  as in \eqv(1.03.11), we set
\be
b_n\equiv (\theta_n\pi_n(T_n))^{-1}\sum_{x\in T_n}E_{\pi_n}\left[F_{\b,\varepsilon,n,}(\ell_n^x(\theta_n))\right].
\Eq(1.theo1.3)
\ee

It now only remains to choose the block length  $\theta_n$. (The notation $x_n\ll y_n$ means that the 
sequences $x_n>0$ and $y_n>0$ satisfy $x_n/y_n\rightarrow 0$ as $n\rightarrow\infty$.)

\begin{theorem}
    \TH(1.theo1)
    
Given $0<\varepsilon<1$ let $\theta_n$ be any sequence such that  
\be
\sfrac{4}{1-\a(\varepsilon)}\log r^{\star}_n<\log \theta_n\ll  n
\Eq(1.theo1.0')
\ee
and let $c_n$ and $a_n$ be as in \eqv(1.theo0.0) and \eqv(1.theo1.0)-\eqv(1.theo1.3), respectively.
Then, for all $0<\varepsilon<1$ and all $\b>\b_c(\varepsilon)$, $\P$-almost surely,
\be
 S^b_n\Rightarrow_{J_1}  V_{\a(\varepsilon)}
 \Eq(1.theo1.1)
\ee
where $ V_{\a(\varepsilon)}$ is  a stable subordinator with zero drift and L\'evy measure $\nu$ defined through
\be
\nu(u,\infty)= u^{-\a(\varepsilon)} ,\quad u>0,
\Eq(1.theo1.2)
\ee 
and where $\Rightarrow_{J_1}$ denotes weak convergence in the space
$D([0,\infty))$ of c\`adl\`ag functions equipped with the Skorokhod $J_1$-topology.  
\end{theorem}

We  again emphasize (see the remark below  Theorem \eqv(1.theoA)) that the $J_1$ 
convergence statement of Theorem \thv(1.theo1) is crucial to the control correlation functions.
Of course, Theorem \thv(1.theo1) implies the weaker result that  the original (non blocked) 
clock process \eqv(1.03.7) converges  to the same limit in the $M_1$ topology of Skorokhod.
Such a result was proved in Ref.~\cite{CW15} (for the clock obtained by taking $\eta_n=1$  
in \eqv(1.4.4)) albeit  only in $\P$-probability and in the restricted domain of parameters 
$\b>\b_c(\varepsilon)$  and $1/2<\varepsilon<1$. As shown in  \cite{G15} (see lemma 2.1) 
the graph structure of the set $T_n$  when $1/2<\varepsilon<1$ reduces to a collection of 
isolated vertices (no element of $T_n$ has a neighbor in $T_n$) and this considerably 
simplifies the analysis.

Let us now examine the sequence $b_n$ introduced in \eqv(1.theo1.0) and defined in \eqv(1.theo1.3).  
This sequence is a priori random in the random environment and depends on a sequence, $\theta_n$, 
that can itself be chosen within the two widely different bounds of \eqv(1.theo1.0'). The next proposition 
provides upper and lower bounds on $b_n$ that are not affected  by the choice of $\theta_n$.

\begin{proposition}
    \TH(1.prop2)
 
Given $0<\varepsilon<1$, let $c_n$ and  $\theta_n$ be as in Theorem \thv(1.theo1).
Then, there exists a subset $\O'\subseteq \O$ with $\P(\O')=1$ such that on $\O'$, 
for all but a finite number of indices $n$
\be
\left(n^{c_-}(r^{\star}_n)^{1+\a_{n}(\varepsilon)+o(1)} \right)^{-1}
\leq b_n\leq n^{c_+}(r^{\star}_n)^{1+\a_{n}(\varepsilon)}
\Eq(1.prop2.1)
\ee
where $0<c_-, c_+\leq \infty$ are numerical constants.
Thus $\lim_{n\rightarrow\infty}n^{-1}\log a_n=\varepsilon$ $\P$-a.s..
\end{proposition}

\begin{remark} 
The definition \eqv(1.theo1.0)-\eqv(1.theo1.3) of $a_n$ and that of the sequence $R_N$ 
in (2.10) of Ref.~\cite{CW15} have an obvious family resemblance. Our control of $a_n$ through 
Proposition \thv(1.prop2) implies the behavior conjectured in item 4 page 4 of that paper.
\end{remark}

\begin{remark} 
One may wonder whether the lower bound of \eqv(1.theo1.0') can be improved.
The main technical obstacle to doing so is the lower bound on mean hitting times 
of Lemma \thv(3.lem2).  In particular, trying to improve the bound \eqv(3.prop1.1) 
on the spectral gap by choosing $\eta_n$ larger, say as large as 1 as in Ref.~\cite{CW15}, 
can at best improve the constant $\sfrac{4}{1-\a(\varepsilon)}$ in front of $\log r^{\star}_n$ 
in \eqv(1.theo1.0').
\end{remark}

The rest of the paper is organized as follows.  Section \thv(S2) is concerned with the properties 
of the REM's landscape: several level sets that play an important role in our analysis are introduced 
and  their properties collected. Section \thv(S3) gathers all needed results on the exploration process $Y_n$.
The proof of Theorem \thv(1.theo1) can then begin.  Section \thv(S4),   \thv(S5), and  \thv(S6) are devoted, 
respectively, to the verification of Condition (B1), (B2), and  (B3) of  Theorem \thv(1.theoB). The proof of 
Theorem \thv(1.theo1) is completed in Section \thv(S7).  Also in Section \thv(S7), the link between the 
blocked clock process  of \eqv(1.theo1.1) and the two-time correlation function \eqv(1.1.2) is made, 
and the proof of Theorem \thv(1.theo0) is concluded. An appendix (Section \thv(S8)) contains the proof 
of the results of Section \thv(S1.3).



\section{Level sets of the REM's landscape: the Top and other sets}
	\TH(S2)
 
Given $V\subseteq\VV_n$ we denote by $G\equiv G(V)$ the undirected graph which has vertex set $V$ 
and edge set $E(G(V))\subseteq \EE_n$ consisting of pairs of vertices $\{x,y\}$ in $V$ with $\dist(x,y)=1$,
where $\dist(x,x')\equiv\frac 12 \sum_{i=1}^n |x_i-x'_i|$ is  the graph distance on $\VV_n$. 
When $\dist(x,y)=1$ we simply write $x\sim y$.
We are concerned with the graph properties of level sets of the form
\be
V_n(\rho)=\left\{x\in\VV_n\mid \t_n(x)\geq r_n(\rho)\right\}
\Eq(2.1)
\ee
where, given  $\rho>0$, the threshold level $r_n(\rho)$ is the sequence defined through
\be
\Eq(2.2)
2^{\rho n}\P(\t_n(x)\geq r_n(\rho))=1.
\ee
Observe that $V_n(\rho)$ can uniquely be decomposed into a collection of subsets
\be
V_n(\rho)=\cup_{l=1}^L C_{n,l}(\rho),\quad C_{n,l}(\rho)\cap C_{n,k}(\rho) \,\,\,\forall  1\leq l\neq k\leq L,\quad L\equiv L_n(\rho),
\Eq(2.4)
\ee
such that each graph $G(C_{n,l}(\rho))$ is connected but any two distinct graphs 
$G(C_{n,l}(\rho))$ and $G(C_{n,k}(\rho))$ are disconnected. With a little abuse of terminology 
we call the sets $C_{n,l}(\rho)$ the connected components of the graph $G(V_n(\rho))$.
As $\rho$ decreases from $\infty$ to $0$, the set  $V_n(\rho)$ grows and the graph $G(V_n(\rho))$ 
potentially acquires new edges. It is  known  \cite{BKL91} that the size of the largest connected 
component $C_{n,l}(\rho)$ undergoes a transiton near the critical value
$
\rho^c\approx\frac{\log n}{n\log 2}
$,
with a unique ``giant'' component of size $\OO(n^{-1}2^n)$ emerging slightly below this value.
As  $\rho$ decreases the small components merge into the giant one, and total connectedness 
is achieved for $\rho$ slightly smaller than $n^{-1}$. One may naturally think of the connected 
components $C_{n,l}(\rho)$ before criticality as containing distinct ``valleys'' of the REM's energy 
landscape,  the level of emergence of the totally connected giant component then being a 
``ground level'' connecting the local valleys.

We now introduce several sets that play  key roles in our analysis. 

\smallskip
\noindent{$\bullet$ \textbf{\emph{The sets  $V^{\star}_n$ and $\overline V^{\star}_n$ (of valleys and hills).}}} 
Let $c_{\star}$ be as in \eqv(1.4.1) and set
\be
\rho_n^{\star}\equiv\frac{c_{\star}\log n}{n\log 2}.
\Eq(2.0)
\ee
Thus, taking $\rho=\rho_n^{\star}$ in \eqv(2.1)-\eqv(2.4), $r^{\star}_n\equiv r_n(\rho_n^{\star})$ and
the set $V^{\star}_n\equiv V_n(\rho_n^{\star})$ decomposes into
\be
V^{\star}_n
=\cup_{l=1}^{L^{\star}} C^{\star}_{n,l},\quad C^{\star}_{n,l}\cap C^{\star}_{n,k}\,\,\,\forall  l\neq k,\quad 
L^{\star}\equiv L_n(\rho_n^{\star}),
\Eq(2.6)
\ee
where the $C^{\star}_{n,l}$ are the connected components of the graph $G(V_n(\rho_n^{\star}))$. 
According to our earlier picture they contain ``valleys'' of the  landscape. Since $H_n(x)$ is symmetrical the set
\be
\overline V^{\star}_n\equiv  \overline V_n(\rho_n^{\star})=\left\{x\in\VV_n\mid \t^{-1}_n(x)\geq r^{\star}_n\right\}
\Eq(2.1anti)
\ee
obtained from $V_n(\rho_n^{\star})$ by substituting  $-\HH_n(x)$ for $\HH_n(x)$ in \eqv(1.1.3) 
has the same random graph properties as $V^{\star}_n$ but now contains ``hills''.
As in \eqv(2.6) we write
\be
\overline V^{\star}_n\equiv \overline V_n(\rho_n^{\star})
=\cup_{l=1}^{L^{\star}} \overline C^{\star}_{n,l},\quad \overline C^{\star}_{n,l}\cap \overline C^{\star}_{n,k}
\,\,\,\forall  l\neq k,\quad \overline L^{\star}\equiv \overline L_n(\rho_n^{\star}),
\Eq(2.6anti)
\ee
where $\overline C^{\star}_{n,l}$ are the connected components of the graph 
$G(\overline V_n(\rho_n^{\star}))$.  With this definition \eqv(1.4.2)  becomes
\be
\wt\l_n(x,y)=
\begin{cases}
\frac{1}{n}
e^{-\b\max(\HH_n(y),\HH_n(x))},
&\hbox{\rm if } 
x\notin\overline V^{\star}_n,\\
\frac{1}{n r^{\star}_n}
e^{-\b\left[\HH_n(y)-\HH_n(x)\right]^+},
&\hbox{\rm if } 
x\in\overline V^{\star}_n.
\end{cases}
\Eq(1.4.6)
\ee
Furthermore, by \eqv(1.03.5), denoting by $\del A=\{x\in\VV_n \mid \dist(x,A)=1\}$ 
the outer boundary of $A\subset \VV_n$,  we have that for all $x\in\del V^{\star}_n$,
\be
\wt\l_n(x)=
\sum_{y\in (V^{\star}_n)^c}\wt\l_n(x,y)
+\left((n r^{\star}_n)^{-1}\1_{\{x\in \overline V^{\star}_n\}}
+\t_n(x)n^{-1}\1_{\{x\in (\overline V^{\star}_n)^c\}}\right)|\del x\cap V^{\star}_n|.
\Eq(1.4.6')
\ee
Hence, conditional on $V^{\star}_n$, the mean holding time at $x\in (V^{\star}_n)^c$ 
does not depend on the variables $\{\t_n(y), y\in V^{\star}_n\}$ but only depends on 
the variables $\{\t_n(y), y\in (V^{\star}_n)^c\}$.

\smallskip
\noindent{\textbf{\emph{$\bullet$ Immersions in $V^{\star}_n$.}}}  Given any subset $A\subset V^{\star}_n$ 
we call the \emph{immersion of $A$ in $V^{\star}_n$} and denote by $A^{\star}$ the set 
\be
A^{\star}\equiv\cup_{l=1}^{L^{\star}} A^{\star}_{n,l},\quad 
A^{\star}_{n,l}=
\begin{cases}
C^{\star}_{n,l},
&\hbox{\rm if}\,\,\, C^{\star}_{n,l}\cap A\neq \emptyset, \\
\emptyset,&\hbox{\rm else}.
\end{cases}
\Eq(2.9)
\ee
Thus the sets $A^{\star}_{n,l}$ are the valleys $C^{\star}_{n,l}$ that contain at least one element of $A$.
Clearly, $\overline V^{\star}_n\cap V^{\star}_n=\emptyset$. Hence by \eqv(1.4.6), immersed sets have the property that
\be
\wt\l_n(x,y)\leq n^{-1}r^{\star}_n\,\,\,\text{for all}\,\,\, x\sim y 
\,\,\,\text{such that}\,\,\, x\in A^{\star}, y\notin A^{\star}
\text{ or } y\in A^{\star}, x\notin A^{\star}.
\Eq(2.10)
\ee

\noindent{\textbf{\emph{$\bullet$ The top, $T_n$, and the associated sets
$T_n^{\star}$, $T^{\circ}_n$ and $I^{\star}_n$.}}} Given a sequence $\d_n\downarrow 0$
as $n\uparrow\infty$, set $\varepsilon_n\equiv \varepsilon-\d_n$ and let the \emph{top} be the set
\be
T_n\equiv V_n(\varepsilon_n)
\Eq(2.7)
\ee
obtained by taking $\rho=\varepsilon_n$ in \eqv(2.4).  
($\d_n$ will  later be chosen  so that the definitions \eqv(2.7) and \eqv(1.theo1.4) of $T_n$ coincide.)
$T_n$ contains the top of the order statistics or the $\t_n(x)$'s, whence its name.
Since $\rho_n^{\star}\ll\varepsilon_n$,
$
T_n\subset V^{\star}_n
$,
and so $T_n$ can be immersed in $V^{\star}_n$. According to \eqv(2.9) we write
\be
T_n^{\star}\equiv\cup_{l=1}^{L^{\star}} T^{\star}_{n,l}.
\Eq(2.13)
\ee
To each $x\in T_n$ corresponds a unique index $1\leq l\equiv l(x)\leq L^{\star}$ 
such that $x\in T^{\star}_{n,l(x)}$. Of course a given valley $T^{\star}_{n,l}$ may 
contain several vertices of $T_n$. A set that is of special importance in the sequel 
is the subset $T^{\circ}_n$ of vertices of $T_n$ that are alone in their valley,
\be
T^{\circ}_n\equiv\left\{
x\in T_n \mid T^{\star}_{n,l(x)}\cap T_n
=\{x\}
\right\}.
\Eq(2.20)
\ee
This in particular implies that
\be
T^{\circ}_n\subseteq\{x\in\VV_n\mid \t_n(x)\geq r_n(\varepsilon_n), 
\forall_{y\sim x} \t_n(y)< r_n(\varepsilon_n)\}.
\Eq(2.21)
\ee
Finally, define
\be
I^{\star}_n
\equiv
\{x\in \VV_n \mid \t_n(x)\geq r_n(\varepsilon_n), \forall_{y\sim x} (r^{\star}_n)^{-1}<\t_n(y)< r^{\star}_n\}
\subseteq T^{\circ}_n.
\Eq(2.22)
\ee
This is  the largest subset of $T^{\circ}_n$ such that 
$
\dist\bigl((V^{\star}_n\cup \overline V^{\star}_n),I^{\star}_n\bigr)\geq 2
$.

Most of the content of the next three lemmata is taken from \cite{G15}.
The first lemma gives estimates on the size of various sets.

\begin{lemma}
  \TH(2.lem3)
There exists $\O^{\star}\subset \O$ with $\P\left(\O^{\star}\right)=1$
such that on $\O^{\star}$, for all but a finite number of indices $n$,
\be
1\leq |C^{\star}_{n,l}|
\leq\{\rho_n^{\star}[1-2c_{\star}^{-1}(1+\OO(\log n/n))]\}^{-1},\quad1\leq l\leq {L^{\star}}.
\Eq(2.lem3.4)
\ee
The same bounds hold replacing $C^{\star}_{n,l}$ by $\overline C^{\star}_{n,l}$ 
and $L^{\star}$ by $\overline L^{\star}$ in \eqv(2.lem3.4). Furthermore,
\bea
\Eq(2.lem3.1) 
|V^{\star}_n|
\hspace{-6pt}&=&\hspace{-6pt} 
2^n n^{-c_{\star}}(1+o(n^{-c_{\star}})) \text{ and }
|\overline V^{\star}_n|
=
2^n n^{-c_{\star}}(1+o(n^{-c_{\star}})),
\\
\Eq(2.lem3.1bis)
\left|T_n\right|
\hspace{-6pt}&=&\hspace{-6pt} 
2^{n(1-\varepsilon_n)}(1+\OO(n2^{-n\varepsilon_n/2})),
\quad
\\
\left|T^{\circ}_n\right|
\hspace{-6pt}&=&\hspace{-6pt} 
2^{n(1-\varepsilon_n)}(1+\OO(n2^{-n\varepsilon_n/2})),
\Eq(2.lem3.5)
\\
\left|T_n\setminus T^{\circ}_n\right|
\hspace{-6pt}&\leq&\hspace{-6pt} 
n^42^{n(1-2\varepsilon_n)}(1+o(1)),
\Eq(2.lem3.6)
\\
\left|I^{\star}_n\right|
\hspace{-6pt}&=&\hspace{-6pt} 
2^{n(1-\varepsilon_n)}(1-2n^{-c_{\star}+1}(1+o(1)),
\Eq(2.lem3.7)
\\
\left|T^{\circ}_n\setminus I^{\star}_n\right|
\hspace{-6pt}&=&\hspace{-6pt} 
2n^{-c_{\star}+1}2^{n(1-\varepsilon_n)}(1+o(1)).
\Eq(2.lem3.8)
\eea
Finally, introducing the set
\be
\MM_n\equiv \{x\in\VV_n\mid \t_n(x)>\t_n(y)\,\,\text{for all}\,\, y\sim x\}
\Eq(2.lem5.1)
\ee
of local minima of the Hamiltonian,
\be
|\overline V^{\star}_n\cap \MM_n|=0
\Eq(2.lem5.2).
\ee
\end{lemma}

\begin{proof}[Proof of Lemma \thv(2.lem3)] 
Recall that by assumption $c_{\star}>1+\log 4>2$. Eq.~\eqv(2.lem3.4)  is (2.9) of Lemma 2.2 of Ref.~\cite{G15}. 
That the same bound holds for $|\overline C^{\star}_{n,l}|$ follows by symmetry of $\HH_n$.
The estimate \eqv(2.lem3.1) on $|V^{\star}_n|$ is (2.11) of Ref.~\cite{G15} and the estimate on $|\overline V^{\star}_n|$ 
follows again by symmetry of $\HH_n$.
Eq.~\eqv(2.lem3.1bis) and \eqv(2.lem3.7) are proved, respectively, as (2.11) of and  (2.10) of Ref.~\cite{G15}.
The proof of \eqv(2.lem3.6) is a simple adaptation of the proof of lemma 7.1 of Ref.~\cite{G15}.
Clearly, \eqv(2.lem3.5) follows from \eqv(2.lem3.1bis) and \eqv(2.lem3.6), and  
\eqv(2.lem3.8) follows from \eqv(2.lem3.5) and \eqv(2.lem3.7).
It only remains to prove \eqv(2.lem5.1). For this note that 
$
x\in \overline V^{\star}_n\cap \MM_n
$
if and only if
$
\t_n(x)\leq (r^{\star}_n)^{-1}
$
and
$
\t_n(y)<\t_n(x)
$
for all $y\sim x$.
Thus
\bea
\P\bigl(
|\overline V^{\star}_n\cap \MM_n|\geq 1
\bigr)
&\leq &
\sum_{x\in\VV_n}
\P\left(
\t_n(x)\leq (r^{\star}_n)^{-1}, \forall_{y\sim x}\t_n(y)<(r^{\star}_n)^{-1}
\right)
\\
&= &
2^n n^{-c_{\star}}\left(n^{-c_{\star}}\right)^n
\eea
which is summable. Thus, by Borel-Cantelli Lemma, there exists a set of full measure 
such that on that set, for all but a finite number of indices, $|\overline V^{\star}_n\cap \MM_n|=0$.
\end{proof}

The second lemma expresses the function $r_n(\rho)$ defined through \eqv(2.2).

\begin{lemma}[Lemma 2.3 of \cite{G15}]
  \TH(2.lem2)
  For all $\rho>0$, possibly depending on $n$, and such that $\rho n\uparrow \infty$ as $n\uparrow \infty$,
\be
\Eq(2.lem2.1)
 r_n(\rho)=\exp\left\{n\b\b_c(\rho)-(\b/2\b_c(\rho))\left[\log(\b^2_c(\rho)n/2)+\log 4\pi\right]+o(\b/\b_c(\rho))\right\}.
\ee
\end{lemma}

\begin{corollary}
  \TH(2.cor1)
  Set $x_n=\delta_n/\varepsilon$ and assume that  $x_n\downarrow 0$ as $n\uparrow\infty$ and 
  $e^{n\b\b_c(\varepsilon)x}\gg r^{\star}_n$. Then
\bea
{r_n(\varepsilon_n)}/{r_n(\varepsilon)}=\exp\left\{-n\b\b_c(\varepsilon)x\left[1+\sfrac{x}{2}+\OO(x^2)\right]\right\}.
\Eq(2.cor1.1)
\eea
\end{corollary}

The third and last lemma states needed bounds, in particular, on the maximal jump rate.  

\begin{lemma}[Lemma 2.4  of \cite{G15}]
  \TH(2.lem4)
There exists a subset $\O_0\subseteq\O$ with $\P\bigl(\O_0\bigr)=1$ such that on $\O_0$,
for all but a finite number of indices $n$ the following holds: 
\bea
\Eq(2.lem4.1)
&e^{-\b\min\{\max(H_n(y),H_n(x))\,\mid\, (x,y)\in\EE_n\}}
\leq 
e^{\b n\sqrt{\log 2}(1+2\log n/n\log 2)}\equiv n\nu_n,&\\
\Eq(2.lem4.2)
&e^{-\b\min\left\{H_n(x)\,|\, x\in \VV_n\right\}}\leq e^{\b n\sqrt{2\log 2}(1+2\log n/n)}.&
\eea
Thus, $\max_{(x,y)\in\EE_n}\wt\l_n(x,y)\leq \nu_n$.
\end{lemma}



\section{Properties of the exploration process $Y_n$}
 \TH(S3)
 
In this Section we establish the properties of the exploration process needed in the rest of the paper.
By \eqv(1.4.3) with $\eta_n\equiv(r^{\star}_n)^{-1}$ and \eqv(2.1anti), the invariant measure $\pi_n$ of 
$Y_n$ can be written as
\be
\pi_n(x)=
\frac{\1_{\{x\notin\overline V^{\star}_n\}}+ r^{\star}_n\t_n(x)\1_{\{x\in\overline V^{\star}_n\}}}{Z_{\b,n}},\quad x\in \VV_n
\Eq(3.lem0.0) 
\ee
where
$
Z_{\b,n}\equiv |\VV_n\setminus \overline V^{\star}_n|+\sum_{x\in\overline V^{\star}_n}r^{\star}_n\t_n(x)
$.

\begin{lemma}
\TH(3.lem0) 
On $\O^{\star}$, for all but a finite number of indices $n$,
\be
2^n(1- n^{-c_{\star}}(1+o(n^{-c_{\star}})) \leq Z_{\b,n}\leq 2^n.
\Eq(3.lem0.1) 
\ee
Therefore, if $A$ is any of the sets $T_n$, $T^{\circ}_n$, $T_n\setminus T^{\circ}_n$, $I^{\star}_n$ 
or $T^{\circ}_n\setminus I^{\star}_n$ in \eqv(2.lem3.1bis)- \eqv(2.lem3.8),
\be
\pi_n(A)=|A|2^{-n}(1+o(1))
\Eq(3.lem0.2) 
\ee
whereas for any $x\in \VV_n$,
\be
\pi_n(x)\leq 2^{-n}(1+o(1)).
\Eq(3.lem0.3) 
\ee
\end{lemma}

\begin{proof} 
Since $\{x\in\overline V^{\star}_n\}=\{r^{\star}_n\t_n(x)\leq 1\}$, 
$
|\VV_n\setminus \overline V^{\star}_n|
\leq 
Z_{\b,n}
\leq  |\VV_n\setminus \overline V^{\star}_n|+|\overline V^{\star}_n|
\leq 2^n
$.
Eq.~\eqv(3.lem0.1)  then follows from \eqv(2.lem3.1) of Lemma \thv(2.lem3). 
Eq.~\eqv(3.lem0.3) is then immediate and \eqv(3.lem0.2)  follows from the fact that 
$A\cap \overline V^{\star}_n=\emptyset$ for each of the mentioned sets.
\end{proof}
 
 \subsection{Spectral gap and mixing condition}
  \TH(S3.1)
  
Denote by $\wt L_n$ the Markov generator matrix of $Y_n$ (that is, the matrix with off-diagonal  
entries $\wt\l_n(x,y)$ and diagonal entries $-\wt\l_n(x)$), and by 
$
0=\vartheta_{n,0}<\vartheta_{n,1}\leq \dots\leq \vartheta_{n,2^n-1}
$
the eigenvalues of $-\wt L_n$.

\begin{proposition}
    \TH(3.prop1)
    
If $c_{\star}>1+\log 4$ then for all $\b>0$, there exists a subset $\O_1\subset\O$ with  
$\P\left(\O_1\right)=1$ such that, on $\O_1$, for all but a finite number of indices $n$, 
\be
1/\vartheta_{n,1}\leq \sfrac{5}{2}n^2r^{\star}_n(1+o(1))\equiv \tilde\kappa_n
\Eq(3.prop1.1)
\ee
\end{proposition}

As a direct consequence on Proposition \thv(3.prop1), Condition (B0) of Theorem \thv(1.theoB) 
is satisfied $\P$-almost surely with e.g.
\be
\kappa_n\equiv\lfloor  n^4r^{\star}_n(1+o(1))\rfloor.
\Eq(3.prop2.1)
\ee

\begin{proposition}
    \TH(3.prop2)
Under the assumptions of Proposition \thv(3.prop1), on $\O_1$, 
for all but a finite number of indices $n$, for all $\b>0, all $pairs $x,y\in\VV_n$, and all $t\geq 0$,
\be
|P_{x}\left(Y_n(t+\kappa_n)=y\right)-\pi_n(y)|\leq \rho_n\pi_n(y),
\Eq(3.prop2.2)
\ee
where $\kappa_n$ is given by \eqv(3.prop2.1) and $\rho_n<e^{-n}$.
\end{proposition}

\begin{proof}[Proof of Proposition \thv(3.prop1)] 
The proof of \eqv(3.prop1.1) relies on a well known Poincar\'e inequality, 
taken from \cite{DS} (see Proposition 1' p.~38), applied to the stochastic matrix
$
\wt P_n=I+\nu_n^{-1}\wt L_n
$
where $I$ denotes the identity matrix and $\nu_n$ is defined in Lemma \thv(2.lem4).
By Lemma \thv(2.lem4),  on $\O_0$, for all  $n$ large enough,
\be
\max_{(x,y)\in\EE_n}\wt\l_n(x,y)\leq\nu_n<\infty.
\Eq(3.prop1.2)
\ee
Thus, on $\O_0$, for large enough $n$, the entries $\wt p_n(x,y)$ of $\wt P_n$ obey 
$0\leq \wt p_n(x,y)\leq 1$ and $\sum_{y\in \VV_n}\wt p_n(x,y)=1$. The Poincar\'e inequality 
of interest now reads as follows. For each pair of distinct vertices $x,y\in \VV_{n}$, 
choose a path $\g_{x,y}$ going from $x$ to $y$  in the graph $G(\VV_{n})$. 
Paths may have repeated vertices but a given edge appears at most once in a  given path. 
Let $\G_n$ denote such a collection of paths (one for each pair $\{x,y\}$). Then
\be
\Eq(3.prop1.3)
\textstyle
1/\vartheta_{n,1}\leq\nu_n^{-1}\max_{e}\rho^{-1}_n(e)\sum_{\g_{x,y}\ni e}\left|\g_{x,y}\right|\pi_n(x)\pi_n(y),
\ee 
where the max is over all edges $e=\{x',y'\}$ of $G(\VV_n)$, $\rho_n(e)\equiv\pi_{n,l}(x')\wt p_n(x',y')$, 
and the summation is over all paths $\g_{x,y}$ in $\G_n$ that pass through $e$. 

The quality of the bound \eqv(3.prop1.3) now depends on making a judicious choice of 
the set of paths $\G_n$. We adopt the following clever choice made in Ref.~\cite{FIKP}.
Given $i\in\{1,\dots\,n\}$ and given two vertices $x$ and $x'\in\VV_n$ such that $x_i\neq x'_i$, 
let $\g_{x,x'}^i$ be the path obtained by going left to right cyclically from  $x$ to $x'$, 
successively flipping the disagreeing coordinates, starting from the $i$-th coordinate.
Set
$
\G_n^i=\left\{\g_{x,x'}^i, x,x'\in\VV_n\right\}
$,
$1\leq i\leq n$. These paths are ordered in an obvious way.
Given $x,x'$ and $\g_{x,x'}$, let $\overline\g_{x,x'}$ be the set of vertices visited
by the path $\g_{x,x'}$, and let $\g_{x,x'}^{int}=\overline\g_{x,x'}\setminus\{x,x'\}$
be the subset of ``interior'' vertices. We next split the set of vertices $\VV_n$ into 
{\it good} ones and {\it bad} ones. Recalling \eqv(2.6anti), we  say that a vertex  is 
good if it does not belong to $\overline V^{\star}_n$; otherwise it is bad. 
We say that a path $\g$ is good if all its interior points $\g^{int}$ are good, 
and that a set of paths is good if all its elements are good.

The (random) set of path $\G_n$ is then constructed as follows:

\item{(i)} Consider pairs $x$ and $x'$ such that $\dist(x,x')\geq n/\log n$.
If $\{\g_{x,x'}^i,  1\leq i\leq n\}$  contains a good path, choose the first such for $\G_n$; 
otherwise choose $\g_{x,x'}^1$.

\item{(ii)} Consider pairs $x$ and $x'$ such that $\dist(x,x')< n/\log n$. If there is 
a good vertex $x''\in\VV_n$ such that  $\dist(x,x'')\geq n/\log n$ and 
$\dist(x'',x')\geq n/\log n$, and if there are good paths, one in 
$\left\{\g_{x,x''}^i,  1\leq i\leq n\right\}$ and one in $\left\{\g_{x'',x'}^i,  1\leq i\leq n\right\}$, 
such that the union of these two good paths is a self avoiding path of length less than 
$n$, select this union as the path connecting $x$ to $x'$ in $\G_n$ 
(notice that this is a good path); otherwise choose $\g_{x,x'}^1$.

It turns out that this $\G_n$ is almost surely good. More precisely,  set
$
\O^{\scriptscriptstyle{\textsf{GOOD}}}_n=\{\G'_n\,\hbox{\rm is good}\,\}
$,
$n\geq 1$,
and 
$
\O^{\scriptscriptstyle{\textsf{GOOD}}}=\liminf_{n\rightarrow \infty}\O^{\scriptscriptstyle{\textsf{GOOD}}}_n
$.

\begin{proposition}[Proposition 4.1 of \cite{FIKP}]
  \TH(4'.prop3)
 If $c_{\star}>1+\log 4$ then $\P\bigl(\O^{\scriptscriptstyle{\textsf{GOOD}}}\bigr)=1$.
\end{proposition}

From now on we assume that  $\o\in \O^{\scriptscriptstyle{\textsf{GOOD}}}$ so that, 
for all large enough $n$, $\G_n$ is good. Note that the paths of $\G_n$ have length 
smaller than $n$. Hence \eqv(3.prop1.3) yields
\be
\begin{split}
1/\vartheta_{n,1}
& \leq  n\max_{e=\{x',y'\}}\bigl(\pi_n(x')\wt\l_n(x',y')\bigr)^{-1}\sum_{\g_{x,y}\ni e}\pi_n(x)\pi_n(y)
\\
 &=\max_{e=\{x',y'\}}\frac{n^2}{\min(\t_n(y'), \t_n(x'))}
\sum_{\g_{x,y}\ni e}\frac{{\min\left(1,r^{\star}_n\t_n(x)\right)\min\left(1,r^{\star}_n\t_n(y)\right)}}{Z_{\b,n}}
\end{split}
\Eq(3.prop1.4)
\ee
where the final equality follows from \eqv(1.4.2),  \eqv(1.4.3)  (with $\eta_n\equiv(r^{\star}_n)^{-1}$), 
and \eqv(3.lem0.0). Also note that since bad vertices (i.e.~vertices of $\overline V^{\star}_n$) can 
appear only at the ends of any path, the paths of $\G_n$ do not contain any edge of the graph 
$G_n(\overline V^{\star}_n)$. This prompts us to write
$
1/\vartheta_{n,1}
\leq \max\{\KK_{1,n},\KK_{2,n}, \KK_{3,n}\}
$ 
where $\KK_{1,n}$, $\KK_{2,n}$, and $\KK_{3,n}$ are obtained, respectively,  
by restricting the maximum in \eqv(3.prop1.4)  to the maximum over edges 
$e=\{x',y'\}$ with $x'\in\overline V^{\star}_n$ and $y'\notin\overline V^{\star}_n$, 
$x'\notin\overline V^{\star}_n$ and  $y'\in\overline V^{\star}_n$, and 
$x'\notin\overline V^{\star}_n$ and $y'\notin\overline V^{\star}_n$.
To bound $\KK_{1,n}$ note that the sum over paths that contain $e=\{x',y'\}$ 
reduces to the sum  over all paths starting in $x'$ that contain $e$, so that
\be
\KK_{1,n}=\max_{e=\{x',y'\}:x'\in\overline V^{\star}_n, y'\notin\overline V^{\star}_n}
\frac{n^2\min\bigl(1,r^{\star}_n\t_n(x')\bigr)}{\min(\t_n(y'), \t_n(x'))}
\sum_{\g_{x',y}\ni e}\pi_n(y)\leq n^2r_n.
\Eq(3.prop1.7)
\ee 
By symmetry of the bound \eqv(3.prop1.4),  $\KK_{2,n}\leq n^2r_n$. Finally, 
$\min(\t_n(y'), \t_n(x'))\geq 1/r^{\star}_n$ for all  $x',y'\notin\overline V^{\star}_n$ and 
$\min\left(1,r^{\star}_n\t_n(x)\right)\min\left(1,r^{\star}_n\t_n(y)\right)\leq 1$ for all  $x,y\in\VV_n$.
Thus
\be
\KK_{3,n}
\leq n^2r^{\star}_nZ^{-1}_{\b,n}
\max_{e\in G(\VV_n)}
|\{\g\in\G_n \mid  e\in \g\}|
\leq n^2r^{\star}_nZ^{-1}_{\b,n} (2^{n-1}+ 2^{2n/\log n}),
\Eq(3.prop1.9)
\ee
where we used that the number of  paths connecting vertices at distance $n/\log n$ 
or more apart is at most $2^{n-1}$ (see e.g. Example 2.2, p.~45 in Ref.~\cite{DS} for 
this well known bound) whereas, arguing as in Ref.~\cite{FIKP} (see Section 4.2.2, page 934), 
the number of  paths connecting vertices less than $n/\log n$ apart and containing $e$ is 
bounded above by the volume of a hypercube of dimension at most $n/\log n$ around $e$, 
and so, is smaller than $2^{2n/\log n}$. In view of Lemma \thv(3.lem0) we have that on 
$\O^{\star}\cap \O^{\scriptscriptstyle{\textsf{GOOD}}}$, for all but a finite number of indices $n$,
\be
\KK_{3,n}
\leq \sfrac{1}{2}n^2r^{\star}_n(1+o(1)).
\Eq(3.prop1.10)
\ee
Collecting our bounds and taking $\O_1=\O_0\cap\O^{\star}\cap \O^{\scriptscriptstyle{\textsf{GOOD}}}$ 
yields \eqv(3.prop1.1) and ends the proof.
\end{proof}

\begin{proof}[Proof of Proposition \thv(3.prop2)] 
It is well know that for reversible irreducible Markov processes, bounds on spectral gaps 
yield bounds on their total variation distance $\|\cdot\|_{\text{var}}$ to stationarity. 
For instance, Proposition 3 of Ref.~\cite{DS} applied to $Y_n$ states that for all $x\in\VV_n$ and all $t>0$,
\be
4\left\|P_{x}\left(Y_n(t)=\cdot\right)-\pi_n(\cdot)\right\|_{\text{var}}^2\leq\sfrac{1-\pi_n(x)}{\pi_n(x)}e^{-2t\vartheta_{n,1}}.
\Eq(3.prop2.3)
\ee
By \eqv(3.lem0.0),  Lemma \thv(3.lem0), and \eqv(2.lem4.2) of Lemma \thv(2.lem4), on $\O_0\cap \O^{\star}$, 
for all but a finite number of indices $n$, 
$
\sup_{z\in\VV_n}\pi^{-1}_n(z)\leq (2^n/r^{\star}_n)e^{\b n\sqrt{2\log 2}(1+2\log n/n)}.
$
The claim of Proposition \thv(3.prop2) now readily follows from this, \eqv(3.prop2.3), and Proposition \thv(3.prop1),  
choosing  $\kappa_n$ as in \eqv(3.prop2.1).
\end{proof}

 \subsection{Hitting time for the stationary chain}
  \TH(S3.2)

Drawing heavily on  Aldous and Brown's work \cite{AB1}, this section collects results on 
hitting times for the process $Y_n$ at stationarity. Let
\be
H(A)=\inf\{t\geq 0 \mid Y_n(t)\in A\}
\Eq(3.2.1)
\ee
be the hitting time of $A\subseteq\VV_n$. We begin with bounds on the mean value of $H(A)$.  

\begin{lemma}
   \TH(3.lem2)
On $\O_1$, for all but a finite number of indices $n$, for all $A\subseteq\VV_n$,
\be
\frac{(1-n\pi_n(A))^2}{r^{\star}_nn\pi_n(A)(1-\pi_n(A))}
\leq
\frac{E_{\pi_n}H(A)}{1-\pi_n(A)}
\leq
\frac{\tilde\kappa_n}{\pi_n(A)}.
\Eq(3.lem2.1)
\ee
(If 
$
\dist(V^{\star}_n,A)>1,
$ 
$n\pi_n(A)$ can be replaced by $\pi_n(A)$ in the right-hand side).
\end{lemma}

The next lemma gives bounds on the density function $h_{n,A}(t)$, $t>0$, of  $H(A)$ 
when $Y_n$ starts in its invariant measure, $\pi_n$. 

\begin{lemma}
   \TH(3.lem3)
On $\O_1$, for all but a finite number of indices $n$, for all $A\subseteq \VV_n$ and all $t>0$,
$$
\frac{1}{E_{\pi_n}H(A)}\left(1-\frac{\tilde\kappa_n}{E_{\pi_n}H(A)}\right)^2\left(1-\frac{t}{E_{\pi_n}H(A)}\right)
\leq
h_{n,A}(t)
\leq 
\frac{1}{E_{\pi_n}H(A)}\left(1+\frac{\tilde\kappa_n}{2t}\right).
$$
\end{lemma}

The bounds of Lemma \thv(3.lem3) imply that $h_{n,A}(t)\approx \frac{1}{E_{\pi_n}H(A)}$ when 
$
\tilde\kappa_n\ll t \ll E_{\pi_n}H(A)
$.
Complementing this, Lemma \thv(3.lem1) is well suited to dealing with ``small'' values of $t$.

\begin{lemma}
   \TH(3.lem1)
On $\O^{\star}$, for all but a finite number of indices $n$, 
for all $A\subseteq\VV_n$ and all $t>0$,
\be
P_{\pi_n}(H(A)> t)\geq (1-n\pi_n(A))\exp\left(-t\frac{r^{\star}_nn\pi_n(A)}{1-n\pi_n(A)}\right).
\Eq(3.lem1.1)
\ee
In particular, for any $A$ and any sequence $t_n$ such that 
$t_nr^{\star}_nn\pi_n(A)\rightarrow 0$ as $n\rightarrow\infty$,
\be
P_{\pi_n}(H(A)\leq  t_n)<t_nr^{\star}_nn\pi_n(A)(1+o(1)).
\Eq(3.lem1.2)
\ee
If $A\subset\VV_n\setminus V^{\star}_n$ the factor $n$  in front of $\pi_n(A)$ in \eqv(3.lem1.1) 
and  \eqv(3.lem1.2) can be suppressed.   
\end{lemma}

The next Corollary is stated for later convenience.

\begin{corollary}
    \TH(3.cor1)
Under the assumptions of Lemma \thv(3.lem1) the following holds:
For all $0<\varepsilon<1$,  for any sequence $t_n$ such that 
$t_n r^{\star}_nn2^{-n\varepsilon_n}\rightarrow 0$ as $n\rightarrow\infty$
\be
P_{\pi_n}(H(T_n\setminus T^{\circ}_n)\leq t_n)
\leq 
t_n r^{\star}_nn^52^{-2n\varepsilon_n}(1+o(1)),
\Eq(3.cor1.1)
\ee
\be
P_{\pi_n}(H(T^{\circ}_n)\leq t_n)
\leq 
t_n r^{\star}_nn2^{-n\varepsilon_n}(1+o(1)).
\Eq(3.cor1.1')
\ee
\end{corollary}
We now prove these results, beginning with Lemma \thv(3.lem1).

\begin{proof}[Proof of Lemma \thv(3.lem1)] 
Write $A=B\cup B^c$ where $B=A\cap V^{\star}_n$ and $B^c= A\setminus B$.
Let $B^{\star}$ be the immersion of  $B$ in $V^{\star}_n$  (see \eqv(2.9)). 
Since $A\subseteq B^{\star}\cup B^c$, $H(A)\geq H(B^{\star}\cup B^c)$, and
\be
P_{\pi_n}(H(A)> t)\geq P_{\pi_n}(H(B^{\star}\cup B^c)> t).
\Eq(3.lem1.5)
\ee
To bound the right-hand side of \eqv(3.lem1.5), we use a well know lower bound on 
hitting times for stationary reversible chains taken from   Ref.~\cite{AB1} 
(combine Theorem 3 and Lemma 2 therein) that states that for all 
$C\subseteq\VV_n$ and all $t>0$,
\be
P_{\pi_n}(H(C)> t)\geq (1-\pi_n(C))\exp\left(-t\frac{q_n(C,C^c)}{1-\pi_n(C)}\right)
\Eq(3.lem1.3)
\ee
where, for for any two sets $C$ and $\wt C$ such that $C\cap \wt C=\emptyset$, 
\be
q_n(C,\wt C)\equiv\sum_{x\in C}\sum_{y\in\wt C}\pi_n(x)\wt\l_n(x,y).
\Eq(3.lem1.4)
\ee
Let us thus evaluate \eqv(3.lem1.4) with $C=B^{\star}\cup B^c$.
Clearly
$
q_n(B^{\star}\cup B^c,(B^{\star}\cup B^c)^c)
\leq 
q_n(B^{\star},(B^{\star}\cup B^c)^c)+q_n(B^c,(B^{\star}\cup B^c)^c)
$.
Clearly also, by \eqv(1.4.6),
$
\wt\l_n(x,y)\leq n^{-1}r^{\star}_n
$
for any $x\in B^c$ and any $y\sim x$. Thus
$
q_n(B^c,(B^{\star}\cup B^c)^c)
\leq r^{\star}_n\pi_n(B^c).
$
Next, by \eqv(2.10),
$
q_n(B^{\star},(B^{\star}\cup B^c)^c)
\leq 
r^{\star}_n\pi_n(B^{\star})
$.
Thus
\be
q_n(B^{\star}\cup B^c,(B^{\star}\cup B^c)^c)
\leq 
r^{\star}_n[\pi_n(B^{\star})+\pi_n(B^c)].
\Eq(3.lem1.6)
\ee
Denoting by  $C^{\star}_{n,l(x)}$ the (unique) component of $B^{\star}$ (see \eqv(2.9)) that contains $x$, we have
$
|B^{\star}|
\leq 
|\cup_{x\in B}C^{\star}_{n,l(x)}|
\leq 
|B|\max_{x\in B}|C^{\star}_{n,l(x)}|
$
where  by \eqv(2.lem3.4), on  $\O^{\star}$, $|C^{\star}_{n,l(x)}|\ll n$. By this and \eqv(3.lem0.0) we get
$
\pi_n(B^{\star})=Z^{-1}_{\b,n}|B^{\star}|\leq nZ^{-1}_{\b,n}|B|=n\pi_n(B)
$.
Therefore,
\be
\pi_n(B^{\star}\cup B^c)\leq \pi_n(B^{\star})+ \pi_n(B^c)\leq 
n\pi_n(B)+ \pi_n(B^c)\leq n\pi_n(B\cup B^c)=n\pi_n(A).
\Eq(3.lem1.7)
\ee
Using \eqv(3.lem1.7) in the right-and side of \eqv(3.lem1.6) and plugging the result in \eqv(3.lem1.3) 
finally yields \eqv(3.lem1.1). Clearly, if  $A\subset\VV_n\setminus V^{\star}_n$ then $B=\emptyset$ 
and the right-and side of \eqv(3.lem1.6) reduces to $r^{\star}_n[\pi_n(\emptyset)+\pi_n(B^c)]=r^{\star}_n\pi_n(A)$.
\end{proof}

\begin{proof}[Proof of Corollary \thv(3.cor1)]
This follows  from \eqv(3.lem0.2)  of Lemma \thv(3.lem0),  \eqv(2.lem3.5), and \eqv(2.lem3.6).
\end{proof}

\begin{proof}[Proof of Lemma \thv(3.lem3)] 
Proceed as in Lemma 13 of Ref.~\cite{AB1} and use Proposition \thv(3.prop1).
\end{proof}

\begin{proof}[Proof of Lemma \thv(3.lem2)] 
The rightmost inequality is that of Lemma 2 of Ref.~\cite{AB1} combined with
Proposition \thv(3.prop1).  Lemma 2 of Ref.~\cite{AB1}  also states that
for $C\subseteq\VV_n$ and $q_n(C,C^c)$ defined as in \eqv(3.lem1.4),
\be
\frac{E_{\pi_n}H(C)}{1-\pi_n(C)}
\geq
\frac{1-\pi_n(C)}{q_n(C,C^c)}.
\Eq(3.lem2.2)
\ee
Given $A\subseteq\VV_n$ let $B^{\star}$ and $B^c$ be defined as in the first 
line of the proof of Lemma \thv(3.lem1). Since $H(A)\geq H(B^{\star}\cup B^c)$,
$
E_{\pi_n}H(A)\geq E_{\pi_n}H(B^{\star}\cup B^c)
$.
Using  \eqv(3.lem2.2) with $C=B^{\star}\cup B^c$,  \eqv(3.lem2.1)  follows from 
\eqv(3.lem1.6) and the bound on $\pi_n(B^{\star}\cup B^c)$ of \eqv(3.lem1.7).
 \end{proof}


\subsection{On hitting the top starting in the top.}
  \TH(S3.3)

Let  $T^{\circ}_n$ and $I^{\star}_n$ be as in  \eqv(2.20) and   \eqv(2.22).

\begin{proposition}
     \TH(3.prop3)
     
Given $\e>0$ there exists a subset $\O^{\circ}\subset\O$ with  $\P\left(\O^{\circ}\right)=1$ such that on $\O^{\circ}$, 
for all but a finite number of indices $n$, for all $s>0$
\be
|T^{\circ}_n|^{-1}\sum_{x\in T^{\circ}_n}P_{x}\left(H(T^{\circ}_n\setminus x)\leq s\right)
\leq sn^{c_{\star}+3}{r^{\star}_n}\pi_n(T^{\circ}_n).
\Eq(3.prop3.1)
\ee    
\end{proposition}

The next proposition is a variant of Proposition \thv(3.prop3) that we state for later convenience. 

\begin{proposition}
     \TH(3.prop3bis)

Under the assumptions and with the notation of Proposition \thv(3.prop3), on $\O^{\circ}$,  
for all but a finite number of indices $n$, for all $s>0$
\be
|T^{\circ}_n\setminus I^{\star}_n|^{-1}\sum_{x\in T^{\circ}_n\setminus I^{\star}_n}P_{x}\left(H(I^{\star}_n)\leq s\right)
\leq s n^{2}r^{\star}_n\pi_n(I^{\star}_n)(1+o(1)).
\Eq(3.prop3.1bis)
\ee    
\end{proposition}

\begin{proof}[Proof of Proposition \thv(3.prop3)]  
A key ingredient of the proof is an explicit expression of the density function  $h^x_{n,A}(t)$, $t\geq 0$, 
of the hitting time $H(A)$ when $Y_n$ starts in  $x\in A^c\equiv\VV_n\setminus A$. We first state this 
expression in full generality as given in   \cite{Kei} (see Section 6.2, p.~83). Consider the stochastic matrix  
$\wt P_n=(\wt p_n(x,y))$ defined above  \eqv(3.prop1.2). Denote by $Q_{n}=(q_n(x,y))$  the matrix with 
entries $q_n(x,y): A^c\times A^c \rightarrow \R$  given by $q_n(x,y)=\wt p_n(x,y)$. This is the sub-matrix 
of $\wt P_n$ on $A^c\times A^c$. Thus $Q_{n}$  is sub-stochastic. Similarly, denote by $R_{n}=(r_n(x,y))$ 
the sub-matrix of $\wt P_n$ on $A^c\times A$. Let $1_{A}$ be the vector of $1$'s on $A$ and let $\d_x$ 
be the vector on $A^c$ taking value 1 at $x$ and zero else. Then, for all $x\in A^c$, 
\be
h^x_{n, A}(t)=\nu_n\sum_{k=0}^{\infty}\frac{(\nu_nt)^k}{k!}e^{-\nu_nt}\left(\d_x,Q_{n}^kR_{n}1_{A}\right), \quad t\geq 0,
\Eq(3.prop3.2)
\ee
where $(\cdot,\cdot)$ denotes the inner product in $\R^{|A^c|}$. Consequently, for all $s>0$,
\be
P_{x}\left(H(A)\leq s\right)
=\int_0^{s}\nu_n\sum_{k=0}^{\infty}\frac{(\nu_nt)^k}{k!}e^{-\nu_nt}\left(\d_x,Q_{n}^kR_{n}1_{A}\right)dt.
\Eq(3.prop3.3)
\ee
For later reference we also denote by 
$
(h^x_{n, y, A}(t))_{y\in A}
$
the vector whose components are, for each $y\in A$, the joint density 
that $A$ is reached at time $t$, and that arrival to that set occurs in state $y$, namely,
$h^x_{n, y, A}(t)$ is defined as in \eqv(3.prop3.2) substituting $\d_y$ for $1_{A}$ therein;
as a result
$
h^x_{n, A}(t)=\sum_{y\in A}h^x_{n, y, A}(t)
$.

Returning to  \eqv(3.prop3.1), a first order Tchebychev inequality yields, for all $\e>0$
\bea
\nonumber
\P\left[
\textstyle
\sum_{x\in T^{\circ}_n}P_{x}\left(H(T^{\circ}_n\setminus x)\leq s\right)
\geq \e
\right]
\hspace{-6pt}&\leq &\hspace{-6pt}
 \e^{-1}\E\left[
 \textstyle
\sum_{x\in T^{\circ}_n}P_{x}\left(H\bigl(T_n^{\star}\setminus T^{\star}_{n,l(x)}\bigr)\leq s\right)
\right]
\\
\hspace{-6pt}&\equiv &\hspace{-6pt}
 \e^{-1}\WW_n,
\Eq(3.prop3.5)
\eea
where
$
T_n^{\star}\equiv\cup_{l=1}^{L^{\star}} T^{\star}_{n,l}
$
is defined in \eqv(2.13) and $1\leq l(x)\leq L^{\star}$ denotes the (unique) index such that 
$T^{\star}_{n,l(x)}\cap T^{\circ}_n=\{x\}$ in \eqv(2.20).
By \eqv(3.prop3.3) with $A=T_n^{\star}\setminus T^{\star}_{n,l(x)}$,
\be
\WW_n
=
\sum_{x\in \VV_n}
\int_0^{s}dt\sum_{k=1}^{\infty}
\frac{(\nu_nt)^k}{k!}e^{-\nu_nt}
\WW_{n,k}(x)
\Eq(3.prop3.8)
\ee
where 
\be
\WW_{n,k}(x)
\equiv  \E\left[\1_{\{x\in T^{\circ}_n\}}\nu_n\left(\d_x,Q_{n}^kR_{n}1_{T_n^{\star}\setminus T^{\star}_{n,l(x)}}\right)\right].
\Eq(3.prop3.9)
\ee
Note that the term $k=0$ is zero. For $k\geq 1$  the matrix term in \eqv(3.prop3.9) reads,
\be
\1_{\{x\in T^{\circ}_n\}}\nu_n\left(\d_x,Q_{n}^kR_{n}1_{T_n^{\star}\setminus T^{\star}_{n,l(x)}}\right)
=\1_{\{x\in T^{\circ}_n\}}\sum_{y\in (T_n^{\star}\setminus T^{\star}_{n,l(x)})^c}q_n^{(k)}(x,y)
\sum_{z\in T_n^{\star}\setminus T^{\star}_{n,l(x)}}\nu_nr_n(y,z)
\Eq(3.prop3.6)
\ee
where $q_n^{(k)}(x,y)$ denotes the entries of $Q^k_n$.  By \eqv(2.10),
for all $y\in (T_n^{\star}\setminus T^{\star}_{n,l(x)})^c$,
\be
\sum_{z\in T_n^{\star}\setminus T^{\star}_{n,l(x)}}\nu_nr_n(y,z)
=
\sum_{z\in T_n^{\star}\setminus T^{\star}_{n,l(x)}}\wt\l_n(y,z)
\leq 
n^{-1}r^{\star}_n\sum_{z\in T_n^{\star}\setminus T^{\star}_{n,l(x)}}\1_{\{z\sim y\}}.
\Eq(3.prop3.7)
\ee
Therefore,  inserting \eqv(3.prop3.7) in \eqv(3.prop3.6), \eqv(3.prop3.9) yields
\be
\WW_{n,k}(x)
 \leq  
\frac{r^{\star}_n}{n}\E\Biggl[\E\biggl[
\1_{\{x\in T^{\circ}_n\}}
\sum_{y\in (T_n^{\star}\setminus T^{\star}_{n,l(x)})^c}q_n^{(k)}(x,y)
\sum_{z\in T_n^{\star}\setminus T^{\star}_{n,l(x)}}\1_{\{z\sim y\}}
\,\bigg|\, V^{\star}_n \biggr]\Biggr]
\Eq(3.prop3.11)
\ee
 where $\E[\cdot\,|\, V^{\star}_n ]$ denotes the conditional expectation given a realization of the set
$V^{\star}_n=\cup_{l=1}^{L^{\star}} C^{\star}_{n,l}$ 
(see \eqv(2.6)), namely, expectation with respect to the measure
\be
\P(\cdot\,|\, V^{\star}_n)=
\frac{\P(\cdot\cap \{\forall_{1\leq l\leq L^{\star}}\forall_{x\in C^{\star}_{n,l}}\t_n(x)\geq r^{\star}_n\}
\cap\{\forall_{x\in C^{\star}_{n,0}}\t_n(x)< r^{\star}_n
\})}
{\P(\{\forall_{1\leq l\leq L^{\star}}\forall_{x\in C^{\star}_{n,l}}\t_n(x)\geq r^{\star}_n\}
\cap\{\forall_{x\in C^{\star}_{n,0}}\t_n(x)< r^{\star}_n
\})},
\Eq(3.prop3.10)
\ee
where we set $C^{\star}_{n,0}\equiv\VV_n\setminus \VV^{\star}_n$ for simplicity.
Thus $\VV_n=\cup_{0\leq l\leq L^{\star}}C^{\star}_{n,l}$ and
 $C^{\star}_{n,l}\cap C^{\star}_{n,l'}=\emptyset$ for all
$0\leq  l\neq l'\leq L^{\star}$, so that if $\LL_i\subset\{0,\dots,L^{\star}\}$, $i=1,\dots,j$, is a collection of disjoint sets,
 functions $f_i$ of the variables $\{\t_n(x), x\in\cup_{l\in\LL_i}C^{\star}_{n,l}\}$, $i=1,\dots,j$, 
are independent under the conditional law \eqv(3.prop3.10). 
Observe now that conditional on $V^{\star}_n$ the entries of the matrix $Q_{n}$ are functions of the variables
$\{\t_n(y), y\in (T_n^{\star}\setminus T^{\star}_{n,l(x)})^c\}$ only:
for off-diagonal entries, i.e.~for $q_n(x,y)$ with $x\neq y$, this is an immediate consequence of \eqv(1.4.6);
for diagonal entries, i.e.~$q_n(x,x)=1-\nu_n^{-1}\wt\l_n(x)$,
this claim follows from \eqv(1.03.5) and \eqv(1.4.6) if  $x\notin\del V^{\star}_n$
and from  \eqv(1.4.6') and \eqv(1.4.6) if  $x\in\del V^{\star}_n$ 
(the boundary set $\del A$ of $A$ is defined above \eqv(1.4.6')).

Next, observe that the sum over $y\in (T_n^{\star}\setminus T^{\star}_{n,l(x)})^c$ in \eqv(3.prop3.11) can 
be restricted to the sum over $y\in \del V^{\star}_n\subseteq C^{\star}_{n,0}$
and use the  definition of $T_n^{\star}$ (see  \eqv(2.13)) to write
\be
\begin{split}
&\sum_{y\in (T_n^{\star}\setminus T^{\star}_{n,l(x)})^c}q_n^{(k)}(x,y)
\sum_{z\in T_n^{\star}\setminus T^{\star}_{n,l(x)}}\1_{\{z\sim y\}}
\\
 =  &
 \sum_{x_1\in\VV_n}\dots\sum_{x_{k-1}\in\VV_n}\sum_{y\in \del V^{\star}_n}\sum_{z\sim y}
\sum_{{0\leq l_1\leq L^{\star}}\atop{C^{\star}_{n,l_1}\cap x_1 \neq \emptyset}}\dots
\sum_{{0\leq l_{k-1}\leq L^{\star}}\atop{C^{\star}_{n,l_{k-1}}\cap x_{k-1} \neq \emptyset}}
\sum_{{1\leq l\neq l(x)\leq L^{\star}}\atop{C^{\star}_{n,l}\cap z\neq \emptyset}}\quad
\\
 & 
 q_n(x,x_1)\dots q_n(x_{k-1},y)
\1_{\{\forall_{x'_1\in C^{\star}_{n,l_1}\setminus\{x\}}\t_n(x'_1)< r_n(\varepsilon_n)\}}\dots
\\
 & 
 \dots
\1_{\{\forall_{x'_{k-1}\in C^{\star}_{n,l_{k-1}}\setminus\{x\}}\t_n(x'_{k-1})< r_n(\varepsilon_n)\}}
\1_{\{\exists_{z'\in C^{\star}_{n,l}}\t_n(z')\geq r_n(\varepsilon_n)\}}.
\end{split}
\Eq(3.prop3.16)
\ee
Since 
$
\1_{\{\forall_{z'\in C^{\star}_{n,l}\setminus\{x\}}\t_n(z')< r_n(\varepsilon_n)\}}
\1_{\{\exists_{z'\in C^{\star}_{n,l}}\t_n(z')\geq r_n(\varepsilon_n)\}}
=0
$ 
for all $l\neq l(x)$, the sums over $l$ in  \eqv(3.prop3.16) can be restricted to 
$1\leq l\neq l(x), l_1,\dots, l_{k-1}\leq L^{\star}$.
We may now multiply \eqv(3.prop3.16) by $\1_{\{x\in T^{\circ}_n\}}$
and take the conditional expectation. 
The variables $\{\t_n(z'), z'\in C^{\star}_{n,l}\}$ being independent of 
 the variables $\{\t_n(x'), x'\in\cup_{0\leq l'\neq l\leq L^{\star}}C^{\star}_{n,l'}\}$,
they can be integrated out first, yielding, for all $y\in \del V^{\star}_n$
\bea
&&
\sum_{z\sim y}\sum_{
{1\leq l\neq l(x), l_1,\dots, l_{k-1}\leq L^{\star}}\atop{C^{\star}_{n,l}\cap z \neq \emptyset}}
\P\left[
\exists_{z'\in C^{\star}_{n,l}}
\t_n(z')\geq r_n(\varepsilon_n)\,\big|\, V^{\star}_n\right]
\\
\hspace{-6pt}&\leq &\hspace{-6pt}
n\max_{
1\leq l\neq l(x), l_1,\dots, l_{k-1}\leq L^{\star}
}|C^{\star}_{n,l}|2^{-(\varepsilon_n-\rho_n^{\star})n}
\Eq(3.prop3.17)
\\
\hspace{-6pt}&\leq &\hspace{-6pt}
n^22^{-(\varepsilon_n-\rho_n^{\star})n},
\Eq(3.prop3.12)
\eea
where we used in \eqv(3.prop3.17)  that the sum over $l$ contains at most one term while   
the sum over $z$ contains at most $n$ terms. Eq.~\eqv(3.prop3.12)
 then follows from  \eqv(2.lem3.4) and so, is valid on $\O^{\star}$ for all large enough $n$.
This bound is uniform in $y\in \del V^{\star}_n$. Therefore, using \eqv(3.prop3.12) in \eqv(3.prop3.16)
and re-summing,  \eqv(3.prop3.11) becomes
\bea
\WW_{n,k}(x)
\hspace{-6pt}&\leq &\hspace{-6pt}
\frac{r^{\star}_n}{n}
n^22^{-(\varepsilon_n-\rho_n^{\star})n}
\E\Biggl[\E\biggl[
\1_{\{x\in T^{\circ}_n\}}
\sum_{y\in \del V^{\star}_n}q_n^{(k)}(x,y)
\,\bigg|\, V^{\star}_n \biggr]\Biggr]
\Eq(3.prop3.18)
\\
\hspace{-6pt}&\leq &\hspace{-6pt}
\frac{r^{\star}_n}{n}
n^22^{-(\varepsilon_n-\rho_n^{\star})n}\P(x\in T^{\circ}_n)
\Eq(3.prop3.19)
\eea
where we used in \eqv(3.prop3.19) that since $Q_{n}$  is sub-stochastic, 
$
\sum_{y\in  \del V^{\star}_n}q_n^{(k)}(x,y)\leq 1
$
for all $x$. Now, by \eqv(2.20) and \eqv(2.2), 
$
\P(x\in T^{\circ}_n)\leq \P(\t_n(x)\geq r_n(\varepsilon_n))=2^{-\varepsilon_nn}
$.
Thus
\be
\WW_{n,k}(x)
\leq
{r^{\star}_n}n2^{-2\varepsilon_nn}2^{\rho_n^{\star}n}
=n^{c_{\star}+1}{r^{\star}_n}2^{-2\varepsilon_nn}.
\Eq(3.prop3.13)
\ee
The last equality is \eqv(2.0).
Using this bound in  \eqv(3.prop3.8) finally yields that on $\O^{\star}$, for all large enough $n$,
\be
\WW_n
=
\sum_{x\in \VV_n}
\int_0^{\theta_n}dt\sum_{k=1}^{\infty}
\frac{(\nu_nt)^k}{k!}e^{-\nu_nt}
S_{n,k}(x)
\leq 
\theta_nn^{c_{\star}+1}{r^{\star}_n}2^n2^{-2\varepsilon_nn}.
\Eq(3.prop3.14)
\ee
It only remains to observe that 
by \eqv(2.lem3.5) and \eqv(3.lem0.2)  of Lemma \thv(3.lem0), on $\O^{\star}$, 
$\pi_n(T^{\circ}_n)=2^{-n\varepsilon_n}(1+o(1))$
for all but a finite number of indices $n$.
Hence
\be
\P\left[
\textstyle
|T^{\circ}_n|^{-1}\sum_{x\in T^{\circ}_n}P_{x}\left(H(T_n^{\star}\setminus T^{\star}_{n,l(x)})\leq s\right)\geq \e
\right]
\leq
\e^{-1}sn^{c_{\star}+1}{r^{\star}_n}\pi_n(T^{\circ}_n)(1+o(1)).
\nonumber
\ee
Choosing
$\e=n^{2}n^{c_{\star}+1}{r^{\star}_n}\pi_n(T^{\circ}_n)$,
the claim of the proposition follows from Borel-Cantelli Lemma.
\end{proof}

\begin{proof}[Proof of Proposition \thv(3.prop3bis)]  This is a rerun of the proof of Proposition \thv(3.prop3).\end{proof}

\subsection{Rough bounds on local times.}
  \TH(S3.4)

\begin{lemma}
   \TH(3.lem4)
For all $0\leq \a\leq 1$, all $x\in \VV_n$, and all $s>0$,
\bea
E_{x}\left[\ell_n^x(s)\right]^{\a} 
\geq
(\wt\l^{-1}_n(x))^{\a}\G(1+\a)[1-c_1\exp(-c_2s\wt\l_n(x))]+s^{\a}\exp(-s\wt\l_n(x))
\Eq(3.lem4.1)
\eea
where $0<c_1,c_2<\infty$ are constants, and if moreover 
$sr^{\star}_nn\pi_n(x)\rightarrow 0$ as $n\rightarrow\infty$,
\bea
E_{x}\left[\ell_n^x(s)\right]^{\a} 
\leq
(1+o(1))\left[\kappa_n^\a+\1_{\{s>\kappa_n\}}s^\a(s-\kappa_n)r^{\star}_nn\pi_n(x)\right].
\Eq(3.lem4.2)
\eea
\end{lemma}

\begin{proof}[Proof of Lemma \thv(3.lem3)] The lower bound follows from the trite 
observation that $\ell_n^x(s)$ is at least as large as the minimum between the first jump of $Y_n$
and $s$, that is,
\be
\ell_n^x(s)\geq \wt\l^{-1}_n(x)e_1\1_{s>\wt\l^{-1}_n(x)e_1}+s\1_{s\leq\wt\l^{-1}_n(x)e_1},
\Eq(3.lem4.3)
\ee
where $e_1$ is an exponential random variable of mean one. Thus
\bea
E_{x}\left[\ell_n^x(s)\right]^{\a} 
\geq
E_{x}\left[\wt\l^{-1}_n(x)e_1\1_{s>\wt\l^{-1}_n(x)e_1}\right]^{\a}
+s^{\a}E_{x}\left[\1_{s\leq \wt\l^{-1}_n(x)e_1}\right]^{\a}.
\eea
 Explicit calculations yield 
\be
E_{x}\left[\wt\l^{-1}_n(x)e_1\1_{s>\wt\l^{-1}_n(x)e_1}\right]^{\a}
\geq 
(\wt\l^{-1}_n(x))^{\a}\G(1+\a)[1-c_1\exp(-c_2s\wt\l_n(x))]
\ee
for some constants $0<c_1,c_2<\infty$.
Eq.~\eqv(3.lem4.1) now readily follows. To get an upper bound write
$
E_{x}\left[\ell_n^x(s)\right]^{\a}\leq\kappa^{\a}_n
$
if $s\leq\kappa_n$. Otherwise write
\bea
E_{x}\left[\ell_n^x(s)\right]^{\a} 
&\leq&
\textstyle
E_{x}\left[\kappa_n+ \int_{\kappa_n}^{s}\1_{\{Y_n(s)=x\}}ds\right]^{\a} 
\\
&\leq&
\textstyle
(1+\rho_n)E_{\pi_n}\left[\kappa_n+ \int_{0}^{s-\kappa_n}\1_{\{Y_n(s)=x\}}ds\right]^{\a} 
\eea
where the last line follows from Proposition \thv(3.prop2) and the Markov property. Next,
\be
\begin{split}
\textstyle
E_{\pi_n}\left[\kappa_n+ \int_{0}^{s-\kappa_n}\1_{\{Y_n(s)=x\}}ds\right]^{\a}  
& 
\leq E_{\pi_n}\left(\kappa_n^\a\1_{\{H(x)>s-\kappa_n\}}+s^\a\1_{\{H(x)\leq s-\kappa_n\}}\right)
\\
 &
\leq \kappa_n^\a+s^\a P_{\pi_n}(H(x)\leq s-\kappa_n)
\\
 &
\leq
\kappa_n^\a+s^\a(s-\kappa_n)r^{\star}_nn\pi_n(x)(1+o(1)),
\end{split}
\ee
the last inequality being \eqv(3.lem1.2) of Lemma  \thv(3.lem1). 
Eq.~\eqv(3.lem4.2) is proved.
\end{proof}



 
 \section{Verification of Condition (B1)} 
 \label{S4}
 
In this section we prove a strong law of large number for the function $\nu_n^{t}(u,\infty)$ 
defined in \eqv(B1.1). Recall that for $r^{\star}_n$ defined in \eqv(1.4.1), we take 
$\eta_n\equiv(r^{\star}_n)^{-1}$  in \eqv(1.03.12), \eqv(1.4.3), and \eqv(1.4.4). Then by 
\eqv(1.03.9)-\eqv(1.03.10), \eqv(1.3.6), and \eqv(1.03.12),
\be
\nu_n^{t}(u,\infty)
= k_n(t)P_{\pi_n}\left(\int_{0}^{\theta_n}\max\left((c_nr^{\star}_n)^{-1},c_n^{-1}\t_n(Y_n(s))\right)ds>u\right)
\Eq(4.1)
\ee
where  $\pi_n$ is the invariant measure \eqv(1.4.3) of $Y_n$, $\theta_n$ is the block length 
of the blocked clock process \eqv(1.03.9),  $k_n(t)=\lf a_n t/\theta_n\rf$, and, given $0<\varepsilon<1$, 
$c_n$ and $a_n$ are defined in \eqv(1.theo0.0)  and \eqv(1.theo1.0)-\eqv(1.theo1.3), respectively.
By Theorem \thv(1.theoB), $\theta_n$ and $a_n$ must obey
\be
\lfloor  n^4r^{\star}_n(1+o(1))\rfloor
\equiv\kappa_n\leq \theta_n\ll a_n,
\ee
where the left-most equality is \eqv(3.prop2.1). 
Further recall from Section  \thv(S2) that for $\rho_n^{\star}$  as in \eqv(2.0),
\be
\rho_n^{\star}\ll\varepsilon_n\equiv\varepsilon-\d_n.
\Eq(4.4)
\ee
(Recall that $0<x_n\ll y_n$ means that  $x_n/y_n\rightarrow 0$ as $n\rightarrow\infty$.)
From now on we take $\d_n$ such that $2^{n\d_n}=\left(n^2\theta_n\right)^{\a(\varepsilon)}$, i.e.
\be
\d_n\equiv  \frac{1}{n\b}\sqrt{\frac{2\varepsilon}{\log2}}\log\left(n^2 \theta_n\right).
\Eq(4.c2)
\ee
Thus, given $0<\varepsilon<1$ and $\b>0$, all sequences except $\theta_n$ are determined.

\begin{proposition}
     \TH(4.prop1)  

Given $0<\varepsilon<1$ and $\b>0$ let the sequences $c_n$ and $a_n$ be defined as in  
\eqv(1.theo0.0) and \eqv(1.theo1.0)-\eqv(1.theo1.3),  respectively,  and let $\theta_n$ be such that 
\bea
&(r^{\star}_n)^{4}\ll\theta_n^{1-\a(\varepsilon)},&
\Eq(4.c4)
\\
&
n^{-1}\log \theta_n\ll 1.&
\Eq(4.cnew)
\eea
Then, for all $0<\varepsilon<1$ and $\b>0$, $\P$-almost surely,
\be
\lim_{n\rightarrow\infty}\nu_n^t(u,\infty)=tu^{\a(\varepsilon)},\quad\forall\, t>0, u>0.
\Eq(4.prop1.1)
\ee
\end{proposition}

\begin{remark} Eq.~ \eqv(4.cnew) implies that $\d_n\ll 1$ and that $\theta_n\ll c_n$ for all $\varepsilon>0$. 
In view of \eqv(2.lem2.2),  \eqv(3.prop1.1), \eqv(4.c2) and \eqv(3.prop2.1), \eqv(4.cnew) also implies that 
\be
c_0n^{c_1}\tilde\kappa_n^{c_2}\kappa_n^{c_3}(r^{\star}_n)^{c_4}\theta_n^{c_5} \ll  2^{\varepsilon n}
\,\,\,\,\text{and}\,\,\,\,
c_0n^{c_1}\tilde\kappa_n^{c_2}\kappa_n^{c_3}(r^{\star}_n)^{c_4}\theta_n^{c_5} \ll  2^{\varepsilon_n n}
\Eq(4.c1'new)
\ee
for all  $\varepsilon>0$ and any choice of constants $0\leq c_i<\infty$.

\end{remark}
\begin{remark} 
In order to guarantee strict equivalence of the definitions \eqv(1.theo1.4)  and \eqv(2.7)  of the set $T_n$ 
when  $\d_n$ is given by \eqv(4.c2), we should replace the term $c_n (n \theta_n)^{-1}$ in \eqv(1.theo1.4) by
\be
c_n \exp\left\{-\log(n^2 \theta_n)\left[1+(1+o(1))(2n\b\b_c(\varepsilon))^{-1}\log(n^2 \theta_n)
\right]\right\}
\ee
(see Corollary \thv(2.cor1)). We didn't state this precise formula to keep the presentation simple.
\end{remark}

The rest of the section is organized as follows. In Section \thv(S4.1)  we show that 
$\nu_n^t(u,\infty)$ can be reduced to the quantity $\nu_n^{\circ, t}(u,\infty)$ defined 
in \eqv(4.19). In Section \thv(S4.3) we prove upper et lower bounds on a sequence, 
$b^{\circ}_n$, defined  as $b_n$ with $T^{\circ}_n$ substituted for $T_n$,
and show that $b_n$ and $b^{\circ}_n$ behave in the same way to leading order. 
In Section \thv(S4.2) we show that $\nu_n^{\circ, t}(u,\infty)$ concentrates around 
its mean value when choosing $a_n=2^{\varepsilon n}/b^{\circ}_n$.
The proof of Proposition \thv(4.prop1) is finally completed in Section \thv(S4.3).

\subsection{Preparations.} 
 \TH(S4.1)

To begin with, we bring the function $\nu_n^{t}(u,\infty)$ given in \eqv(4.1) into a form amenable to treatment.
Let $T_n$ be as in \eqv(2.7). For all $0<\varepsilon<1$ and $\d_n$ as in \eqv(4.c2),
\be
\begin{split}
0\leq
\int_{0}^{\theta_n}
\max\left((c_nr^{\star}_n)^{-1},c_n^{-1}\t_n(Y_n(s))\right)
\1_{\{Y_n(s)\notin T_n\}}ds
\leq & 
\theta_n \frac{r_n(\varepsilon_n)}{r_n(\varepsilon)}\leq   n^{-2}
\end{split}
\Eq(4.0)
\ee
as follows from \eqv(2.cor1.1).
Hence visits of $Y_n$ outside the set $T_n$ only yield a negligible contribution to the event in \eqv(4.1), implying that
\be
\check{\nu}_n^t(u,\infty)\leq \nu_n^{t}(u,\infty)\leq \check{\nu}_n^t\left(u-n^{-2},\infty\right)
\Eq(4.5)
\ee
where
\be
\check{\nu}_n^t(u,\infty)
\equiv k_n(t)P_{\pi_n}\left(\int_{0}^{\theta_n}c_n^{-1}\t_n(Y_n(s))\1_{\{Y_n(s)\in T_n\}}ds>u\right).
\Eq(4.6)
\ee
Our next step consists in reducing visits to $T_n$ in $\check{\nu}_n^t(u,\infty)$ to visits to 
the subset $T^{\circ}_n$ defined in \eqv(2.20). Set
\be
\bar\nu_n^t(u,\infty)
\equiv k_n(t)P_{\pi_n}\left(\int_{0}^{\theta_n}c_n^{-1}\t_n(Y_n(s))\1_{\{Y_n(s)\in T^{\circ}_n\}}ds>u\right).
\Eq(4.6')
\ee

\begin{lemma}
    \TH(4.lem1)  
Assume that  \eqv(4.cnew) holds. Then on $\O^{\star}$, for all but a finite number of indices $n$,
\be
|\check{\nu}_n^t(u,\infty)-\bar\nu_n^t(u,\infty)|
\leq 
2 k_n(t)\theta_n r^{\star}_nn^52^{-2n\varepsilon_n}(1+o(1)).
\Eq(4.lem1.1)
\ee
\end{lemma}

\begin{proof}[Proof of Lemma \thv(4.lem1)] Decomposing the event appearing in the probability in \eqv(4.6) 
according to whether $\{H(T_n\setminus T^{\circ}_n)\leq \theta_n\}$ or $\{H(T_n\setminus T^{\circ}_n)> \theta_n\}$, 
\eqv(4.lem1.1) follows from \eqv(3.cor1.1) of Corollary \thv(3.cor1) applied with $t_n=\theta_n$,
 which is licit by virtue of  \eqv(4.cnew) (see also \eqv(4.c1'new)).
\end{proof}

\noindent We next decompose \eqv(4.6') according to the hitting time,
$
H(T^{\circ}_n)
$,
and hitting place,
$
Y_n(H(T^{\circ}_n))
$,
of the set $T^{\circ}_n$. The density of the joint distribution of $H(T^{\circ}_n)$ 
and $Y_n(H(T^{\circ}_n))$ is a $|T^{\circ}_n|$-dimensional vector, 
$
(h_{n, x})_{x\in T^{\circ}_n}
$, 
whose components are, for each $x\in T^{\circ}_n$, the joint density 
that $T^{\circ}_n$ is reached at time $v$, and that arrival to that set occurs in state $x$,
\be
P_{\pi_n}(H(T^{\circ}_n)\leq s, Y_n(H(T^{\circ}_n))=x)=\int_0^sh_{n, x}(v)dv.
\Eq(4.7)
\ee
For this vector of densities we have
\be
\sum_{x\in T^{\circ}_n}\int_0^\infty h_{n, x}(v)dv=1,
\Eq(4.8)
\ee
and, denoting by $h_{n, T^{\circ}_n}$ the density of $H(T^{\circ}_n)$,
\be
h_{n, T^{\circ}_n}=\sum_{x\in T^{\circ}_n}h_{n, x}.
\Eq(4.9)
\ee
In the notation of Section \thv(S3.3) (see the paragraph below \eqv(3.prop3.3)) 
$h_{n, x}=\sum_{y\in\VV_n}\pi_n(y)h^y_{n, x, T^{\circ}_n}$ where, for $y\in T^{\circ}_n$, 
$h^y_{n, x, T^{\circ}_n}=\d_y$. From this and the strong Markov property it follows that
\be
\bar\nu_n^t(u,\infty)
= k_n(t)
\sum_{x\in T^{\circ}_n}\int_0^{\theta_n}h_{n, x}(v)
P_{x}\left(\int_{0}^{\theta_n-v}c_n^{-1}\t_n(Y_n(s))\1_{\{Y_n(s)\in T^{\circ}_n\}}ds>u\right)dv.
\Eq(4.10)
\ee
Denote by $\overline Q^{u, v}_n(x)$ the probability appearing in \eqv(4.10).
Notice that $Y_n$ starts in $x\in T^{\circ}_n$ and further decompose this probability according to whether
$
\{H(T^{\circ}_n\setminus x)\leq \theta_n-v\}
$
or 
$
\{H(T^{\circ}_n\setminus x)> \theta_n-v\}
$,
that is, write
$
\overline Q^{u, v}_n(x)\equiv\wt Q^{u, v}_n(x)+\wh Q^{u, v}_n(x)
$,
\be
\wt Q^{u, v}_n(x)
=P_{x}\left(\int_{0}^{\theta_n-v}c_n^{-1}\t_n(Y_n(s))\1_{\{Y_n(s)\in T^{\circ}_n\}}ds>u, H(T^{\circ}_n\setminus x)\leq \theta_n-v\right),
\Eq(4.11)
\ee
\be
\wh Q^{u, v}_n(x)
=P_{x}\left(\int_{0}^{\theta_n-v}c_n^{-1}\t_n(Y_n(s))\1_{\{Y_n(s)\in T^{\circ}_n\}}ds>u, H(T^{\circ}_n\setminus x)> \theta_n-v\right),
\Eq(4.12)
\ee
and split \eqv(4.10) accordingly.
Clearly, for all $v>0$
\be
\wt Q^{u, v}_n(x)\leq P_{x}\left(H(T^{\circ}_n\setminus x)\leq \theta_n\right).
\Eq(4.13)
\ee
This and the bound
$
\int_0^{\theta_n}h_{n, x}(v)dv
\leq P_{\pi_n}(H(x)\leq \theta_n)
$
(that follows from \eqv(4.7)), yield
\bea
&&k_n(t) \sum_{x\in I^{\circ}_n}\int_0^{\theta_n}h_{n, x}(v)\wt Q^{u, v}_n(x)dv
\Eq(4.14)\\
\hspace{-6pt}&\leq&\hspace{-6pt} 
k_n(t) \sum_{x\in T^{\circ}_n}P_{\pi_n}(H(T^{\circ}_n)\leq \theta_n, Y_n(H(T^{\circ}_n))=x)P_{x}\left(H(T^{\circ}_n\setminus x)\leq \theta_n\right)
\Eq(4.15)
\\
\hspace{-6pt}&\leq&\hspace{-6pt} 
 \wt\nu_n^t
 \Eq(4.15')
\eea
where
\be
\wt\nu_n^t\equiv
k_n(t) \sum_{x\in T^{\circ}_n}P_{\pi_n}(H(x)\leq \theta_n)P_{x}\left(H(T^{\circ}_n\setminus x)\leq \theta_n\right).
\Eq(4.lem2.1)
\ee

\begin{lemma}
     \TH(4.lem2)
Assume that  \eqv(4.cnew) holds. Then on $\O^{\star}$, for all but a finite number of indices $n$,
\be
\wt\nu_n^t 
\leq k_n(t)n^{c_{\star}+4}\left(\theta_n \pi_n(T^{\circ}_n)r^{\star}_n\right)^2(1+o(1)).
\Eq(4.lem2.2)
\ee
\end{lemma}

\begin{proof}
[Proof of Lemma \thv(4.lem2)]  
By \eqv(3.lem0.2), \eqv(2.lem3.5), \eqv(4.4) and \eqv(4.c2),  on $\O^{\star}$, for all large enough $n$,
$\theta_n n\pi_n(T^{\circ}_n)r^{\star}_n= n^{1+2\a(\varepsilon)}r^{\star}_n\theta^{1+\a(\varepsilon)}_n 2^{-n\varepsilon}(1+o(1))$, wich decays to zero as $n$ diverges by \eqv(4.cnew) (see also \eqv(4.c1'new)).
We may thus use \eqv(3.lem1.2) of Lemma \thv(3.lem1) to bound the term $P_{\pi_n}(H(x)\leq \theta_n)$
in \eqv(4.lem2.1), and by this and  \eqv(3.lem0.2) we get that on $\O^{\star}$, for all large enough $n$, 
\be
\wt\nu_n^t
\textstyle
\leq k_n(t)\theta_n n\pi_n(T^{\circ}_n)r^{\star}_n(1+o(1))
|T^{\circ}_n|^{-1}\sum_{x\in T^{\circ}_n}P_{x}\left(H(T^{\circ}_n\setminus x)\leq \theta_n\right).
\Eq(4.lem2.3)
\ee
The lemma now follows from Proposition \thv(3.prop3).
\end{proof}

Consider now the contribution to \eqv(4.10) coming from \eqv(4.12). By definition,
\be
\wh Q^{u, v}_n(x)
=P_{x}\left(c_n^{-1}\t_n(x)\ell_n^x(\theta_n-v)>u, H(T^{\circ}_n\setminus x)> \theta_n-v\right).
\Eq(4.16)
\ee
Thus
\bea
&&\wh\nu_n^t(u,\infty)\\
\hspace{-6pt}&\equiv&\hspace{-6pt}  k_n(t) \sum_{x\in T^{\circ}_n}\int_0^{\theta_n}h_{n, x}(v)\wh Q^{u, v}_n(x)dv
\Eq(4.17)
\\
\hspace{-6pt}&=&\hspace{-6pt} 
k_n(t) \sum_{x\in T^{\circ}_n}
\int_0^{\theta_n}h_{n, x}(v)P_{x}\left(c_n^{-1}\t_n(x)\ell_n^x(\theta_n-v)>u, H(T^{\circ}_n\setminus x)> \theta_n-v\right)dv.\quad\quad
\Eq(4.18)
\eea
Setting
\be
\nu_n^{\circ, t}(u,\infty)
\equiv
k_n(t) \sum_{x\in T^{\circ}_n}
\int_0^{\theta_n}h_{n, x}(v)P_{x}\left(c_n^{-1}\t_n(x)\ell_n^x(\theta_n-v)>u\right)dv,
\Eq(4.19)
\ee
we have
\be
\nu_n^{\circ, t}(u,\infty)-w_n^t(u,\infty)\leq \wh\nu_n^t(u,\infty)\leq \nu_n^{\circ, t}(u,\infty)
\Eq(4.20)
\ee
where
\bea
w_n^t(u,\infty)
\hspace{-6pt}&\equiv&\hspace{-6pt}
k_n(t) \sum_{x\in T^{\circ}_n}
\int_0^{\theta_n}h_{n, x}(v)P_{x}\left(c_n^{-1}\t_n(x)\ell_n^x(\theta_n-v)>u, H(T^{\circ}_n\setminus x)\leq \theta_n-v\right)dv
\nonumber
\\
\hspace{-6pt}&\leq&\hspace{-6pt}
k_n(t) \sum_{x\in T^{\circ}_n}\int_0^{\theta_n}h_{n, x}(v)P_{x}\left(H(T^{\circ}_n\setminus x)\leq \theta_n-v\right)dv
\leq \wt\nu_n^t.
\Eq(4.22)
\eea
Inserting our bounds in \eqv(4.10), we finally get that for all $u>0$
\be
\left|\nu_n^{\circ, t}(u,\infty)-\bar\nu_n^t(u,\infty)\right|\leq  \wt\nu_n^t.
\Eq(4.24)
\ee

Our aim now is to prove almost sure convergence of $\nu_n^{\circ, t}(u,\infty)$. To do so we will need certain properties 
a sequence, $b^{\circ}_n$, associated  to the sequence $b_n$, that we now define.


\subsection{Properties of the sequences $b_n$ and $b^{\circ}_n$.} 
 \TH(S4.3)

For $F_{\b,\varepsilon,n}(x)$ as in \eqv(4.prop2.16) define
\be
b^{\circ}_n
\equiv(\theta_n\pi_n(T^{\circ}_n))^{-1}\sum_{x\in T^{\circ}_n}\int_0^{\theta_n}h_{n, x}(v)E_{x}[F_{\b,\varepsilon,n}(\ell_n^x(\theta_n-v))]dv.
\Eq(4.prop2.18)
\ee
Thus $b^{\circ}_n$ is nothing but $b_n$ (see  \eqv(1.theo1.3)) with $T^{\circ}_n$ substituted for $T_n$.
The next lemma collects properties of the sequences $b_n$ and $b^{\circ}_n$
needed in the verification of both Condition (B1) and (B2).
Set 
$
\II_n(a,b)=(\theta_n\pi_n(T^{\circ}_n))^{-1}\sum_{x\in T^{\circ}_n}\JJ^x_n(a,b)
$,
\be
\JJ^x_n(a,b)=\int_{a}^{b}h_{n, x}(v)E_{x}[F_{\b,\varepsilon,n}(\ell_n^x(\theta_n-v))]dv,
\Eq(4.lem3.1)
\ee
and given $0<\zeta_n<\theta_n$ split
$
b^{\circ}_n
$
into
$
b^{\circ}_n=\II_n(0,\kappa_n)+\II_n(\kappa_n, \theta_n-\zeta_n)+\II_n(\theta_n-\zeta_n, \theta_n)
$.

\begin{lemma}
     \TH(4.lem3)
Assume that \eqv(4.c4) and \eqv(4.cnew)  hold. Let $\zeta_n>0$ be a sequence satisfying 
\be
n^{-1}|\log\zeta_n|\ll 1,\,\,\,
\text{and}\,\,\,
\tilde\kappa_n(r^{\star}_n)^{1+\a_{n}(\varepsilon)+o(1)}\zeta_n^{\a_{n}(\varepsilon)+o(1)}\downarrow 0\,\,\,\text{as}\,\,\, n\uparrow\infty.
\Eq(4.lem3.2)
\ee
Then, on $\O_1\cap\O^{\circ}\cap \O^{\star}$, for all but a finite number of indices $n$,
\be
\frac{\II_n(0,\kappa_n)}{\II_n(\kappa_n, \theta_n-\zeta_n)}\leq
\theta_n^{-1}\tilde\kappa_n\kappa_n^{1+\a_{n}(\varepsilon)}
(nr^{\star}_n)^{1+\a_{n}(\varepsilon)+o(1)},
\Eq(4.lem3.3)
\ee
\be
0\leq (b_n-b^{\circ}_n)/b^{\circ}_n
\leq
n(r^{\star}_n)^{1+\a_{n}(\varepsilon)+o(1)}\kappa_n^{1+\a_{n}(\varepsilon)}2^{-n\varepsilon_n},
\Eq(4.lem3.6)
\ee
and the right-hand sides  of \eqv(4.lem3.3) and \eqv(4.lem3.6) decay to zero as $n$ diverges.  Furthermore
\be
\kappa_n^{-1}(r^{\star}_n)^{-\{\a_{n}(\varepsilon)+o(1)\}}
\leq 
b^{\circ}_n
\leq (1+o(1))
nr^{\star}_n\kappa_n^{\a_{n}(\varepsilon)}.
\Eq(4.lem3.6')
\ee
\end{lemma}

\begin{proof}[Proof of Lemma \thv(4.lem3)]  
We first prove a lower bound on $\II_n(\kappa_n, \theta_n-\zeta_n)$. For this write
\be
\JJ^x_n(\kappa_n, \theta_n-\zeta_n)\geq \JJ^x_{n,1}
\equiv
\int_{\kappa_n}^{\theta_n-\zeta_n}h_{n, x}(v)E_{x}[F_{\b,\varepsilon,n}(\ell_n^x(\theta_n-v))
\1_{\{\zeta_n<\ell_n^x(\theta_n-v)\leq \theta_n\}}]dv.
\nonumber 
\ee
Since
$
F_{\b,\varepsilon,n}(x)=(1+o(1))x^{\a_{n}(\varepsilon)+o(1)}
$
for all $\zeta_n<x\leq \theta_n$, 
\be
\begin{split}
\JJ^x_{n,1} & \geq
(1+o(1))\int_{\kappa_n}^{\theta_n-\zeta_n}h_{n, x}(v)
E_{x}[\ell_n^x(\theta_n-v)]^{\a_{n}(\varepsilon)+o(1)}(1-\1_{\{\ell_n^x(\theta_n-v)<\zeta_n\}})dv
\\
 &\equiv\JJ^x_{n,3}-\JJ^x_{n,4}
\end{split}
\Eq(4.lem3.13)
\ee
where we used the left-most inequality in \eqv(4.prop2.14) to relax the constraint $\ell_n^x(\theta_n-v)\leq \theta_n$. 
Let us bound $\JJ^x_{n,3}$ for $x\in I^{\star}_n$. Note that by  \eqv(2.22) and \eqv(1.4.6)
\be
(r^{\star}_n)^{-1}\leq \wt\l_n(x)\leq r^{\star}_n, \quad \forall x\in I^{\star}_n.
\Eq(4.lem3.14')
\ee
Thus, setting $\zeta'_n\equiv nr^{\star}_n$, it follows from \eqv(3.lem4.1) of Lemma \thv(3.lem4) that  for all
 $x\in I^{\star}_n$, 
\be
\JJ^x_{n,3}
\geq 
c_3(\wt\l^{-1}_n(x))^{\a_{n}(\varepsilon)+o(1)}
\int_{\kappa_n}^{\theta_n-\zeta'_n
}h_{n, x}(v)dv
\Eq(4.lem3.14)
\ee
for some numerical constant $0<c_3<\infty$.  Summing over $x$, wet get
\be
\sum_{x\in T^{\circ}_n}\JJ^x_{n,3}
\geq
 \sum_{x\in I^{\star}_n}\JJ^x_{n,3}
\geq 
c_3(r^{\star}_n)^{-\{\a_{n}(\varepsilon)+o(1)\}}
 \sum_{x\in I^{\star}_n}\int_{\kappa_n}^{\theta_n-\zeta'_n}h_{n, x}(v)dv
 \Eq(4.lem3.22)
\ee
where the last sum in the right-hand side of \eqv(4.lem3.22) is equal to
\be
P_{\pi_n}(\kappa_n<H(I^{\star}_n)<\theta_n-\zeta'_n, H(I^{\star}_n)<H(T^{\circ}_n\setminus I^{\star}_n)).
\ee
Decomposing this probability into
\be
p_1-p_2\equiv 
P_{\pi_n}(\kappa_n<H(I^{\star}_n)<\theta_n-\zeta'_n)
-P_{\pi_n}(\kappa_n<H(I^{\star}_n)<\theta_n-\zeta'_n, H(I^{\star}_n)>H(T^{\circ}_n\setminus I^{\star}_n))\quad\,\,
\nonumber
\ee
we have, by Lemma \thv(3.lem3) and  \eqv(3.lem2.1), whenever 
$
\theta_nr^{\star}_nn\pi_n(I^{\star}_n)\rightarrow 0
$,
\be
p_1
\geq 
\tilde\kappa_n^{-1}\theta_n\pi_n(I^{\star}_n)
(1-\theta_n^{-1}\zeta'_n)(1+o(1))
=
\tilde\kappa_n^{-1}\theta_n\pi_n(I^{\star}_n)(1+o(1))
\Eq(4.lem3.24)
\ee
where the last equality follows from   \eqv(4.c4).
To get an upper bound on $p_2$, write
\be
\begin{split}
p_2\leq & 
P_{\pi_n}(H(T^{\circ}_n\setminus I^{\star}_n)<\kappa_n)+
P_{\pi_n}(H(T^{\circ}_n\setminus I^{\star}_n)<H(I^{\star}_n)<\theta_n)
\equiv p_3+p_4.
\end{split}
\Eq(4.lem3.25)
\ee
By \eqv(3.lem1.2),
$
p_3\leq \kappa_nr^{\star}_nn\pi_n(T^{\circ}_n\setminus I^{\star}_n)(1+o(1))
$,
whereas proceeding as in \eqv(4.14)-\eqv(4.lem2.1),
\bea
p_4
\hspace{-6pt}&\leq&\hspace{-6pt} 
\sum_{x\in T^{\circ}_n\setminus I^{\star}_n}P_{\pi_n}(H(x)\leq \theta_n)P_{x}\left(H(I^{\star}_n)\leq \theta_n\right)\quad\quad
\Eq(4.lem3.26)
\\
\hspace{-6pt}&=&\hspace{-6pt} 
n^3 (\theta_nr^{\star}_n)^2
\pi_n(T^{\circ}_n\setminus I^{\star}_n)\pi_n(I^{\star}_n)(1+o(1))
\Eq(4.lem3.27)
\eea
where the last equality follows from \eqv(3.lem1.2) and \eqv(3.prop3.1bis).
By \eqv(2.lem3.7), \eqv(2.lem3.8), and \eqv(3.lem0.2), on $\O^{\star}$ and for large enough $n$,
$
\pi_n(I^{\star}_n)
=2^{-n\varepsilon_n}(1-n^{-c_{\star}}(1+o(1)))
$
and 
$
\pi_n(T^{\circ}_n\setminus I^{\star}_n)
=n^{-c_{\star}+1}2^{-n\varepsilon_n}(1+o(1))
$
(thus in particular, $\pi_n(I^{\star}_n)/\pi_n(T^{\circ}_n)=1+o(1)$).
In view of this, \eqv(4.c4), and \eqv(4.cnew), one checks that 
$
\theta_nr^{\star}_nn\pi_n(I^{\star}_n)\rightarrow 0
$
(as requested above \eqv(4.lem3.24)) and that
$
p_2
= o(p_1)
$.
Thus
$
p_1-p_2=p_1(1+o(1))
$
and by this, \eqv(4.lem3.24), and \eqv(4.lem3.22),
\be
(\theta_n\pi_n(T^{\circ}_n))^{-1}\sum_{x\in T^{\circ}_n}\JJ^x_{n,3}
\geq 
\tilde\kappa_n^{-1}(r^{\star}_n)^{-\{\a_{n}(\varepsilon)+o(1)\}}(1+o(1)).
\Eq(4.lem3.15)
\ee

Turning to $\JJ^x_{n,4}$ we have
\be
\sum_{x\in T^{\circ}_n}\JJ^x_{n,4}\leq
(1+o(1))\zeta_n^{\a_{n}(\varepsilon)+o(1)}\sum_{x\in T^{\circ}_n}\int_{\kappa_n}^{\theta_n-\zeta_n}h_{n, x}(v)dv,
\ee
where the last sum is equal to $P_{\pi_n}(\kappa_n<H(T^{\circ}_n)<\theta_n-\zeta_n)$.
Since by Lemma \thv(3.lem3) and
\eqv(3.lem2.1),
$
P_{\pi_n}(\kappa_n<H(T^{\circ}_n)<\theta_n-\zeta_n)\leq(1+o(1))r^{\star}_nn\theta_n\pi_n(T^{\circ}_n),
$
we get
\be
(\theta_n\pi_n(T^{\circ}_n))^{-1}\sum_{x\in T^{\circ}_n}\JJ^x_{n,4}
\Eq(4.lem3.16)
\leq
(1+o(1))nr^{\star}_n\zeta_n^{\a_{n}(\varepsilon)+o(1)}.
\ee
At this point we may observe that the right-most condition in \eqv(4.lem3.2) is tailored to guarantee that 
$
\sum_{x\in T^{\circ}_n}\JJ^x_{n,3}\gg \sum_{x\in T^{\circ}_n}\JJ^x_{n,4}
$.
Hence, collecting our bounds,
\bea
\II_n(\kappa_n, \theta_n-\zeta_n)
\hspace{-7pt}&=&\hspace{-7pt} 
\frac{1+o(1)}{\theta_n\pi_n(T^{\circ}_n)}\sum_{x\in T^{\circ}_n}\int_{\kappa_n}^{\theta_n-\zeta_n}h_{n, x}(v)
E_{x}[\ell_n^x(\theta_n-v)]^{\a_{n}(\varepsilon)+o(1)}\quad\quad
\Eq(4.lem3.17)
\\
\hspace{-6pt}&\geq&\hspace{-6pt} 
\tilde\kappa_n^{-1}(r^{\star}_n)^{-\{\a_{n}(\varepsilon)+o(1)\}}.
\Eq(4.lem3.18)
\eea

We now prove an upper bound on $\II_n(0,\kappa_n)$. Using that
$
F_{\b,\varepsilon,n}(x)\leq (1+o(1))x^{\a_{n}(\varepsilon)}
$
for all $0<x\leq \theta_n$ together with \eqv(3.lem4.2) of Lemma \thv(3.lem4)
(which by \eqv(4.cnew) and \eqv(3.lem0.3)  is licit),
\be
\JJ^x_n(0,\kappa_n)
\leq (1+o(1))\kappa_n^{\a_{n}(\varepsilon)}
\int_{0}^{\kappa_n}h_{n, x}(v)dv.
\ee
Summing over $x\in T^{\circ}_n$ and using \eqv(3.cor1.1') and \eqv(4.cnew)
 to bound the resulting probability,
\be
\II^x_n(0,\kappa_n)
\leq (1+o(1))
nr^{\star}_n\theta_n^{-1}\kappa_n^{1+\a_{n}(\varepsilon)}.
\Eq(4.lem3.21)
\ee
One proves in the same way that
\be
\II^x_n(0,\theta_n)
\leq (1+o(1))
nr^{\star}_n\kappa_n^{\a_{n}(\varepsilon)}\left[1+\theta_n^{1+\a_{n}(\varepsilon)}r^{\star}_nn\kappa_n^{-\a_{n}(\varepsilon)}2^{-n}\right],
\Eq(4.lem3.21bis)
\ee
where by \eqv(4.cnew)  the term in square brackets (that comes from  \eqv(3.lem4.2)) is equal to $1+o(1)$. 

Combining \eqv(4.lem3.21) and \eqv(4.lem3.18)  proves \eqv(4.lem3.3). 
Since $\II_n(\kappa_n, \theta_n-\zeta_n)\leq b^{\circ}_n= \II_n(0, \theta_n)$, \eqv(4.lem3.18) and 
\eqv(4.lem3.21bis) yield, respectively, the lower and upper bounds of \eqv(4.lem3.6'). 
It remains to prove \eqv(4.lem3.6). By definition (see \eqv(1.theo1.3), \eqv(4.prop2.18), 
and the second remark below \eqv(4.prop1.1) on the definition of $T_n$)
\be
|T_n|b_n-|T^{\circ}_n|b^{\circ}_n
=2^n\theta_n^{-1}\sum_{x\in T_n\setminus T^{\circ}_n}E_{\pi_n}\left[F_{\b,\varepsilon,n}(\ell_n^x(\theta_n))\right].
\ee
Conditioning on the time of the first visit to $x$, and proceeding as in \eqv(4.lem3.21)-\eqv(4.lem3.21bis) to bound the expectation starting in $x$, 
$
E_{\pi_n}\left[F_{\b,\varepsilon,n}(\ell_n^x(\theta_n))\right]
\leq
 (1+o(1))P_{\pi_n}(H(x)\leq \theta_n) \kappa_n^{\a_{n}(\varepsilon)}
$.
From this and \eqv(3.lem1.2),
$
|T_n|b_n-|T^{\circ}_n|b^{\circ}_n
\leq 
(1+o(1))r^{\star}_nn2^n\pi_n(T_n\setminus T^{\circ}_n)\kappa_n^{\a_{n}(\varepsilon)}
$.
Now by \eqv(2.lem3.1bis)-\eqv(2.lem3.6), $|T_n|=|T^{\circ}_n|(1+o(1))$ and 
$\left|T_n\setminus T^{\circ}_n\right|=|T^{\circ}_n|n^42^{-n\varepsilon_n}(1+o(1))$.
Hence
$
b_n-b^{\circ}_n\leq (1+o(1))n^5r^{\star}_n\kappa_n^{\a_{n}(\varepsilon)}2^{-n\varepsilon_n}
$.
Combining this and \eqv(4.lem3.18) yields  \eqv(4.lem3.6).
 The proof of  Lemma \thv(4.lem3) is now complete.
 \end{proof}
 
 \begin{proof}[Proof of Proposition \thv(1.prop2)] This is a straightforward consequence 
of  \eqv(4.lem3.6), \eqv(4.lem3.6'), the assumptions of \eqv(1.theo1.0'), and \eqv(2.lem2.2).
\end{proof}


\subsection{Concentration of $\nu_n^{\circ, t}(u,\infty)$.} 
 \TH(S4.2)

Let us now focus on the term  $\nu_n^{\circ, t}(u,\infty)$ of \eqv(4.19).
Recall the definitions of $k_n(t)$ and $b^{\circ}_n$ from \eqv(1.03.8) and  \eqv(4.prop2.18), respectively.

\begin{proposition}
     \TH(4.prop2)  
Choose $a_n=2^{\varepsilon n}/b^{\circ}_n$ in $k_n(t)$ and assume that \eqv(4.cnew) holds.
Let $\P^{\circ}$  denote the law of the collection $\{\t_n(x), x\in T^{\circ}_n\}$
conditional on $T^{\circ}_n$,
\be
\P^{\circ}(\cap_{x\in T^{\circ}_n}\{\t_n(x)\in \cdot\})
=
\P(\cap_{x\in T^{\circ}_n}\{\t_n(x)\in \cdot\}\mid T^{\circ}_n).
\Eq(4.prop2.20)
\ee
Then, 
for any sequence $u_n>0$ such that $0<u-u_n<n^{-1}$ and all $u>0$ and $t>0$,
\be
\P^{\circ}\left(
\left|
\nu_n^{\circ, t}(u_n,\infty)-\E^{\circ}\nu_n^{\circ, t}(u_n,\infty)
\right|>
n\sqrt{t\Xi_n \E^{\circ}\nu_n^{\circ, t}(u_n,\infty)}
\right)
\leq n^{-2}(1+o(1))
\Eq(4.prop2.10')
\ee
where
$
\Xi_n\equiv (2^{\varepsilon n}/b^{\circ}_n)nr^{\star}_n{2^{-n}}
$
and
\be
\lim_{n\rightarrow\infty}\E^{\circ}\nu_n^{\circ, t}(u_n,\infty)
=tu^{\a(\varepsilon)}.
\Eq(4.prop2.2)
\ee  
\end{proposition}

\begin{proof}[Proof of Proposition \thv(4.prop2)]  We assume throughout that $\o\in\O^{\star}$.
A key ingredient of the proof is the observation that the generator $\wt L_n$ of $Y_n$ 
is independent of the values 
of the Hamiltonian at its local minima. More precisely,  recalling the definition of the set, $\MM_n$,  
of local minima from \eqv(2.lem5.1), it follows from  \eqv(1.4.6) and \eqv(2.lem5.2) that 
on $\O^{\star}$, for all  $n$ large enough, for all $x\in\MM_n$, and $y\sim x$,  
\be
\wt\l_n(x,y)=n^{-1}\t_n(y)\,\,\text{ and }\,\,
\wt\l_n(y,x)=
\begin{cases}
n^{-1}\t_n(y),
&\hbox{\rm if } 
y\notin\overline V^{\star}_n,\\
n^{-1},
&\hbox{\rm if } 
y\in\overline V^{\star}_n,\\
\end{cases}
\Eq(4.prop2.21)
\ee
 (note that if $x\in\MM_n$ and $y\sim x$ then $y\notin\MM_n$). Hence the law of $Y_n$ does not depend on the $\t_n(x)$'s in $\MM_n$ (but it does depend on $\MM_n$).
Now by \eqv(2.21),
\be
T^{\circ}_n\subseteq \MM_n\cap T_n\subseteq  \MM_n.
\Eq(4.prop2.3)
\ee
Furthermore, one easily checks that $\P^{\circ}$ in \eqv(4.prop2.20) is the product measure
\be
\P^{\circ}(\cap_{x\in T^{\circ}_n}\{\t_n(x)\in \cdot\})
=\prod_{x\in T^{\circ}_n}
\frac{\P(\t_n(x)\in \cdot,  \t_n(x)\geq r_n(\varepsilon_n))}{\P(\t_n(x)\geq r_n(\varepsilon_n))}.
\Eq(4.prop2.4)
\ee
Consequently,  for fixed $T^{\circ}_n$, the collection
$
\{X_n(x), x\in T^{\circ}_n\}
$,
\be
X_n(x)\equiv\int_0^{\theta_n}h_{n, x}(v)P_{x}\left(c_n^{-1}\t_n(x)\ell_n^x(\theta_n-v)>u_n\right)dv,
\Eq(4.prop2.5)
\ee
viewed as a collection of r.v.'s on the sub-sigma field 
$\FF^{\circ}_n=\s(\{\t_n(x), x\in T^{\circ}_n\})$,
forms a collection of independent random variables under $\P^{\circ}$ 
(that of course still depend on the variables $\t_n(x)$ in $(T^{\circ}_n)^c$). 
The proof now hinges on a simple mean and variance argument. 
We deal with the variance first. By \eqv(4.19) and \eqv(4.prop2.5),
\be
\E^{\circ}\nu_n^{\circ, t}(u_n,\infty)=k_n(t) \sum_{x\in T^{\circ}_n}\E^{\circ}X_n(x),
\Eq(4.prop2.5')
\ee
and by independence 
\be
\E^{\circ}(\nu_n^{\circ, t}(u_n,\infty)-\E^{\circ}\nu_n^{\circ, t}(u_n,\infty))^2
\leq
k^2_n(t)\sum_{x\in T^{\circ}_n}\E^{\circ}(X_n(x))^2.
\Eq(4.prop2.6')
\ee
Note that since
\be
X_n(x)
\leq\int_0^{\theta_n}h_{n, x}(v)dv
\leq P_{\pi_n}(H(x)\leq \theta_n)
\leq \theta_nr^{\star}_nn2^{-n}(1+o(1)),
\Eq(4.prop2.7')
\ee
(the last inequality is  \eqv(3.lem1.2) combined with \eqv(3.lem0.2)) then
\bea
k^2_n(t)\sum_{x\in T^{\circ}_n}\E^{\circ}(X_n(x))^2
\leq
t(2^{\varepsilon n}/b^{\circ}_n)r^{\star}_nn2^{-n}(1+o(1))\E^{\circ}\nu_n^{\circ, t}(u_n,\infty),
\Eq(4.prop2.9')
\eea
where we used that for 
$
a_n=2^{\varepsilon n}/b^{\circ}_n
$,
$\theta_nk_n(t)=\theta_n\lf t(2^{\varepsilon n}/b^{\circ}_n)/\theta_n\rf=t(2^{\varepsilon n}/b^{\circ}_n)(1+o(1))$.
Inserting \eqv(4.prop2.9') in \eqv(4.prop2.6'),
a second order Tchebychev inequality then yields \eqv(4.prop2.10').

To estimate $\E^{\circ}\nu_n^{\circ, t}(u_n,\infty)$ in \eqv(4.prop2.5') we first use Fubini to write,
\be
\E^{\circ}X_n(x)=
\int_0^{\theta_n}h_{n, x}(v)E_{x}\P^{\circ}\left(c_n^{-1}\t_n(x)\ell_n^x(\theta_n-v)>u_n\right)dv.
\Eq(4.prop2.11)
\ee
Denoting by $\P^{x}$  the law of the single variable $\t_n(x)$,
\bea
\P^{\circ}\left(c_n^{-1}\t_n(x)\ell_n^x(\theta_n-v)>u\right)
\Eq(4.prop2.12)
\hspace{-6pt}&=&\hspace{-6pt} 
\frac{\P^{x}\left(c_n^{-1}\t_n(x)\ell_n^x(\theta_n-v)>u_n, \t_n(x)\geq r_n(\varepsilon_n)\right)}
{\P^{x}(\t_n(x)\geq r_n(\varepsilon_n))}\quad\quad\\
\Eq(4.prop2.13)
\hspace{-6pt}&=&\hspace{-6pt} 
\frac{\P^{x}\left(c_n^{-1}\t_n(x)\ell_n^x(\theta_n-v)>u_n\right)}
{\P^{x}(\t_n(x)\geq r_n(\varepsilon_n))}
\eea
where \eqv(4.prop2.13) follows from the definition of $c_n$  (see \eqv(1.theo0.0)), the a priory bound
\be
\ell_n^x(\theta_n-v)\leq \theta_n-v\ll c_n, \quad 0\leq v\leq \theta_n,
\Eq(4.prop2.14)
\ee 
and the fact that $\d_n$ in \eqv(4.c2) in chosen in such a way that 
$
\theta_n r_n(\varepsilon_n)r^{-1}_n(\varepsilon)\leq   n^{-2}\downarrow 0
$
as $n\uparrow\infty$ (see the last inequality in \eqv(4.0)).
Using classical estimates on the asymptotics of gaussian integrals
(see e.g.~\cite{AS} p.~932), Lemma \thv(2.lem2), and again the definition of $c_n$, 
simple calculations yield that for all $0<u<\infty$ and $0\leq v<\theta_n$, \eqv(4.prop2.13) is equal to
\bea
&&{(1+o(1))}F_{\b,\varepsilon,n}\left(\sfrac{\ell_n^x(\theta_n-v)}{u_n}\right)
\frac{\P\left(\t_n(x)>c_n\right)}{\P(\t_n(x)\geq r_n(\varepsilon_n))}
\Eq(4.prop2.15)
\eea
where
$
F_{\b,\varepsilon,n}(x)
$
is defined in  \eqv(4.prop2.16).
Furthermore, by   \eqv(1.theo0.0), $2^{\varepsilon n}\P(\t_n(x)\geq c_n)=1$
whereas by \eqv(2.2), \eqv(2.lem3.5), and \eqv(3.lem0.2),
$
\P(\t_n(x)\geq r_n(\varepsilon_n))=\pi_n(T^{\circ}_n)(1+o(1))
$.
In view of this and \eqv(4.prop2.18) we get, combining \eqv(4.prop2.15),  
\eqv(4.prop2.11), \eqv(4.prop2.5'), and using the a priori bound \eqv(4.prop2.14) that 
\be
\E^{\circ}\nu_n^{\circ, t}(u_n,\infty)=
(1+o(1))k_n(t)\theta_n (b^{\circ}_n/2^{\varepsilon n})
\frac{I_{(0,\theta_n)}(u_n)}{I_{(0,\theta_n)}(1)}
\Eq(4.prop2.30)
\ee
where for $w>0$
\bea
I_{(a,b)}(w)
&=&
\sum_{x\in T^{\circ}_n}\int_0^{\theta_n}h_{n, x}(v)E_{x}\bigl[F_{\b,\varepsilon,n}\bigl(\sfrac{\ell_n^x(\theta_n-v)}{w}\bigr)\bigr]
\1_{\{a\leq \ell_n^x(\theta_n-v)< b\}}dv.
\Eq(4.prop2.32)
\eea
To evaluate the ratio in \eqv(4.prop2.30) set 
$0<\zeta_n\equiv e^{-n^{9/10}}\downarrow 0$ and split the integral in $I_{(0,\theta_n)}(u_n)$
into  $I_{(0,\theta_n)}(u_n)\equiv I_{(0,\zeta_n)}(u_n)+I_{(\zeta_n,\theta_n)}(u_n)$. Note that
$n^{-1}|\log\zeta_n|=n^{-1/10}$, $n^{-1}(\log\zeta_n)^2=n^{4/5}$, while for all $u>0$,
$n^{-1}\log u_n\downarrow 0$,  $n^{-1}(\log u_n)^2\downarrow 0$ as $n\uparrow\infty$.
Using that $F_{\b,\varepsilon,n}(x)$ is increasing on the domain $(0,\zeta_n/u_n)$
\bea
 I_{(0,\zeta_n)}(u_n)
&\leq &
F_{\b,\varepsilon,n}\bigl(\sfrac{\zeta_n}{u_n}\bigr)P_{\pi_n}(H(T^{\circ}_n)<\theta_n)
\Eq(4.prop2.33)
\eea
where
$
F_{\b,\varepsilon,n}\bigl(\sfrac{\zeta_n}{u_n}\bigr)
=
e^{o(1)\log u_n}F_{\b,\varepsilon,n}(\zeta_n)F_{\b,\varepsilon,n}(u^{-1}_n)
$ and
$
F_{\b,\varepsilon,n}(\zeta_n)\leq e^{-\a_{n}(\varepsilon)n^{9/10}-{n^{4/5}}/{2\b^2}}
$.
By this, \eqv(3.lem1.2),  the lower bound \eqv(4.lem3.6') on $b^{\circ}_n$, and our assumptions on $u_n$,
\be
\frac{ I_{(0,\zeta_n)}(u_n)}{ I_{(0,\theta_n)}(1)}=
e^{o(1)\log u_n}F_{\b,\varepsilon,n}(u^{-1}_n)
F_{\b,\varepsilon,n}(\zeta_n)
n\kappa_n(r^{\star}_n)^{1+\a_{n}(\varepsilon)+o(1)}
\rightarrow 0
\Eq(4.prop2.34)
\ee
as $n\rightarrow\infty$.
Next, since $n^{-1}\log l\downarrow 0$ as $n\uparrow\infty$ for all $\zeta_n\leq l\leq \theta_n$ 
we have, using \eqv(4.prop2.14),
\be
\frac{I_{(\zeta_n,\theta_n)}(u_n)}{ I_{(0,\theta_n)}(1)}=
e^{o(1)\log u_n}F_{\b,\varepsilon,n}(u^{-1}_n)\left[1-\sfrac{ I_{(0,\zeta_n)}(u_n)}{ I_{(0,\theta_n)}(1)}\right]
\rightarrow u^{-\a(\varepsilon)}
\Eq(4.prop2.35)
\ee
as $n\rightarrow\infty$ for all $u>0$.
Inserting \eqv(4.prop2.34) and \eqv(4.prop2.35) in \eqv(4.prop2.30), choosing
$
a_n=2^{\varepsilon n}/b^{\circ}_n
$,
and passing to the limit $n\rightarrow\infty$ finally gives \eqv(4.prop2.2). 
The proof of the lemma is done.
\end{proof}

\subsection{Proof of Proposition \thv(4.prop1).} 
 \TH(S4.4)

By \eqv(4.cnew), \eqv(4.4)-\eqv(4.c2),  and the bound $\kappa_n\leq \theta_n$, \eqv(4.lem3.6) implies that 
on $\O_1\cap\O^{\circ}\cap \O^{\star}$, for large enough $n$,
$
b_n=b^{\circ}_n(1+o(1))
$.
The assumption that
$
a_n=2^{\varepsilon n}/b_n
$
in \eqv(4.1)  can thus be replaced by
$
a_n=2^{\varepsilon n}/b^{\circ}_n
$.
Consider now \eqv(4.prop2.10') and note that by \eqv(4.lem3.6'), \eqv(3.prop2.1), \eqv(2.lem2.2), and \eqv(4.cnew) 
(see also \eqv(4.c1'new)), for all $0<\varepsilon<1$,
\be
(2^{\varepsilon n}/b^{\circ}_n)r^{\star}_nn^32^{-n}
\leq
\kappa_n(r^{\star}_n)^{1+\a_{n}(\varepsilon)+o(1)}n^32^{n\varepsilon}2^{-n}\rightarrow 0
\Eq(4.prop1.0)
\ee
as $n\rightarrow\infty$.
Thus, by Proposition \thv(4.prop2) and Borel-Cantelli Lemma we get that for all $u>0$ and all $t>0$,  
 \be
\lim_{n\rightarrow\infty}\nu_n^{\circ, t}(u,\infty)=tu^{\a(\varepsilon)}\quad
\P-\text{almost surely}.
\Eq(4.prop1.6)
\ee
In the same way we get that   for all $u>0$ and all $t>0$,
\be
\lim_{n\rightarrow\infty}\nu_n^{\circ, t}(u,\infty)=tu^{\a(\varepsilon)}\quad
\P-\text{almost surely}.
\Eq(4.prop1.6')
\ee

Next, by Lemma \thv(4.lem1),  Lemma \thv(4.lem2), and \eqv(4.24) we have
that on $\O^{\star}$, for all but a finite number of indices $n$, 
\bea
&&\hspace{-6pt}\left|\check{\nu}_n^t(u,\infty)-\nu_n^{\circ, t}(u,\infty)\right|
\\
\hspace{-6pt}&\leq&\hspace{-6pt} 
t(b^{\circ}_n)^{-1}[2 r^{\star}_nn^5\theta_n2^{-n\varepsilon+2\delta_n n}
+n^{c_{\star}+4}2^{n\varepsilon}\left(\theta_n \pi_n(T^{\circ}_n)r^{\star}_n\right)^2]
(1+o(1)) \quad
\Eq(4.prop1.2)
\\
\hspace{-6pt}&\leq&\hspace{-6pt} 
2tn^{c_{\star}+4(1+\a_{n}(\varepsilon))}(r^{\star}_n)^{\a_{n}(\varepsilon)+2+o(1)}\kappa_n\theta_n^{2+2\a(\varepsilon)}2^{-n\varepsilon}
(1+o(1)) 
\Eq(4.prop1.3)
\eea
where the last inequality follows from \eqv(4.lem3.6'), \eqv(2.lem3.5), \eqv(4.4), and \eqv(4.c2).
Since $\kappa_n\leq \theta_n$, \eqv(4.cnew) (see also \eqv(4.c1'new))  implies that
\eqv(4.prop1.3)  decays to zero as $n\rightarrow\infty$.
From this and \eqv(4.prop1.6) we get that  for all $u>0$ and all $t>0$,  
$
\lim_{n\rightarrow\infty}\check{\nu}_n^t(u,\infty)=tu^{\a(\varepsilon)}
$ 
$\P$-almost surely. One proves in the same way that   for all $u>0$ and all $t>0$, 
$
\lim_{n\rightarrow\infty}\check{\nu}_n^t\left(u-n^{-2},\infty\right)=tu^{\a(\varepsilon)}
$ 
$\P$-almost surely. Therefore, by \eqv(4.5),   for all $u>0$ and all $t>0$, 
\be
\lim_{n\rightarrow\infty}\nu_n^t(u,\infty)=tu^{\a(\varepsilon)}\quad\P-\text{almost surely}.
\Eq(4.prop1.4'')
\ee
Since $\nu_n^t$ is increasing both in $t$ and $u$ and since  its limit continuous in those two variables, \eqv(4.prop1.4'') implies that $\P$-almost surely,
\be
\lim_{n\rightarrow\infty}\nu_n^t(u,\infty)=tu^{\a(\varepsilon)},\quad\forall\, u>0, t>0.
\Eq(4.prop1.5)
\ee
The proof of Proposition \thv(4.prop1) is done.


 
 \section{Verification of Condition (B2)} 
 \label{S5}
 
By \eqv(1.03.9)-\eqv(1.03.10), \eqv(1.3.6), and \eqv(1.03.12),
Condition (B2) in \eqv(B2.1) states that 
\be
\s_n^{t}(u,\infty)
\equiv k_n(t)\sum_{y\in\VV_n}\pi_n(y)
\left[P_{y}\left(\int_{0}^{\theta_n}\max\left((c_nr^{\star}_n)^{-1},c_n^{-1}\t_n(Y_n(s))\right)ds>u\right)\right]^2
\Eq(5.1)
\ee
decays to zero as $n$ diverges. We prove in this section that this holds true $\P$-almost surely.
\begin{proposition}
   \TH(5.prop1)
   Under the assumptions of Proposition \thv(4.prop1), for all $0<\varepsilon<1$ and $\b>0$, $\P$-almost surely,
\be
\lim_{n\rightarrow\infty}\s_n^t(u,\infty)=0,\quad\forall\, t>0, u>0.
\Eq(5.prop1.1)
\ee
\end{proposition}

As in the proof of Proposition \thv(4.prop1) we first bring $\s_n^{t}(u,\infty)$ into a suitable form.
Proceeding as in \eqv(4.5)-\eqv(4.6), we first write
\be
\check{\s}_n^t(u,\infty)\leq \s_n^{t}(u,\infty)\leq \check{\s}_n^t(u-n^{-2},\infty)
\Eq(5.2)
\ee
where
\be
\check{\s}_n^t(u,\infty)
\equiv 
k_n(t)\sum_{y\in\VV_n}\pi_n(y)\left[P_{y}\left(
\int_{0}^{\theta_n}c_n^{-1}\t_n(Y_n(s))\1_{\{Y_n(s)\in T_n\}}ds>u
\right)\right]^2,
\Eq(5.3)
\ee
and next reduce visits to $T_n$ in \eqv(5.3) to visits to visits to $T^{\circ}_n$, just as in Lemma \thv(4.lem1). Set
\be
\bar\s_n^t(u,\infty)
\equiv
 k_n(t)\sum_{y\in\VV_n}\pi_n(y)\left[P_{y}\left(
 \int_{0}^{\theta_n}c_n^{-1}\t_n(Y_n(s))\1_{\{Y_n(s)\in T^{\circ}_n\}}ds>u
 \right)\right]^2.
\Eq(5.4)
\ee

\begin{lemma}
    \TH(5.lem1)  
Assume that  \eqv(4.cnew) holds. Then on $\O^{\star}$, for all but a finite number of indices $n$,
\be
|\check{\s}_n^t(u,\infty)-\bar\s_n^t(u,\infty)|
\leq
6 k_n(t)\theta_n n^5r^{\star}_n2^{-2n\varepsilon_n}(1+o(1)).
\Eq(5.lem1.0)
\ee
\end{lemma}

\begin{proof}[Proof of lemma \thv(5.lem1)]  
As in the Proof of Lemma \thv(4.lem1) we decompose the event appearing in the probability in \eqv(5.3) 
according to whether $\{H(T_n\setminus T^{\circ}_n)\leq \theta_n\}$ or 
not, that is,  setting
\bea
q_1(y)\hspace{-6pt}&=&\hspace{-6pt} \textstyle
P_{y}\bigl(
\int_{0}^{\theta_n}c_n^{-1}\t_n(Y_n(s))\1_{\{Y_n(s)\in T_n\}}ds>u, H(T_n\setminus T^{\circ}_n)\leq \theta_n
\bigr),
\Eq(5.lem1.1)
\\
q_2(y)\hspace{-6pt}&=&\hspace{-6pt} \textstyle
P_{y}\bigl(
\int_{0}^{\theta_n}c_n^{-1}\t_n(Y_n(s))\1_{\{Y_n(s)\in T_n\}}ds>u, H(T_n\setminus T^{\circ}_n)> \theta_n
\bigr),
\Eq(5.lem1.2)
\eea
we write 
$
\check{\s}_n^t(u,\infty)=k_n(t)\sum_{y\in\VV_n}\pi_n(y)[q_1(y)+q_2(y)]^2
$.
In the same way write
$
\bar\s_n^t(u,\infty)=k_n(t)\sum_{y\in\VV_n}\pi_n(y)[\bar q_1(y)+\bar q_2(y)]^2
$
where $\bar q_1(y)$ and $\bar q_2(y)$ are defined as in \eqv(5.lem1.1) and \eqv(5.lem1.2), respectively, substituting $T^{\circ}_n$ for $T_n$. Note that
\be
[x_1+x_2]^2
\leq 3x_1+x_2^2,\quad 0\leq x_1,x_2\leq 1.
\Eq(5.lem1.3)
\ee
Applying \eqv(5.lem1.3) to  the terms $[q_1(y)+q_2(y)]^2$ and $[\bar q_1(y)+\bar q_2(y)]^2$, and  observing that 
$q^2_2= \bar q^2_2$, we get
\bea
|\check{\s}_n^t(u,\infty)-\bar\s_n^t(u,\infty)|
\hspace{-6pt}&\leq&\hspace{-6pt} 
\Eq(5.lem1.4)
3k_n(t)\sum_{y\in\VV_n}\pi_n(y)(q_1(y)+\bar q_1(y))
\\
\hspace{-6pt}&\leq&\hspace{-6pt} 
6k_n(t)P_{\pi_n}(H(T_n\setminus T^{\circ}_n)\leq \theta_n).
\Eq(5.lem1.5)
\eea
The Lemma now follows from \eqv(3.cor1.1) of Corollary \thv(3.cor1).
\end{proof}
 
We continue our parallel with the proof of Proposition \thv(4.prop1) and decompose  \eqv(5.4) according to the hitting time
and  hitting place of the set $T^{\circ}_n$. We slightly abuse the notation of Section 3 (see the paragraph below \eqv(3.prop3.3)) and denote by $h^y_{n, x}$ (instead of $h^y_{n, x, T^{\circ}_n}$) the joint density that 
$T^{\circ}_n$ is reached at time $t$, and that arrival to that set occurs in state $x$, given that the process starts in $y$.
As already observed (see the paragraph below \eqv(4.9)), $h_{n, x}=\sum_{y\in\VV_n}\pi_n(y)h^y_{n, x}$.
Proceeding as in \eqv(4.10)-\eqv(4.12) we then get
\be
\bar\s_n^t(u,\infty)
=
k_n(t)\sum_{y\in\VV_n}\pi_n(y)\left[\overline R^{u}_n(y)\right]^2
\Eq(5.5)
\ee
where, using \eqv(4.11) and \eqv(4.12),
\bea
\overline R^{u}_n(y)
\hspace{-6pt}&\equiv&\hspace{-6pt} 
\sum_{x\in T^{\circ}_n}\int_0^{\theta_n}h^y_{n, x}(v) \left(\wt Q^{u, v}_n(x)+\wh Q^{u, v}_n(x)\right) dv
\equiv\wt R^{u}_n(y)+\wh R^{u}_n(y).
\Eq(5.7)
\eea
By analogy with \eqv(4.17) we also set
\bea
\wh\s_n^t(u,\infty)
\hspace{-6pt}&\equiv&\hspace{-6pt}  
k_n(t)\sum_{y\in\VV_n}\pi_n(y)\bigl[\wh R^{u}_n(y)\bigr]^2.
\Eq(5.8)
\eea
The next lemma plays the role of Lemma \thv(4.lem2).

\begin{lemma}
     \TH(5.lem2)
Assume that  \eqv(4.cnew) holds. Then $\O^{\star}$, for all but a finite number of indices $n$,
\be
0\leq 
\bar\s_n^t(u,\infty)-\wh\s_n^t(u,\infty)
\leq 
3k_n(t)n^{c_{\star}+4}\left(\theta_n \pi_n(T^{\circ}_n)r^{\star}_n\right)^2(1+o(1)).
\Eq(5.lem2.1)
\ee
\end{lemma}

\begin{proof}[Proof of Lemma \thv(5.lem2)]  
As in the proof of Lemma \thv(5.lem1), the proof of Lemma \thv(5.lem2) relies on the observation that since
$0\leq \wt R^{u}_n(y),\wh R^{u}_n(y)\leq 1$ in \eqv(5.7) for all $y\in\VV_n$,
then by \eqv(5.lem1.3),
\bea
\textstyle
0<
\bar\s_n^t(u,\infty)-\wh\s_n^t(u,\infty)
\hspace{-6pt}&\leq&\hspace{-6pt}  
3k_n(t)\sum_{y\in\VV_n}\pi_n(y)\wt R^{u}_n(y)
\Eq(5.lem2.3)
\\
\hspace{-6pt}&=&\hspace{-6pt}  
3k_n(t) \sum_{x\in T^{\circ}_n}\int_0^{\theta_n}h_{n, x}(v)\wt Q^{u, v}_n(x)dv
\leq 3\wt\nu_n^t.\quad\quad
\Eq(5.lem2.4)
\eea
The equality in \eqv(5.lem2.4)
follows from the identity $h_{n, x}(v)=\sum_{y\in\VV_n}\pi_n(y)h^y_{n, x}(v)$,
and the final inequality is \eqv(4.15').
The claim of the lemma now follows from Lemma \thv(4.lem2).
\end{proof}

We now need an upper bound on $\wh\s_n^t(u,\infty)$. For this we proceed as in \eqv(4.18)-\eqv(4.20) and write that
$
0\leq \wh\s_n^t(u,\infty)\leq \s_n^{\circ, t}(u,\infty)
$
where, by analogy with \eqv(4.20),
\be
\s_n^{\circ, t}(u,\infty)
=
k_n(t)\sum_{y\in\VV_n}\pi_n(y)\Biggl[
\sum_{x\in T^{\circ}_n}\int_0^{\theta_n}
h^y_{n, x}(v) P_{x}\left(c_n^{-1}\t_n(x)\ell_n^x(\theta_n-v)>u\right)dv
\Biggr]^2
\Eq(5.9)
\ee
Again, the quantity in between the square brackets is in $[0,1]$. Thus, splitting the integral into the sum of the integrals over 
$[0,\kappa_n]$ and $[\kappa_n, \theta_n]$, we get, using \eqv(5.lem1.3) and reasoning as in \eqv(5.lem2.3)-\eqv(5.lem2.4), 
\be
\s_n^{\circ, t}(u,\infty)
\leq
3\bar\eta_n^{\circ, t}(u,\infty)+\eta_n^{\circ, t}(u,\infty)
\Eq(5.10)
\ee
where
\bea
\bar\eta_n^{\circ, t}(u,\infty)
\hspace{-6pt}&\equiv&\hspace{-6pt} 
k_n(t)\sum_{x\in T^{\circ}_n}\int_0^{\kappa_n}
h_{n, x}(v) P_{x}\left(c_n^{-1}\t_n(x)\ell_n^x(\theta_n-v)>u\right)dv,
\Eq(5.11)
\\
\eta_n^{\circ, t}(u,\infty)
\hspace{-6pt}&\equiv&\hspace{-6pt} 
k_n(t)\sum_{y\in\VV_n}\pi_n(y)\Biggl[
\sum_{x\in T^{\circ}_n}\int_{\kappa_n}^{\theta_n}
h^y_{n, x}(v) P_{x}\left(c_n^{-1}\t_n(x)\ell_n^x(\theta_n-v)>u\right)dv
\Biggr]^2. \quad\,\,\,
\Eq(5.12)
\eea
The next two propositions bound \eqv(5.11) and \eqv(5.12) in terms of the quantities
$\nu_n^{\circ, t}(u_n,\infty)$ and $\E^{\circ}\nu_n^{\circ, t}(u_n,\infty)$
defined in  \eqv(4.19) and  \eqv(4.prop2.5'), respectively.

\begin{proposition}
     \TH(5.prop2)  
Choose $a_n=2^{\varepsilon n}/b^{\circ}_n$ in \eqv(1.03.8).
Then, for any sequence $u_n>0$ such that $0<u-u_n<n^{-1}$ and all $u>0$,
\be
\P\left(\bar\eta_n^{\circ, t}(u_n,\infty)\geq 
 t\E^{\circ}\nu_n^{\circ, t}(u_n,\infty)
n^2\theta_n^{-1}\tilde\kappa_n\kappa_n^{1+\a_{n}(\varepsilon)}(nr^{\star}_n)^{1+\a_{n}(\varepsilon)+o(1)}
\right)
\leq
n^{-2}.
\Eq(5.prop2.1)  
\ee
\end{proposition}

\begin{proposition}
     \TH(5.prop3)  
On $\O^{\star}\cap \O_1$, for all but a finite number of indices $n$ and all $u>0$,
\be
\eta_n^{\circ, t}(u,\infty)
\leq 
\nu_n^{\circ, t}(u,\infty)\theta_nr^{\star}_nn2^{-n\varepsilon_n}(1+o(1)).
\Eq(5.prop3.2)  
\ee
\end{proposition}

\begin{proof}[Proof of Proposition \thv(5.prop2)] 
As in the proof of Proposition \thv(4.prop2) denote by  $\P^{\circ}$ 
the law of the collection $\{\t_n(x), x\in T^{\circ}_n\}$ conditional on $T^{\circ}_n$.
By a first order Tchebychev inequality,
\be
\P\left(\bar\eta_n^{\circ, t}(u_n,\infty)\geq \e\right)
\leq
\e^{-1}\E\left[\E^{\circ}\bar\eta_n^{\circ, t}(u_n,\infty)\right].
\Eq(5.prop2.3)  
\ee
Note that 
$
\E^{\circ}\bar\eta_n^{\circ, t}(u,\infty)
$
only differs from the term $\E^{\circ}\nu_n^{\circ, t}(u_n,\infty)$ of \eqv(4.prop2.5')
in that the integral in \eqv(5.11) is over $[0,\kappa_n]$ instead of $[0, \theta_n]$. 
Taking $a_n=2^{\varepsilon n}/b^{\circ}_n$, a simple adaptation of the proof of
\eqv(4.prop2.2) (see \eqv(4.prop2.11)-\eqv(4.prop2.35))  yields
\be
\E^{\circ}\bar\eta_n^{\circ, t}(u_n,\infty)=
 t(1+o(1))\E^{\circ}\nu_n^{\circ, t}(u_n,\infty)
\frac{\II_n(0,\kappa_n)}{\II_n(0, \theta_n)}
\Eq(5.prop2.4)  
\ee
where $\II_n(a,b)$ is defined above \eqv(4.lem3.1).
Eq.~\eqv(4.lem3.3) of Lemma \thv(4.lem3) was  designed precisely to control the ratio in \eqv(5.prop2.4). Namely,
on $\O^{\circ}\cap \O^{\star}$, for all but a finite number of indices $n$,
\be
\frac{\II_n(0,\kappa_n)}{\II_n(0, \theta_n)}\leq
\frac{\II_n(0,\kappa_n)}{\II_n(\kappa_n, \theta_n-\zeta_n)}\leq
\theta_n^{-1}\tilde\kappa_n\kappa_n^{1+\a_{n}(\varepsilon)}(nr^{\star}_n)^{1+\a_{n}(\varepsilon)+o(1)}.
\Eq(5.prop2.5)  
\ee
The combination of \eqv(5.prop2.3), \eqv(5.prop2.4), and \eqv(5.prop2.5) gives \eqv(5.prop2.1). The proof is complete.
\end{proof}

\begin{proof}[Proof of Proposition \thv(5.prop3)] 

To prove \eqv(5.prop3.2) first observe that 
\bea
\sum_{x\in T^{\circ}_n}\int_{\kappa_n}^{\theta_n}
h^y_{n, x}(v) P_{x}\left(c_n^{-1}\t_n(x)\ell_n^x(\theta_n-v)>u\right)dv
\hspace{-6pt}&\leq&\hspace{-6pt} 
P_y(\kappa_n< H(T^{\circ}_n)\leq \theta_n)
\Eq(5.prop3.6) 
\\
\hspace{-6pt}&\leq&\hspace{-6pt} 
(1+o(1))P_{\pi_n}(H(T^{\circ}_n)\leq \theta_n)\quad\quad
\Eq(5.prop3.7) 
\eea
where the last line follows from Proposition \thv(3.prop2) and the Markov property,
and is valid on $\O_1$, for all but a finite number of indices $n$.
Applying this bound to one of the two square brackets in \eqv(5.12) and using \eqv(4.19) to bound the remaining term,
we get, under the same assumptions as above,  that
\bea
\eta_n^{\circ, t}(u,\infty)
\hspace{-6pt}&\leq&\hspace{-6pt} 
(1+o(1))\nu_n^{\circ, t}(u,\infty)P_{\pi_n}(H(T^{\circ}_n)\leq \theta_n).
\Eq(5.prop3.8)
\eea
Using Corollary \eqv(3.cor1.1') to bound the last probability yields the claim of the proposition.
\end{proof}

We are now ready to complete the

\begin{proof}[Proof of Proposition \thv(5.prop1)]  
Recall from the proof of Proposition \thv(4.prop1) that on $\O_1\cap\O^{\circ}\cap \O^{\star}$
$
a_n=2^{\varepsilon n}/b_n=2^{\varepsilon n}/b^{\circ}_n(1+o(1))
$
for large enough $n$ and consider \eqv(5.prop2.1). By \eqv(4.c4),
$
n^2\theta_n^{-1}\tilde\kappa_n\kappa_n^{1+\a_{n}(\varepsilon)}(nr^{\star}_n)^{1+\a_{n}(\varepsilon)+o(1)}
\downarrow 0
$
as $n\uparrow\infty$
and by \eqv(4.prop2.2),  for all  $u>0$ and $t>0$
$
\lim_{n\rightarrow\infty}\E^{\circ}\nu_n^{\circ, t}(u_n,\infty)=tu^{\a(\varepsilon)}
$.
Thus, by Proposition \thv(5.prop2) and Borel-Cantelli Lemma we get that for all $u>0$ and $t>0$,
\be
\lim_{n\rightarrow\infty}\bar\eta_n^{\circ, t}(u,\infty)=0\quad
\P-\text{almost surely}.
\ee
Turning to \eqv(5.prop3.2) and invoking \eqv(4.cnew) (see also \eqv(4.c1'new)),
it follows from Proposition \thv(5.prop2) that for all $0<\varepsilon<1$ and for all $u>0$ and $t>0$,
\be
\lim_{n\rightarrow\infty}\eta_n^{\circ, t}(u,\infty)=0\quad
\P-\text{almost surely}.
\ee
Hence by \eqv(5.10), for all $u>0$ and $t>0$,
\be
\lim_{n\rightarrow\infty}\s_n^{\circ, t}(u,\infty)=0\quad
\P-\text{almost surely}.
\ee
From there on the proof is a rerun of the proof of  Proposition \thv(4.prop1) 
with  Lemma \thv(5.lem1) and Lemma \thv(5.lem2)
playing the role of Lemma \thv(4.lem1) and Lemma \thv(4.lem2), respectively.
We omit the details.
\end{proof}



 \section{Verification of Condition (B3)} 
 \label{S6}
 
By \eqv(1.03.9)-\eqv(1.03.11), \eqv(1.3.6), and \eqv(1.03.12), Condition (B3) in \eqv(B3.1)  
will be verified if we can establish that:

\begin{proposition}
   \TH(6.prop1)
   
Under the assumptions of Proposition \thv(4.prop1), for all $0<\varepsilon<1$ and all $\b>\b_c(\varepsilon)$, $\P$-almost surely,
\be
\lim_{\e\downarrow 0}\limsup_{n\uparrow \infty}
k_n(t)E_{\pi_n}\int_{0}^{\theta_n}
\MM_n(Y_n(s))
\1_{\{\int_{0}^{\theta_n}
\MM_n(Y_n(s))
ds\leq \e\}}=0,\quad \forall t>0.
\Eq(6.prop1.1)
\ee
where $\MM_n(Y_n(s))=\max\left((c_nr^{\star}_n)^{-1},c_n^{-1}\t_n(Y_n(s))\right)$.
\end{proposition}
 
The Lemma below is central to the proof.

\begin{lemma}
   \TH(6.lem1)
There are constants $K,K'<\infty$ such that for $\a_{n}(\varepsilon)$ as in \eqv(4.prop2.17) and any sequence $\e_n>0$ such that $i\a^{-1}_c(\varepsilon)-1-\frac{\log \e_n}{n\b\b_c(\varepsilon)}>0$ where $i=1$ in \eqv(6.lem1.1) and $i=2$ in \eqv(6.lem1.2), we have, for all large enough $n$,
\be
\E 2^{\varepsilon n} c_n^{-1}\t_n(x)\1_{\{c_n^{-1}\t_n(x)\leq \e_n\}}
\leq 
K
\frac{
\e_n^{1-\a_{n}(\varepsilon)-\frac{\log \e_n}{2n\b^2}}
}
{
\a^{-1}_c(\varepsilon)-1-\frac{\log \e_n}{n\b\b_c(\varepsilon)}
},
\Eq(6.lem1.1)
\ee
\be
\E \left(2^{\varepsilon n} c_n^{-1}\t_n(x)\1_{\{c_n^{-1}\t_n(x)\leq \e_n\}}\right)^2
\leq 
K'
\frac{
\e_n^{2-\a_{n}(\varepsilon)-\frac{\log \e_n}{2n\b^2}}
}
{
2\a^{-1}_c(\varepsilon)-1-\frac{\log \e_n}{n\b\b_c(\varepsilon)}
}.
\Eq(6.lem1.2)
\ee
\end{lemma}

\begin{proof}[Proof of Lemma \thv(6.lem1)]  
Using standard estimates on the asymptotics of Gaussian integrals (see e.g.~\cite{AS} p.~932)  the claimed
result follows from straightforward computations.
\end{proof}

\begin{proof}[Proof of Proposition \thv(6.prop1)]  We assume throughout that 
$\o\in\O_1\cap\O^{\circ}\cap \O^{\star}$ and that $n$ is as large as desired.
Note that $\MM_n(Y_n(s))\leq (c_nr^{\star}_n)^{-1}+c_n^{-1}\t_n(Y_n(s))$ and that the contribution to \eqv(6.prop1.1)
coming from the term $(c_nr^{\star}_n)^{-1}$ if or order $o(1)$. Indeed by \eqv(1.03.8), \eqv(1.theo1.0), 
the lower bound on $b_n$ obtained by combining \eqv(4.lem3.6') and \eqv(4.lem3.6), the expression \eqv(1.theo0.0') of $c_n$, the expression  \eqv(3.prop2.1) of $\kappa_n$, and the fact, that follows from \eqv(1.2.3),
that $2^n=e^{n\b^2_c(\varepsilon)/2}$, 
\be
k_n(t)\theta_n(c_nr^{\star}_n)^{-1}
\leq
2tn^4(r^{\star}_n)^{\a_{n}(\varepsilon)+o(1)}e^{n\b^2_c(\varepsilon)/2}e^{-n\b\b_c(\varepsilon)(1+o(1))}
\Eq(6.prop1.01)
\ee
and  so, for all $0<\varepsilon<1$ and $\b>\b_c(\varepsilon)$, by virtue of \eqv(4.cnew) (see also \eqv(4.c1'new))
\be
k_n(t)\theta_n(c_nr^{\star}_n)^{-1}
\leq 
2tn^4(r^{\star}_n)^{\a_{n}(\varepsilon)+o(1)}
e^{-n\b^2_c(\varepsilon)(1+o(1))/2}\rightarrow 0
\Eq(6.prop1.02)
\ee
as $n\rightarrow\infty$. To prove Proposition \thv(6.prop1) it thus suffices to establish that $\P$-almost surely,
\be
\lim_{\e\downarrow 0}\limsup_{n\uparrow \infty}
k_n(t)E_{\pi_n}\int_{0}^{\theta_n}
c_n^{-1}\t_n(Y_n(s))
\1_{\{\int_{0}^{\theta_n}
c_n^{-1}\t_n(Y_n(s))
ds\leq \e\}}=0,\quad \forall t>0.
\Eq(6.prop1.1bis)
\ee

For $T_n$ as in \eqv(2.7) with $\d_n$ given by \eqv(4.c2), set
\bea
\SS^{(1)}_{n,\e}(t)
\hspace{-6pt}&\equiv&\hspace{-6pt} 
k_n(t)E_{\pi_n}\int_{0}^{\theta_n}c_n^{-1}\t_n(Y_n(s))\1_{\{Y_n(s)\in T_n\}}
\1_{\{\int_{0}^{\theta_n}c_n^{-1}\t_n(Y_n(s))ds\leq \e\}}
ds,
\Eq(6.prop1.2)
\\
\SS^{(2)}_{n,\e}(t)
\hspace{-6pt}&\equiv&\hspace{-6pt} 
k_n(t)E_{\pi_n}\int_{0}^{\theta_n}c_n^{-1}\t_n(Y_n(s))\1_{\{Y_n(s)\notin T_n\}}
\1_{\{\int_{0}^{\theta_n}c_n^{-1}\t_n(Y_n(s))ds\leq \e\}}
ds.
\Eq(6.prop1.3)
\eea
To bound $\SS^{(2)}_{n,\e}(t)$ simply note that, using \eqv(3.lem0.3),
\bea
\SS^{(2)}_{n,\e}(t)
\hspace{-6pt}&\leq&\hspace{-6pt} 
k_n(t)E_{\pi_n}\int_{0}^{\theta_n}c_n^{-1}\t_n(Y_n(s))\1_{\{\t_n(Y_n(s))\leq r_n(\varepsilon_n)\}}
ds
\Eq(6.prop1.4)
\\
\hspace{-6pt}&\leq &\hspace{-6pt} 
k_n(t)\theta_n 2^{-n}(1+o(1))\sum_{x\in\VV_n}c_n^{-1}\t_n(x)\1_{\{\t_n(x)\leq r_n(\varepsilon_n)\}}.
\Eq(6.prop1.5)
\eea
Take $\e_n=c_n^{-1}r_n(\varepsilon_n)$ and note that by \eqv(2.cor1.1),
the definition of $c_n$, and \eqv(4.cnew),
\be
-(n\b\b_c(\varepsilon))^{-1}\log \e_n=o(1)\,\,\,\text{and}\,\,\,
 \left(n^{2(1+c_{\star}6^2/\a(\varepsilon))}\theta_n\right)^{-1}\leq \e_n\leq (n^2\theta_n)^{-1}.
 \Eq(6.prop1.5')
\ee 
Thus, by Lemma \thv(6.lem1)  and a first order Tchebychev inequality, for all large enough $n$,
\be
\P\left(
\SS^{(2)}_{n,\e}(t)
\geq n^2tb_n^{-1}(c_n^{-1}r_n(\varepsilon_n))^{1-\a(\varepsilon)+o(1)}
\right)
\leq
n^{-2}K''
\Eq(6.prop1.6)
\ee
for some constant $K''>0$.
Using the upper bound on $\e_n$ of \eqv(6.prop1.5') and the lower bound on $b_n$ of  Lemma \thv(4.lem3)
obtained by combining \eqv(4.lem3.6') and \eqv(4.lem3.6), 
\be
n^2b_n^{-1}(c_n^{-1}r_n(\varepsilon_n))^{1-\a(\varepsilon)+o(1)}
\leq
n^2\kappa_n(r^{\star}_n)^{\a_{n}(\varepsilon)+o(1)}\left(n^2\theta_n\right)^{-1+\a(\varepsilon)+o(1)}\rightarrow 0
\Eq(6.prop1.7)
\ee
as $n\rightarrow\infty$ by \eqv(4.c4).
Hence by \eqv(6.prop1.6),  \eqv(6.prop1.7), and Borel-Cantelli Lemma,
for all $\e>0$,
\be
\lim_{n\rightarrow\infty}\SS^{(2)}_{n,\e}(t)=0,\quad \P-\text{almost surely.}
\Eq(6.prop1.8)
\ee

To deal with $\SS^{(1)}_{n,\e}(t)$ we further decompose it into $\SS^{(1)}_{n,\e}(t)=\SS^{(3)}_{n,\e}(t)+\SS^{(4)}_{n,\e}(t)$,
where
\bea
\SS^{(3)}_{n,\e}(t)
\hspace{-6pt}&\equiv&\hspace{-6pt} 
k_n(t)E_{\pi_n}\int_{0}^{\theta_n}c_n^{-1}\t_n(Y_n(s))
\1_{\{Y_n(s)\in T^{\circ}_n\}}\1_{\{\int_{0}^{\theta_n}c_n^{-1}\t_n(Y_n(s))ds\leq \e\}}ds,
\Eq(6.prop1.9)
\\
\SS^{(4)}_{n,\e}(t)
\hspace{-6pt}&\equiv&\hspace{-6pt} 
k_n(t)E_{\pi_n}\int_{0}^{\theta_n}c_n^{-1}\t_n(Y_n(s))
\1_{\{Y_n(s)\in T_n\setminus T^{\circ}_n\}}\1_{\{\int_{0}^{\theta_n}c_n^{-1}\t_n(Y_n(s))ds\leq \e\}}
ds.
\Eq(6.prop1.10)
\eea
Since $\SS^{(4)}_{n,\e}(t)$ is non zero only if  the event $\{H(T_n\setminus T^{\circ}_n)\leq \theta_n\}$ occurs, 
\bea
\SS^{(4)}_{n,\e}(t)
\leq
\e k_n(t)E_{\pi_n}\1_{\{H(T_n\setminus T^{\circ}_n)\leq \theta_n\}}.
\Eq(6.prop1.12)
\eea
Using assertion (ii) of Corollary \thv(3.cor1) with $t_n=\theta_n$ as in the proof of Lemma \thv(4.lem1),
we get, assuming \eqv(4.cnew),  that on $\O^{\star}$, for all but a finite number of indices $n$,
\be
\SS^{(4)}_{n,\e}(t)
\leq
\e k_n(t)\theta_n r^{\star}_nn2^{-2n\varepsilon_n}(1+o(1)),
\Eq(6.prop1.13)
\ee
Proceeding as in \eqv(6.prop1.7) to bound $b_n$,
 \eqv(4.cnew) (see also \eqv(4.c1'new)) guarantees that for all $\e>0$
\be
\lim_{n\rightarrow\infty}\SS^{(4)}_{n,\e}(t)=0,\quad \P-\text{almost surely.}
\Eq(6.prop1.14)
\ee
Using next that
$
\int_{0}^{\theta_n}c_n^{-1}\t_n(Y_n(s))\1_{\{Y_n(s)\in A\}}=\sum_{x\in A}c_n^{-1}\t_n(x)\ell_n^x(\theta_n)
$
for any $A\subseteq\VV_n$,
\be
\SS^{(3)}_{n,\e}(t)
\leq 
\SS^{(5)}_{n,\e}(t)
\equiv
k_n(t)E_{\pi_n}\sum_{x\in T^{\circ}_n}c_n^{-1}\t_n(x)\ell_n^x(\theta_n)
\1_{\{\sum_{x\in T^{\circ}_n}c_n^{-1}\t_n(x)\ell_n^x(\theta_n)\leq \e\}}.
\Eq(6.prop1.15)
\ee
With the notation of \eqv(4.7)-\eqv(4.9),
\be
\nonumber
\SS^{(5)}_{n,\e}(t)
=
 k_n(t)
\sum_{y\in T^{\circ}_n}\int_0^{\theta_n}dv h_{n, y}(v)
E_{y}\sum_{x\in T^{\circ}_n}c_n^{-1}\t_n(x)\ell_n^x(\theta_n-v)
\1_{\{\sum_{x\in T^{\circ}_n}c_n^{-1}\t_n(x)\ell_n^x(\theta_n-v)\leq \e\}}.
\ee
We further split the sum over $x$ above into $x=y$ and $x\neq y$. The latter contribution is
\be
\nonumber
\SS^{(6)}_{n,\e}(t)
\equiv
 k_n(t)
\sum_{y\in T^{\circ}_n}\int_0^{\theta_n}dv h_{n, y}(v)
E_{y}\sum_{x\in T^{\circ}_n\setminus y}c_n^{-1}\t_n(x)\ell_n^x(\theta_n-v)
\1_{\{\sum_{x\in T^{\circ}_n}c_n^{-1}\t_n(x)\ell_n^x(\theta_n-v)\leq \e\}}.
\ee
Observing that
\bea
E_{y}\sum_{x\in T^{\circ}_n\setminus y}c_n^{-1}\t_n(x)\ell_n^x(\theta_n-v)
\1_{\{\sum_{x\in T^{\circ}_n}c_n^{-1}\t_n(x)\ell_n^x(\theta_n-v)\leq \e\}}
\leq
\e P_y(H(T^{\circ}_n\setminus y)\leq \theta_n),\,\,
\Eq(6.prop1.16)
\eea
yields the bound
$
\SS^{(6)}_{n,\e}(t)\leq \e \wt\nu_n^t
$
where $\wt\nu_n^t$ is defined in \eqv(4.lem2.1). Thus by Lemma \thv(4.lem2), reasoning as in the paragraph below  \eqv(4.prop1.3), we get that for all $\e>0$
\be
\lim_{n\rightarrow\infty}\SS^{(6)}_{n,\e}(t)=0,\quad \P-\text{almost surely.}
\Eq(6.prop1.17)
\ee
It remains to bound
$
\SS^{(5)}_{n,\e}(t)-\SS^{(6)}_{n,\e}(t)
$.
For this we write 
$
\SS^{(5)}_{n,\e}(t)-\SS^{(6)}_{n,\e}(t)
\leq
\SS^{(7)}_{n,\e}(t)
$
where
\be
\SS^{(7)}_{n,\e}(t)
\equiv
 k_n(t)
\sum_{y\in T^{\circ}_n}\int_0^{\theta_n}dv h_{n, y}(v)
E_{y}c_n^{-1}\t_n(y)\ell_n^y(\theta_n-v)
\1_{\{c_n^{-1}\t_n(y)\ell_n^x(\theta_n-v)\leq \e\}}.
\Eq(6.prop1.18)
\ee
Let us now establish that for $b^{\circ}_n$  as in \eqv(4.prop2.18), $\SS^{(7)}_{n,\e}(t)$ obeys the following

\begin{lemma}
     \TH(6.lem2)  
Let the sequences  $a_n$,  $c_n$, $\theta_n$ be as in Proposition \thv(6.prop1). 
Then, under the assumptions and with the notation of Proposition \thv(4.prop2),  
\be
\P^{\circ}\left(
\left|
\SS^{(7)}_{n,\e}(t)-\E^{\circ}\SS^{(7)}_{n,\e}(t)
\right|>t \e^{1/2} n 2^{-n(1-\varepsilon)/2}
\right)
\\
\leq n^{-2}(1+o(1))
\Eq(6.lem2.1)
\ee
for all $\e>0$, and
\be
\lim_{\e\rightarrow 0}\lim_{n\rightarrow\infty}\E^{\circ}\SS^{(7)}_{n,\e}(t)=0.
\Eq(6.lem2.2)
\ee  
\end{lemma}

\begin{proof}[Proof of lemma \thv(6.lem2)]  The proof closely follows that of Proposition \thv(4.prop2). We only point out
the main differences. The random variables  \eqv(4.prop2.5) are now replaced by
\be
X_n(y)\equiv
\int_0^{\theta_n}dv h_{n, y}(v)E_{y}c_n^{-1}\t_n(y)\ell_n^y(\theta_n-v)
\1_{\{c_n^{-1}\t_n(y)\ell_n^y(\theta_n-v)\leq \e\}}
\Eq(6.lem2.3)
\ee
and
\be
\E^{\circ}\SS^{(7)}_{n,\e}(t)=
\textstyle
k_n(t) \sum_{y\in T^{\circ}_n}\E^{\circ}X_n(y).
\Eq(6.lem2.4)
\ee
Proceeding as in \eqv(4.prop2.12)-\eqv(4.prop2.14) to deal with the conditional expectation and using  that
$\P(\t_n(x)\geq r_n(\varepsilon_n))=\pi_n(T^{\circ}_n)(1+o(1))$ (see the paragraph below \eqv(4.prop2.15)),
we get
\be
\E^{\circ}\SS^{(7)}_{n,\e}(t)=
\frac{k_n(t)(1+o(1))}{\pi_n(T^{\circ}_n)}
\sum_{y\in T^{\circ}_n}
\int_0^{\theta_n}dv h_{n, y}(v)E_{y}\ell_n^y(\theta_n-v)\E^{y}c_n^{-1}\t_n(y)\1_{\{c_n^{-1}\t_n(y)\leq \e_n\}}
\nonumber
\ee
where $\P^{y}$  denotes the law of $\t_n(y)$ and where $\e_n\equiv\e_n(y)=\e/\ell_n^y(\theta_n-v)$. 
Using \eqv(6.lem1.1) if
$
\ell_n^y(\theta_n-v)>\e e^{-n\b\b_c(\varepsilon)(\a^{-1}_c(\varepsilon)-1)}
$
and using that if 
$
\ell_n^y(\theta_n-v)\leq\e e^{-n\b\b_c(\varepsilon)(\a^{-1}_c(\varepsilon)-1)}
$
then
\be
E_{y}\ell_n^y(\theta_n-v)\E^{y}c_n^{-1}\t_n(y)\1_{\{c_n^{-1}\t_n(y)\leq \e_n\}}
\leq
\e e^{-n\b\b_c(\varepsilon)(\a^{-1}_c(\varepsilon)-1)}c_n^{-1}e^{n\b^2/2},
\ee
we readily see that
\be
\begin{split}
\E^{\circ}\SS^{(7)}_{n,\e}(t) & \leq 
C_1t\frac{
\e^{1-\a_{n}(\varepsilon)-\frac{\log \e}{2n\b^2}}
}{
b^{\circ}_n\theta_n\pi_n(T^{\circ}_n)
}
\sum_{y\in T^{\circ}_n}\int_0^{\theta_n}dv h_{n, y}(v)E_{y}\wt F_{\b,\varepsilon,\e, n}(\ell_n^y(\theta_n-v))
\\
& +
C_2 \e n^{\a_{n}(\varepsilon)/2}e^{-n\b^2/2}
k_n(t)(\pi_n(T^{\circ}_n))^{-1}P_{\pi_n}(H(T^{\circ}_n)\leq \theta_n)
\end{split}
\Eq(6.lem2.6)
\ee
where here and below $C_i>0$, $i=1,2,\dots$ are constants, and for  $F_{\b,\varepsilon,n}$ as  in \eqv(4.prop2.16),
\be
\wt F_{\b,\varepsilon,\e, n}(z)
=
F_{\b,\varepsilon,n}(z)
\frac{
z^{\frac{\log \e}{n\b^2}}\left(1-\sfrac{\log z}{n\b\b_c(\varepsilon)}\right)
}{
\a^{-1}_c(\varepsilon)-1-\frac{\log \e}{n\b\b_c(\varepsilon)}+\frac{\log z}{n\b\b_c(\varepsilon)}
}\1_{\left\{z>\e e^{-n\b\b_c(\varepsilon)(\a^{-1}_c(\varepsilon)-1)}\right\}}.
\Eq(6.lem2.7)
\ee
By the leftmost inequality of \eqv(4.prop2.14) and \eqv(4.cnew),
$
\wt F_{\b,\varepsilon,\e, n}(z)\leq C_3F_{\b,\varepsilon,n}(z)
$.
Thus, by \eqv(4.prop2.18), the first summand in \eqv(6.lem2.6) is bounded above by
\be
C_4t
\e^{1-\a_{n}(\varepsilon)-\frac{\log \e}{2n\b^2}}.
\Eq(6.lem2.8)
\ee
Using \eqv(3.cor1.1')  and proceeding as in \eqv(6.prop1.01) to bound $k_n(t)$,
the second  summand is bounded above by
\be
C_5te^{-n(\b^2-\b^2_c(\varepsilon))/2} \kappa_n n^{\a_{n}(\varepsilon)/2+1}
(r^{\star}_n)^{1+\a_{n}(\varepsilon)+o(1)}\rightarrow 0
\Eq(6.lem2.9)
\ee
as $n\rightarrow \infty$  by virtue of \eqv(3.prop2.1), \eqv(2.lem2.2), and the assumption that $\b>\b_c(\varepsilon)$
where $0<\varepsilon<1$.
Note in particular that $\lim_{n\rightarrow\infty}\a_{n}(\varepsilon)=\a(\varepsilon)<1$. Hence, inserting \eqv(6.lem2.8) and \eqv(6.lem2.9) in \eqv(6.lem2.6) and passing to the limit 
\be
\lim_{\e\rightarrow 0}\limsup_{n\rightarrow \infty}
\E^{\circ}\SS^{(7)}_{n,\e}(t) =0,\quad \forall t>0.
\Eq(6.lem2.10)
\ee
This proves \eqv(6.lem2.2). Turning to the variance we have, as in \eqv(4.prop2.6'), by independence, that
\be
\Va^{\circ}(\SS^{(7)}_{n,\e}(t)) \equiv \E^{\circ}(\SS^{(7)}_{n,\e}(t)-\E^{\circ}\SS^{(7)}_{n,\e}(t))^2
\leq
k^2_n(t)\sum_{y\in T^{\circ}_n}\E^{\circ}(X_n(y))^2.
\Eq(6.lem2.11)
\ee
Proceeding as in the proof of \eqv(6.lem2.6) but using \eqv(6.lem1.2) and the line below \eqv(6.lem2.7), we get that
\bea
\Va^{\circ}(\SS^{(7)}_{n,\e}(t)) 
\hspace{-6pt}&\leq&\hspace{-6pt} 
C_6t^2\frac{
\e^{2-\a_{n}(\varepsilon)-\frac{\log \e}{2n\b^2}}
}{
(b^{\circ}_n\theta_n)^2\pi_n(T^{\circ}_n)
}
\sum_{y\in T^{\circ}_n}\left(\int_0^{\theta_n}dv h_{n, y}(v)E_{y} F_{\b,\varepsilon,\e, n}(\ell_n^y(\theta_n-v))\right)^2
\nonumber
\\
\hspace{-6pt}&+&\hspace{-6pt} 
C_7 \e n^{\a_{n}(\varepsilon)/2}e^{-n\b\b_c(\varepsilon)}\frac{k^2_n(t)\theta}{\pi_n(T^{\circ}_n)}
\sum_{y\in T^{\circ}_n}\left(\int_0^{\theta_n}dv h_{n, y}(v)\right)^2.
\nonumber
\eea
From the bound
$
\int_0^{\theta_n}dv h_{n, y}(v)E_{y} F_{\b,\varepsilon,\e, n}(\ell_n^y(\theta_n-v))
\leq 
(1+o(1))\int_0^{\theta_n}dv h_{n, y}(v)\theta_n^{\a_{n}(\varepsilon)}
\leq 
(1+o(1))\theta_n^{\a_{n}(\varepsilon)}P_{\pi_n}(H(y)\leq \theta_n)
$
and \eqv(3.lem1.2), \eqv(4.lem3.6'), we get that on $\O^{\star}$, for all but a finite number of indices $n$,
the first summand is bounded above by
\be
C_8t^2\e^{2-\a_{n}(\varepsilon)-\frac{\log \e}{2n\b^2}}
\left(n\kappa_n\theta_n^{\a_{n}(\varepsilon)}(r^{\star}_n)^{1+\a_{n}(\varepsilon)+o(1)}\right)^22^{-n}.
\Eq(6.lem2.13)
\ee
Using the bound
$
\sum_{y\in T^{\circ}_n}\bigl(\int_0^{\theta_n}dv h_{n, y}(v)\bigr)^2
\leq \sup_{y\in T^{\circ}_n}P_{\pi_n}(H(y)\leq \theta_n)P_{\pi_n}(H(T^{\circ}_n)\leq \theta_n)
$,
and proceeding as in \eqv(6.lem2.9), the second summand is bounded above by 
\be
C_9 t^2 \e n^{\a_{n}(\varepsilon)/2}
\left(n^{2}\kappa_n(r^{\star}_n)^{1+\a_{n}(\varepsilon)+o(1)}\right)^2
\theta_ne^{-n\b_c(\varepsilon)(\b-\b_c(\varepsilon))}2^{-n}.
\Eq(6.lem2.14)
\ee
Since by assumption $\b>\b_c(\varepsilon)$ and $0<\varepsilon<1$,
\eqv(4.cnew) (see also \eqv(4.c1'new)) enables us to conclude that  on $\O^{\star}$, for all large enough $n$,
\be
\Va^{\circ}(\SS^{(7)}_{n,\e}(t)) \leq C_{10}t^2 \e 2^{-n(1-\varepsilon)}.
\Eq(6.lem2.15)
\ee
This yields \eqv(6.lem2.1) and concludes the proof of the Lemma.
\end{proof}

Arguing as in the proof of Proposition \thv(4.prop1) that 
$
b_n=b^{\circ}_n(1+o(1))
$
on $\O_1\cap\O^{\circ}\cap \O^{\star}$ for all large enough $n$,
it follows from Lemma \thv(6.lem2)  and Borel-Cantelli Lemma that
\be
\lim_{\e\rightarrow 0}\lim_{n\rightarrow\infty}\left(\SS^{(5)}_{n,\e}(t)-\SS^{(6)}_{n,\e}(t)\right)=0,\quad \P-\text{almost surely.}
\Eq(6.prop1.19)
\ee
Collecting \eqv(6.prop1.8), \eqv(6.prop1.14), \eqv(6.prop1.17) and  \eqv(6.prop1.19) yields \eqv(6.prop1.1bis). The proof of Proposition \thv(6.prop1) is complete.\end{proof}



 
 \section{Proof of Theorem  \thv(1.theo0)  and Theorem \thv(1.theo1)} 
 \label{S7}

\begin{proof}[Proof of Theorem \thv(1.theo1)]
By Proposition \eqv(3.prop2), Proposition \eqv(4.prop1), Proposition \eqv(5.prop1) and Proposition \eqv(6.prop1), 
under the assumptions of Proposition \eqv(4.prop1) and Proposition \eqv(6.prop1), Conditions (B0), (B1), (B2), and (B3) of Theorem \thv(1.theoB) are satisfied $\P$-a.s.. It remains to check Condition (A0), i.e.~to prove that $\P$-a.s.,
for all $u>0$,
\be
\lim_{n\rightarrow\infty}P_{\mu_n}(Z_{n,1}>u)=0
\Eq(A0.1bis)
\ee
where   
$
Z_{n,1}=\int_{0}^{\theta_n}\max\left((c_nr^{\star}_n)^{-1},c_n^{-1}\t_n(Y_n(s))\right)ds
$
and $\mu_n$ is  the uniform measure on $\VV_n$.
By \eqv(3.lem0.2) and \eqv(3.lem0.3)
\bea
P_{\mu_n}(Z_{n,1}>u)
&\leq& 
\textstyle
(1+o(1))P_{\pi_n}(Z_{n,1}>u)+ \sum_{x\in\overline V^{\star}_n}\mu_n(x)P_{x}(Z_{n,1}>u)
\\
&\leq&
(1+o(1))P_{\pi_n}(Z_{n,1}>u)+  n^{-c_{\star}}(1+o(1))
\eea
where the last line is \eqv(2.lem3.1). Thus \eqv(A0.1bis) is an immediate consequence of Proposition \eqv(4.prop1). 
One readily checks that the assumptions on $a_n$, $c_n$, and $\theta_n$ of the theorem imply that the conditions
\eqv(4.c4) and \eqv(4.cnew) of Proposition \eqv(4.prop1) are verified. 
The proof of  \thv(1.theo1) is complete.
\end{proof}

\begin{proof}[Proof of Theorem \thv(1.theo0)] 
Reasoning as in the proof of Theorem \thv(1.theo1), we may assume that the process starts in its invariant measure $\pi_n$.
The main idea behind the proof is now classical.
Suppose that
\be
\PP_{\pi_n}\left(A_n(t,s)\right)
=
P_{\pi_n}\left(\{\RR_n\cap (t,t+s)=\emptyset\}\right)+o(1)
\Eq(1.theo0.4)
\ee
where 
$
A_n(t,s)\equiv \{X(c_nt)=X(c_n(t+s))\}
$
and where $\RR_n$ denotes the range of the rescaled blocked clock process $S^b_n(t)$.
Then Theorem \thv(1.theo0) is a direct consequence of
Theorem \thv(1.theo1) and the arcsine law for stable subordinators.
We refer to Ref.~\cite{G12} for  a detailed proof of this statement (see the proof of Theorem 1.6 therein)
and focus on  establishing \eqv(1.theo0.4).
For $k\geq 1$ and $Z_{n,i}$ as in \eqv(1.03.10) set
\be
\textstyle
\BB_k=\left\{\sum_{i=1}^{k}Z_{n,i}<t, \sum_{i=1}^{k+1}Z_{n,i}>t+s\right\}.
\Eq(1.theo0.5)
\ee
Then by \eqv(1.03.9),
$
\{\RR_n\cap (t,t+s)\neq\emptyset\}=\{\cup_{k\geq 1}\BB_k\}
$.
Furthermore, for any $T>0$,
\be
P_{\pi_n}\left(\cup_{k> k_n(T)}\BB_k\right)
\leq 
P_{\pi_n}\left(S^b_n(T)<t\right)
\underset{n\rightarrow\infty}{\longrightarrow }
P\left(V_{\a(\varepsilon)}(T)<t\right)
\leq \d
\Eq(1.theo0.6)
\ee
where convergence is almost sure in the random environment as follows from Theorem \thv(1.theo1),
and where  $\d$ can be made as small as desired by taking $T$ large enough. Therefore
\be
0\leq P_{\pi_n}\left(\{\RR_n\cap (t,t+s)=\emptyset\}\right)
-
P_{\pi_n}\left(\cup_{1\leq k\leq k_n(T)}\BB_k\right)
\leq \d.
\Eq(1.theo0.7)
\ee

Note that the event $\BB_k$ is non empty if and only if the increment $Z_{n,k+1}$ straddles over the interval $(t,t+s)$.
To show that \eqv(1.theo0.4) holds it now suffices to establish the following two facts:

\noindent{\textbf {Fact 1.}} 
Almost surely in the random environment, with overwhelming probability,
non-empty events $\BB_k$, $k\leq k_n(T)$, are produced by visits 
of the process $Y_n$ to the set $T^{\circ}_n$ and, more precisely,
by (many) visits of the process to one and the same element of $T^{\circ}_n$,
no other element of $T^{\circ}_n$ being visited in the time interval $(t,t+s)$. 
This implies that $\P$-a.s.
\be
P_{\pi_n}\left(A_n (t,s)\cap\{\cup_{1\leq k\leq k_n(T)}\BB_k\}\right)
\geq P_{\pi_n}\left(\cup_{1\leq k\leq k_n(T)}\BB_k\right)+o(1)
\Eq(1.theo0.8)
\ee

\noindent{\textbf {Fact 2.}}  If $\BB_k$ and $\BB_k'$, $1\leq k\neq k'\leq k_n(T)$, are two non-empty events then, almost surely in the random environment they are produced by visits to two distinct elements of $T^{\circ}_n$ with overwhelming probability. This implies that $\P$-a.s.
\be
P_{\pi_n}\left(A_n (t,s)\cap(\cap_{1\leq k\leq k_n(T)}\BB^c_k)\right)\rightarrow 0,\quad n\rightarrow\infty
\Eq(1.theo0.9)
\ee

Combining \eqv(1.theo0.7), \eqv(1.theo0.8), and \eqv(1.theo0.9) then establishes that
\be
\left|P_{\pi_n}\left(A_n(t,s)\right)
-
P_{\pi_n}\left(\{\RR_n\cap (t,t+s)=\emptyset\}\right)\right|\leq \d+o(1)
\Eq(1.theo0.10)
\ee
which is tantamount to \eqv(1.theo0.4).

The proofs of Facts 1 and 2 mostly use information already obtained in the course of the verification of Conditions (B1)-(B3). We present them succinctly below, beginning with the proof of Fact 1. Fix $0<T<\infty$  and assume that the assumption of Proposition \eqv(4.prop1) are satisfied. 
Let
$
H_k(A)=\inf\{t\geq \theta_nk \mid Y_n(t)\in A\}
$
be the first hitting time of $A\subseteq\VV_n$ after time $\theta_nk$.
Note first that $\BB_k=\BB_k\cap\{Z_{n,k+1}>s\}$ and thus, by  \eqv(4.0), 
\be
P_{\pi_n}\left(\cup_{1\leq k\leq k_n(T)}(\BB_k\cap \{H_k(T_n)> \theta_n\})\right)=0
\ee
for all large enough $n$.
Note next that reasoning as in \eqv(6.prop1.12)-\eqv(6.prop1.14),  on $\O^{\circ}\cap \O^{\star}$,
\be
P_{\pi_n}\left(\cup_{1\leq k\leq k_n(T)}(\BB_k\cap \{H_k(T_n\setminus T^{\circ}_n)\leq \theta_n\})\right)
\leq k_n(T)P_{\pi_n}\left(H_k(T_n\setminus T^{\circ}_n)\leq \theta_n\right)\rightarrow 0
\nonumber
\ee
as $n\rightarrow\infty$ by virtue of \eqv(4.cnew).
Hence on $\O^{\circ}\cap \O^{\star}$, for all large enough $n$,
\be
\begin{split}
 & P_{\pi_n}\left(\cup_{1\leq k\leq k_n(T)}\BB_k\right)
\\
= & 
\textstyle
P_{\pi_n}\left(\cup_{1\leq k\leq k_n(T)}
(\BB_k\cap\left\{H_k(T^{\circ}_n)\leq \theta_n\right\}\cap\left\{H_k(T_n\setminus T^{\circ}_n)>\theta_n)\right\}
\right)+o(1).
\end{split}
\Eq(1.theo0.11)
\ee
This means that for $\BB_k$ to be non-empty, the increment $Z_{n,k+1}$ must be produced by visits of $Y_n$ 
to $T^{\circ}_n$, and $T^{\circ}_n$ only. To prove that all these visits, if there are several of them,
must be to a single vertex it suffices to show that  as $n\rightarrow\infty$,
\be
p_n\equiv P_{\pi_n}\left(\cup_{1\leq k\leq k_n(T)}
(\BB_k\cap\left\{H_k(T^{\circ}_n)\leq \theta_n\right\}\cap\CC_n(Y_n(H_k(T^{\circ}_n)))
\right)\rightarrow 0,
\Eq(1.theo0.12)
\ee
where
\be
\CC_n(Y_n(H_k(T^{\circ}_n)))\equiv
\bigl\{
\inf\{t> H_k(T^{\circ}_n) \mid Y_n(t)\in T^{\circ}_n\setminus Y_n(H_k(T^{\circ}_n))\}\leq \theta_n
\bigr\}.
\Eq(New1)
\ee
Now, 
\be
\begin{split}
p_n 
 & 
 =
 P_{\pi_n}\left(\cup_{1\leq k\leq k_n(T)}\cup_{x\in T^{\circ}_n}
(\BB_k\cap\left\{H_k(T^{\circ}_n)\leq \theta_n, Y_n(H_k(T^{\circ}_n))=x\right\}\cap\CC_n(x)
\right)
\\
& \leq
\wt\nu_n^T
\end{split}
\Eq(1.theo0.13)
\ee
where $\wt\nu_n^T$ is defined in \eqv(4.lem2.1) and bounded in Lemma \thv(4.lem2). 
Reasoning as in the paragraph below  \eqv(4.prop1.3) then yields that  under the assumptions  \eqv(4.c4) and \eqv(4.cnew),
on $\O^{\circ}\cap \O^{\star}$, $\lim_{n\rightarrow\infty}\wt\nu_n^T= 0$. 
Fact 1 is now proved.

Fact 2 will be established if we can prove that as $n\rightarrow\infty$,
\be
\bar p_n\equiv P_{\pi_n}\left(\cup_{1\leq k\leq k_n(T)}
\left(\left\{H_k(T^{\circ}_n)\leq \theta_n\right\}\cap \DD_{n,k}(Y_n(H_k(T^{\circ}_n)))\right)
\right)
\rightarrow 0,
\Eq(1.theo0.14)
\ee
where
\be
\DD_n(Y_n(H_k(T^{\circ}_n)))\equiv\bigl\{
\inf\{t> (k+1)\theta_n\mid Y_n(t)=Y_n(H_k(T^{\circ}_n))\}\leq\theta_n k_n(T)
\bigr\}.
\Eq(New2)
\ee
To prove this observe that the event in \eqv(1.theo0.14) can be written as
\be
\cup_{x\in T^{\circ}_n}\cup_{y\in T^{\circ}_n}
\left(\left\{H_k(T^{\circ}_n)\leq \theta_n, Y_n(H_k(T^{\circ}_n))=x\right\}\cap\{Y_n(\theta_n (k+1))=y\}\cap \DD_{n,k}(x)\right)
\nonumber
\ee
Thus, by the Markov property we have, using the notation of \eqv(4.7)-\eqv(4.9) and the bound
$
P_{y}\left(H(x)\leq \theta_n(k_n(T)-(k+1))\right)\leq P_{y}\left(H(x)\leq \theta_nk_n(T)\right)
$,
\be
\bar p_n
\leq
\sum_{1\leq k\leq k_n(T)}\sum_{x\in T^{\circ}_n}\sum_{y\in T^{\circ}_n}
\int_0^{\theta_n}dv h_{n, x}(v)
P_{x}\left(Y_n(\theta_n-v)=y\right)P_{y}\left(H(x)\leq \theta_nk_n(T)\right).
\nonumber
\ee
To proceed, we split the domain of integration into $[0, \theta_n-\kappa_n)\cup[\theta_n-\kappa_n, \theta_n]$. 
Using that by Proposition \thv(3.prop2), on $\O_1$, for all $n$ large enough,
$P_{x}\left(Y_n(\theta_n-v)=y\right)=\pi_n(y)(1+o(1))$ for all $v\in[0, \theta_n-\kappa_n)$, 
the contribution coming from this domain is at most
\bea
&&
(1+o(1))
\sum_{1\leq k\leq k_n(T)}\sum_{x\in T^{\circ}_n}
\int_0^{\theta_n}dv h_{n, x}(v)
\sum_{y\in T^{\circ}_n}\pi_n(y)P_{y}\left(H(x)\leq \theta_nk_n(T)\right)\quad
\Eq(1.theo0.16)
\\
\hspace{-6pt}&\leq &\hspace{-6pt} 
(1+o(1))k_n(T)
P_{\pi_n}(H(T^{\circ}_n)\leq \theta_n) \sup_{y\in T^{\circ}_n}P_{\pi_n}(H(y)\leq \theta_nk_n(T))
\Eq(1.theo0.17)
\\
\hspace{-6pt}&\leq &\hspace{-6pt} 
(1+o(1))\left(\theta_nk_n(T)r^{\star}_nn2^{-n}\right)^22^{n}\pi_n(T^{\circ}_n)
\Eq(1.theo0.18)
\eea
where we used \eqv(3.cor1.1') with $t_n=\theta_n$ (which is licit as we many times saw) and \eqv(3.lem1.2)  with $t_n=\theta_nk_n(T)$,
which is licit provided that $\theta_nk_n(T)r^{\star}_nn2^{-n}\rightarrow 0$ as $n\rightarrow\infty$, 
and this is guaranteed by our assumptions on $a_n$.
Indeed,
proceeding as in the proof of Proposition \thv(4.prop1) (see \eqv(4.prop1.0) and the paragraph above) 
we get
that on  $\O^{\circ}\cap \O^{\star}\cap\O_1$, for large enough $n$,
\be
\theta_nk_n(T)r^{\star}_nn2^{-n}
\leq
\kappa_n(r^{\star}_n)^{1+\a_{n}(\varepsilon)+o(1)}n2^{-(1-\varepsilon)n}
\rightarrow 0
\Eq(1.theo0.19)
\ee
as $n\rightarrow\infty$ for all $0<\varepsilon<1$.
Since furthermore
$
2^{n}\pi_n(T^{\circ}_n)=(1+o(1))2^{(1-\varepsilon)n}\left(n^2\theta_n\right)^{\a(\varepsilon)}
$
by \eqv(4.c2), \eqv(3.lem0.2), and \eqv(2.lem3.5),  and 
we get that on $\O^{\circ}\cap \O^{\star}\cap\O_1$, \eqv(1.theo0.18) is bounded above by
\be
(1+o(1))\left(\kappa_n(r^{\star}_n)^{1+\a_{n}(\varepsilon)+o(1)}n\right)^2
\left(n^2\theta_n\right)^{\a(\varepsilon)}2^{-(1-\varepsilon)n},
\ee
and by \eqv(4.cnew) this decays to zero as $n\rightarrow\infty$ for all $0<\varepsilon<1$.

Consider next the domain $[\theta_n-\kappa_n, \theta_n]$ and note that since
\be
\sum_{y\in T^{\circ}_n}P_{x}\left(Y_n(\theta_n-v)=y\right)P_{y}\left(H(x)\leq \theta_nk_n(T)\right)
\leq 1
\Eq(1.theo0.20)
\ee
the corresponding contribution is bounded above by
$
k_n(T)P_{\pi_n}\left(\theta_n-\kappa_n\leq H(T^{\circ}_n)\leq  \theta_n\right)
$.
By the upper bound of \eqv(3.lem3) and the lower bound of \eqv(3.lem2), on $\O^{\star}$, for all but a finite number of indices $n$, this is in turn bounded above by
\be
n^{1+2\a_{n}(\varepsilon)}\theta_n^{-(1-\a(\varepsilon))}\kappa_n^2(r^{\star}_n)^{1+\a_{n}(\varepsilon)+o(1)}
\rightarrow 0
\Eq(1.theo0.21)
\ee
as $n\rightarrow\infty$, where we again used that  $2^{n\d_n}=\left(n^2\theta_n\right)^{\a(\varepsilon)}$  by \eqv(4.c2) whereas $0<\a(\varepsilon)<1$ by assumption; the final convergence
then follows from \eqv(4.c4). 
Combining the conclusions of \eqv(1.theo0.19) and \eqv(1.theo0.21) we get that on $\O^{\circ}\cap \O^{\star}\cap\O_1$,
\be
\lim_{n\rightarrow\infty}\bar p_n=0.
\ee
This concludes the proof of Fact 2. The proof of Theorem \thv(1.theo0) is now complete.
\end{proof}
 


 \section{Appendix:  Proof of Theorem  \thv(1.theoA) and Theorem \thv(1.theoB)} 
 \label{S8}

\begin{proof}[Proof of Theorem \thv(1.theoA)]
The proof closely follows that of Theorem 1.2 of Ref.~\cite{BG13}. 
Throughout we fix a realization $\o\in\O$ of the random environment but do not make this explicit in the notation. 
We set
\be
\wh S^b_n(t)\equiv S_n^b(t)-Z_{n,1}.
\Eq(1.theoA.2)
\ee
Condition (A0) ensures that $S^b_n -\wh S_n^b$ converges to zero, uniformly.
Thus we must show that  under Conditions (A1), (A2), and (A3),
\be
\wh S^b_n\Rightarrow_{J_1}  S_\nu.
\Eq(1.theoA.3)
\ee
For this we rely on Theorem 1.1 of Ref.~\cite{BG13}.
 (This result is itself a specialized form of Theorem 4.1 of Ref.~\cite{DR78} suited to the present setting.)
Namely, we want to show that Conditions (A1), (A2), and (A3) imply the conditions of Theorem 1.1 of Ref.~\cite{BG13}.

To this end let  $\{\FF_{n,i}, n\geq 1, i\geq 0\}$ be the array of sub-sigma fields of $\FF^Y$ defined  (with obvious notation) through
$
\FF_{n,i}=\s\left(Y_n(s), s\leq\theta_ni\right)
$,
for $i\geq 0$.
Note that for each $n$ and $i\geq 1$, $Z_{n,i}$ is $\FF_{n,i}$ measurable and $\FF_{n,i-1}\subset\FF_{n,i}$.
Next observe that by the Markov property
and the fact that, for all $i\geq 1$ and $y\in\VV_n$,
$
\PP_{y}(Z_{n,i}>u)=\PP_{y}(Z_{n,1}>u)
$,
\be
\PP_{\mu_n}\left(Z_{n,i}>u\,\big|\,\FF_{n,i-1}\right)=\sum_{y\in\VV_n}\1_{\{Y_n((i-1)\theta)=y\}}\PP_{y}(Z_{n,1}>u).
\Eq(1.theoA.4)
\ee
In view of this, \eqv(1.3.5), \eqv(1.3.6), and \eqv(1.3.7)
\be
\sum_{i=2}^{k_n(t)}\PP_{\mu_n}\left(Z_{n,i}>u \mid \FF_{n,i-1}\right)=\nu_n^{Y,t}(u,\infty),
\Eq(1.theoA.5)
\ee
and in view of \eqv(1.3.8)
\be
\sum_{i=2}^{k_n(t)}
\left[\PP_{\mu_n}\left(Z_{n,i}>u \mid \FF_{n,i-1}\right)\right]^2
=\s_n^{Y,t}(u,\infty).
\Eq(1.theoA.6)
\ee
 From \eqv(1.theoA.5) and \eqv(1.theoA.6) it follows that  Conditions (A1) and (A2) of
Theorem \thv(1.theoA) are exactly the conditions of Theorem 1.1 of Ref.~\cite{BG13}.
Similarly Condition (A3)  is condition (1.9).
Therefore the conditions of Theorem 1.1 of Ref.~\cite{BG13} are verified, and so
$
\wh S^b_n\Rightarrow_{J_1} S_\nu
$
in $D([0,\infty))$ where $S_\nu$ is a subordinator with L\'evy measure $\nu$ and zero drift.
\end{proof}

The proof of Theorem \thv(1.theoB) centers of the 

\begin{proposition}
    \TH(1.prop1)
     Assume that Condition (B1) is satisfied. Then, choosing $\theta_n\geq \kappa_n$,
the following holds for all initial distributions $\mu_n$: for all $t>0$, all $u>0$, and all $\e>0$,
\be
P_{\mu_n}\left(\left|\nu_n^{Y,t}(u,\infty)-\nu_n^{t}(u,\infty)\right|\geq\e\right)
\leq
5\e^{-2}\left[\rho_n\left(\nu_n^{t}(u,\infty)\right)^2+\s_n^{t}(u,\infty)\right]\,,
\Eq(1.prop1.1)
\ee
and
\be
P_{\mu_n}\left(\s_n^{Y,t}(u,\infty)\geq\e\right)
\leq
{\e}^{-1}(1+\rho_n)\s_n^{t}(u,\infty)\,.
\Eq(1.prop1.2)
\ee
\end{proposition}

\begin{proof}[Proof of Proposition \thv(1.prop1)]  We assume throughout that $\theta_n\geq \kappa_n$. 
To prove \eqv(1.prop1.2), simply 
note that by a first order Tchebychev inequality
\bea
P_{\mu_n}\left(\s_n^{Y,t}(u,\infty)\geq\e\right)
&\leq&
\textstyle
{\e}^{-1}k_n(t) \sum_{y\in\VV_n}E_{\mu_n}(\pi_n^{Y,t}(y))\left[Q^{u}_n(y)\right]^2
\Eq(1.prop1.3)
\\
&\leq&
{\e}^{-1}(1+\rho_n)\s_n^{t}(u,\infty),
\Eq(1.prop1.4)
\eea
where we used in the last line that by  \eqv(B0.1),
\be
|E_{\mu_n}(\pi_n^{Y,t}(y))-\pi_n(y)|\leq \rho_n\pi_n(y).
\Eq(1.prop1.4')
\ee
Turning to  \eqv(1.prop1.1), a second order Chebychev inequality yields
\bea
\nonumber
&&P_{\mu_n}\left(\left|\nu_n^{Y,t}(u,\infty)-\nu_n^{t}(u,\infty)\right|\geq\e\right)
\\
\nonumber
&&
\textstyle
\leq\e^{-2}E_{\mu_n}\Bigl[k_n(t)\sum_{y\in\VV_n}\left(\pi_n^{Y,t}(y)-\pi_n(y)\right)Q^{u}_n(y)\Bigr]^2
\\
&&
\textstyle
=\e^{-2}\sum_{x\in\VV_n}\sum_{y\in\VV_n}Q^{u}_n(x)Q^{u}_n(y)
\sum_{i=1}^{k_n(t)-1}\sum_{j=1}^{k_n(t)-1}\Delta_{ij}(x,y)
\Eq(1.prop1.5)
\eea
where
\be
\begin{split}
\Delta_{ij}(x,y) & \equiv P_{\mu_n}\left(Y_n(i\theta_n)=x, Y_n(j\theta_n)=y\right)+\pi_n(x)\pi_n(y)
\\
 & -\pi_n(y)P_{\mu_n}\left(Y_n(i\theta_n)=x\right)-\pi_n(x)P_{\mu_n}\left(Y_n(j\theta_n)=y\right).
\end{split}
\Eq(1.prop1.6)
\ee
Using again \eqv(B1.1) yields 
\be
|\Delta_{ij}(x,y)|\leq
\begin{cases}
\rho_n(4+\rho_n)\pi_n(x)\pi_n(y),
&\hbox{\rm if}\,\, i\neq j,\\
(1+\rho_n)\pi_n(x)+(1+2\rho_n)\pi^2_n(x),
&\hbox{\rm if}\,\, i= j\, \text{\rm and}\,\, x=y,\\
0&\hbox{\rm else}.
\end{cases}
\Eq(1.prop1.7)
\ee
Thus \eqv(1.prop1.5) is bounded above by
\be
\e^{-2}\rho_n(4+\rho_n)\Bigl[k_n(t)\sum_{y\in\VV_n}\pi_n(y)Q^{u}_n(y)\Bigr]^2
+\e^{-2}(2+3\rho_n)k_n(t)\sum_{y\in\VV_n}\pi_n(y)\left[Q^{u}_n(y)\right]^2
\Eq(1.prop1.8)
\ee
Since by assumption $\rho_n\downarrow 0$ as $n\uparrow\infty$,
\eqv(1.prop1.8) is tantamount to the right-hand side of \eqv(1.prop1.1).
Proposition \thv(1.prop1) is proven.
\end{proof}

\begin{proof}[Proof of Theorem \thv(1.theoB)]
The proof of Theorem \thv(1.theoB) is now immediate: 
Condition (B2) combined with the conclusions of Proposition \thv(1.prop1)  implies both conditions (A1) and (A2), and Condition (B3) combined with \eqv(1.prop1.4') implies Condition (A3).
\end{proof}




\def\cprime{$'$}

\end{document}